\documentclass{amsproc}
\usepackage[arrow, curve, frame, matrix, graph, tips]{xy}
\usepackage{amssymb}
\usepackage{amsthm}
\usepackage{amsmath}
\usepackage[colorlinks=true,urlcolor=blue,linkcolor=blue]{hyperref}
\usepackage{enumitem}

\newtheorem{theorem}{Theorem}[subsection]
\newtheorem{lemma}[theorem]{Lemma}
\newtheorem{proposition}[theorem]{Proposition}
\newtheorem{corollary}[theorem]{Corollary}

\theoremstyle{definition}
\newtheorem{definition}[theorem]{Definition}
\newtheorem{example}[theorem]{Example}
\newtheorem{examples}[theorem]{Examples}
\newtheorem{non-example}[theorem]{Non-Example}

\theoremstyle{remark}
\newtheorem{remark}[theorem]{Remark}

\newcommand{\tn}[1]{\textnormal{#1}}
\newcommand{\tnb}[1]{\textnormal{\bf #1}}
\newcommand{\tensor}{\otimes}
\newcommand{\PSh}[1]{\widehat{#1}}
\newcommand{\C}{\mathbb{C}}
\newcommand{\D}{\mathbb{D}}

\newcommand{\N}{\mathbb{N}}
\newcommand{\comp}{\circ}
\newcommand{\id}{\tn{id}}
\newcommand{\ca}[1]{\mathcal{#1}}
\newcommand{\ladj}{\dashv}
\newcommand{\iso}{\cong}
\newcommand{\catequiv}{\simeq}
\newcommand{\Set}{\tnb{Set}}
\newcommand{\Top}{\tnb{Top}}
\newcommand{\Cat}{\tnb{Cat}}
\newcommand{\CAT}{\tnb{CAT}}
\newcommand{\op}{\tn{op}}

\renewcommand{\implies}{\Rightarrow}

\renewcommand{\iff}{\Leftrightarrow}
\newcommand{\from}{\leftarrow}

\DeclareMathOperator*{\Tbar}{\overline{T}}

\DeclareMathOperator*{\opE}{E}

\DeclareMathOperator*{\opF}{F}

\newcommand{\Graph}{\tnb{Gph}}
\newcommand{\Enrich}[1]{#1{\textnormal{-Cat}}}

\newcommand{\Trimble}[1]{\tn{Trm}_{#1}}
\newcommand{\col}{\tn{col}}
\newcommand{\res}{\tn{res}}
\newcommand{\G}{\mathbb{G}}
\DeclareMathOperator*{\colim}{\textnormal{colim}}
\newcommand{\Fam}{\tn{Fam}}
\newcommand{\LaxAlg}[1]{{\textnormal{Lax-}}#1{\textnormal{-Alg}}}
\newcommand{\OpLaxAlg}[1]{{\textnormal{OpLax-}}#1{\textnormal{-Alg}}}

\newcommand{\End}{\textnormal{End}}

\newcommand{\Mult}[1]{#1{\textnormal{-Mult}}}

\newcommand{\NOp}[1]{#1{\textnormal{-Op}_0}}
\newcommand{\MND}{\textnormal{MND}}
\newcommand{\OpMND}{\textnormal{OpMND}}

\newcommand{\DISTMULT}{\textnormal{DISTMULT}}
\newcommand{\OpDISTMULT}{\textnormal{OpDISTMULT}}

\newcommand{\GMND}{\ca G{\tn{-MND}}}
\newcommand{\GOpMND}{\ca G{\tn{-OpMND}}}
\newcommand{\CatAr}{\tn{CAT-Ar}}

\newcommand{\PbSq}[8]{\xymatrix{{#1} \ar[d]_{#5}
\save \POS?(.3)="lpb" \restore
\ar[r]^-{#6} \save \POS?(.3)="tpb" \restore
& {#2} \ar[d]^{#7} \save \POS?(.3)="rpb" \restore \\
{#4} \ar[r]_-{#8} \save \POS?(.3)="bpb" \restore & {#3}
\POS "rpb"; "lpb" **@{}; ?!{"bpb";"tpb"}="cpb" **@{}; ? **@{-};
"tpb"; "cpb" **@{}; ? **@{-}}}

\newcommand{\LaxSq}[9]{\xymatrix{{#1} \ar[r]^-{#6}
\ar[d]_{#5} \save \POS?="dom" \restore
& {#2} \ar[d]^{#7} \save \POS?="cod" \restore \\
{#3} \ar[r]_-{#8} & {#4}
\POS "dom"; "cod" **@{} ?(.35) \ar@{=>}^{#9} ?(.65)}}

\begin{document}

\title{Multitensors as monads on categories of enriched graphs}

\author{Mark Weber}
\address{Department of Mathematics, Macquarie University}
\email{mark.weber.math@gmail.com}
\thanks{}
\maketitle
\begin{abstract}
In this paper we unify the developments of \cite{Batanin-MonGlobCats}, \cite{BataninWeber-EnHop} and \cite{Cheng-ComparingOperadic} into a single framework in which the interplay between multitensors on a category $V$, and monads on the category $\ca GV$ of graphs enriched in $V$, is taken as fundamental. The material presented here is the conceptual background for subsequent work: in \cite{BatCisWeb-LiftingTheorem} the Gray tensor product of 2-categories and the Crans tensor product \cite{Crans-ATensorForGrayCats} of Gray categories are exhibited as existing within our framework, and in \cite{Weber-Funny} the explicit construction of the funny tensor product of categories is generalised to a large class of Batanin operads.
\end{abstract}
\tableofcontents

\section{Introduction}
A monad on a category $\ca C$ is an excellent way of defining extra structure on the objects of $\ca C$. For instance in the globular approach to higher category theory \cite{Batanin-MonGlobCats} an $n$-dimensional categorical structure of a given type is defined as the algebras for a given monad on the category $\PSh {\G}_{{\leq}n}$ of $n$-globular sets.

Multitensors are another way of defining extra structure. Recall \cite{BataninWeber-EnHop} that a multitensor $E$ on a category $V$ is simply the structure of a lax monoidal category on $V$. As such it includes the assignment
\[ (X_1,...,X_n) \mapsto E(X_1,...,X_n) \]
of the $n$-fold tensor product of any finite sequence of objects of $V$ and non-invertible coherences including unit maps $u_X:X \to E(X)$ and substitution maps
\[ E(E(X_{11},...,X_{1n_1}),...,E(X_{k1},...,X_{kn_k})) \to E(X_{11},...,X_{1n_1},...,X_{k1},...,X_{kn_k}) \]
which satisfy some natural axioms. In particular, the unary case $n=1$ is interesting, and restricting attention just to this case one has a monad $E_1$ on $V$. On the other hand in the case where the unit is the identity and the substitutions are invertible, one refinds the usual notion of monoidal category, though expressed in an ``unbiased'' way in terms of $n$-ary tensor products.

From this perspective the notion of enriched category does not require the invertibility of these coherence maps, and so one has the notion of a category enriched in $E$ (also known as an ``$E$-category'') for any multitensor. Thus, a multitensor on a category $V$ is a way of endowing graphs enriched in $V$ with extra structure. Recall a graph $X$ enriched in $V$ is simply a set $X_0$ of objects, together with objects $X(a,b)$ of $V$ called ``homs'' for all pairs $(a,b)$ of objects of $V$. In particular an $E$-category structure on $X$ includes the structure of an $E_1$ algebra on the homs of $X$.

These two ways of defining extra structure are related. If $V$ has coproducts and the assignation
\[ (X_1,...,X_n) \mapsto E(X_1,...,X_n) \]
preserves coproducts in each variable, in which case we say that $E$ is a \emph{distributive} multitensor, then in a straight forward manner $E$ defines a monad $\Gamma E$ on the category $\ca GV$ of graphs enriched in $V$, whose algebras are $E$-categories. The purpose of this article is to study this process
\[ (V,E) \mapsto (\ca GV,\Gamma E) \]
of assigning a monad to a distributive multitensor in a systematic way.

The developments presented in this article are applied to simpifying and unifying earlier work in the subject \cite{Batanin-MonGlobCats} \cite{BataninWeber-EnHop} \cite{Cheng-ComparingOperadic}, and as a springboard for subsequent developments. In \cite{Weber-Funny} the funny tensor product of categories is exhibited as a special case of a symmetric monoidal closed structure that can be exhibited on the category of algebras of a wide class of higher operads. In \cite{BatCisWeb-LiftingTheorem} the Gray and Crans tensor products are exhibited within our emerging framework, weak $n$-categories with strict units are defined and then exhibited as obtainable via some iterated enrichment. For both \cite{Weber-Funny} and \cite{BatCisWeb-LiftingTheorem}, the work presented here is used extensively.

This article is organised as follows. In section(\ref{sec:categories-of-enriched-graphs}) categories of enriched graphs are studied. This uses basic categorical notions recalled and defined in appendix(\ref{sec:lcp}) related to the theory of locally presentable categories. In section(\ref{sec:multitensor-to-monad-construction}) the construction of monads from multitensors is discussed, and how properties on the multitensor correspond to properties on the corresponding monad is spelled out in detail in theorem(\ref{thm:preservation-by-Gamma-E}). The monads that arise from multitensors via our construction are characterised in section(\ref{sec:Monads-Operads-Multitensors}) theorem(\ref{thm:characterisation-of-image-of-Gamma}). Later in the same section the 2-functors underlying the multitensor to monad construction are given, at which point the connection with the formal theory of monads \cite{Street-FTM} is made.

This connection is exploited to explain the ubiquity of the distributive laws that arise in higher category theory \cite{Cheng-IteratedDist}. In section(\ref{sec:dist-laws-from-monoidal-monads}) the senses in which a monad and multitensor may distribute is spelled out as part of a generalisation of the classical theory of monad distributive laws of Beck \cite{Beck-DLaws}. As an application we give a very efficient construction of the monads for strict $n$-categories in section(\ref{sec:strict-n-cat-monad}). This is the construction at the level of monads which corresponds at the level of theories to the inductive formula $\Theta_{n{+}1} = \Delta \wr \Theta_n$ of \cite{Berger-IteratedWreathProduct}. We recover this formula from our perspective in section(\ref{sec:strict-n-cat-monad}), from more general considerations in section(\ref{ssec:wreath}) which bring together the developments of \cite{BergMellWeber-MonadsArities} with those of the present article.

In the setting of the theory of cartesian monads \cite{Burroni-TCats,Hermida-RepresentableMulticategories,Leinster-HDA-book} a \emph{$T$-operad} for a cartesian monad $T$ on a category $\ca E$ with pullbacks consists of another monad $A$ on $\ca E$ together with a cartesian monad morphism{
\footnotemark{\footnotetext{That is, $\alpha$'s naturality squares are pullbacks, and $\alpha$ satisfies axioms expressing its compatibility with the monad structures on $A$ and $T$.}}} $\alpha:A \to T$. Similarly given a cartesian multitensor $E$ on a category $V$ one defines an $E$-multitensor to consist of another multitensor $F$ on $V$ together with a cartesian multitensor morphism $\phi:F \to E$ \cite{BataninWeber-EnHop}. In section(\ref{sec:operads-multitensors-basic}) the basic correspondence between $E$-multitensors and $\Gamma E$-operads is given.

A weak $n$-category is an algebra of a \emph{contractible} $n$-operad{\footnotemark{\footnotetext{In this work we use the notion of contractibility given in \cite{Leinster-HDA-book} rather than the original notion of \cite{Batanin-MonGlobCats}.}}}. In section(\ref{sec:contractible-operads-and-multitensors}) we recall this notion, give an analogous notion of contractible multitensor, and in corollary(\ref{cor:contractible}), give the canonical relationship between them. Finally in section(\ref{sec:Trimble}) we recover Cheng's description \cite{Cheng-ComparingOperadic} of Trimble's definition of weak $n$-category.

\emph{Notation and terminology}. Given a monad $T$ on a category $V$ the forgetful functor from the category of Eilenberg-Moore algebras of $T$ is denoted as $U^T:V^T \to V$.  We denote a $T$-algebra as a pair $(X,x)$, where $X$ is the underlying object and $x:TX \to X$ is the algebra structure. When thinking of monads in a 2-category, it is standard practise to refer to them as pairs $(A,t)$ where $A$ is the underlying object, $t$ is the underlying endomorphism, and the unit and multiplication are left implicit. Similarly we refer to a lax monoidal category as a pair $(V,E)$ where $V$ is the underlying category, $E$ is the multitensor, and the unit and substitution are left implicit.

The category of presheaves on a given category $\C$ is denoted $\PSh {\C}$. Given a functor $F:\C \to \D$ we denote by $\D(F,1):\D \to \PSh {\C}$ with object map $D \mapsto \D(F(-),D)$. For the category of globular sets it is typical to denote the image of the yoneda embedding as
\[ \xygraph{{0}="p0" [r] {1}="p1" [r] {2}="p2" [r] {3}="p3" [r] {...}="p4" "p0":@<1ex>"p1":@<1ex>"p2":@<1ex>"p3":@<1ex>"p4":@{<-}@<1ex>"p3":@{<-}@<1ex>"p2":@{<-}@<1ex>"p1":@{<-}@<1ex>"p0"} \]
but then $0$ denotes the globular set with one vertex and no edges or higher cells. Thus we adopt the convention of using $0$ to denote objects of categories that we wish to think of as representing some underlying objects functor. Since initial objects are also important for us, we use the notation $\emptyset$ to denote them. While multicategories aren't directly multitensors, they become so after convolution -- see \cite{DayStreet-LaxMonoids}. Moreover when working seriously with multitensors, one is always manipulating functors of many variables, and so in fact working inside the $\CAT$-enriched multicategory of categories. It is for these reasons that we find the term ``multitensor'' appropriate.

\section{Categories of enriched graphs}
\label{sec:categories-of-enriched-graphs}

Preliminary to the correspondence between monads and multitensors that we describe in this paper, is the passage $V \mapsto \ca GV$ from an arbitrary category $V$, to the category $\ca GV$ of graphs enriched in $V$. In section(\ref{ssec:enriched-graphs}) we describe the basic properties of $\ca G$ as an endofunctor of $\CAT$, whose object map is $V \mapsto \ca GV$. Then in section(\ref{ssec:properties-of-GV}), we describe what categorical properties $\ca G$ preserves in theorem(\ref{thm:GV-locally-c-presentable}). From this it is clear that as far as basic categorical properties are concerned, $\ca GV$ is at least as good as $V$.

\subsection{Enriched graphs}
\label{ssec:enriched-graphs}
\begin{definition}\label{def:enriched-graph}
Let $V$ be a category. A \emph{graph $X$ enriched in $V$} consists of an underlying set $X_0$ whose elements are called \emph{objects}, together with an object $X(a,b)$ of $V$ for each ordered pair $(a,b)$ of objects of $X$. The object $X(a,b)$ will sometimes be called the \emph{hom} from $a$ to $b$. A morphism $f:X{\rightarrow}Y$ of $V$-enriched graphs consists of a function $f_0:X_0{\rightarrow}Y_0$ together with a morphism $f_{a,b}:X(a,b){\rightarrow}Y(fa,fb)$ for each $(a,b)$. The category of $V$-graphs and their morphisms is denoted as $\ca GV$, and we denote by $\ca G$ the obvious 2-functor
\[ \begin{array}{lccr} {\ca G : \CAT \rightarrow \CAT} &&& {V \mapsto \ca GV} \end{array} \]
with object map as indicated.
\end{definition}
Note that for $n \in \N$, $\ca G^n\Set$ is the category of $n$-globular sets, and that $\ca G\tn{Glob} \iso \tn{Glob}$ where $\tn{Glob}$ denotes the category of globular sets. In fact applying the 2-functor $\ca G$ successively to the inclusion of the empty category into the point (ie the terminal category), one obtains the inclusion of the category of $(n{-}1)$-globular sets into the category of $n$-globular sets. In the case $n>0$ this is the inclusion with object map
\[ \begin{array}{lcr}
{\xygraph{{X_0}="l" [r] {...}="m" [r] {X_{n{-}1}}="r" "r":@<1ex>"m":@<1ex>"l" "r":@<-1ex>"m":@<-1ex>"l"}} & \mapsto &
{\xygraph{{X_0}="l" [r] {...}="m" [r] {X_{n{-}1}}="r" [r] {\emptyset}="rr" "rr":@<1ex>"r":@<1ex>"m":@<1ex>"l" "rr":@<-1ex>"r":@<-1ex>"m":@<-1ex>"l"}}
\end{array} \]
and when $n{=}0$ this is the functor $1{\rightarrow}\Set$ which picks out the empty set. Thus there is exactly one $(-1)$-globular set which may be identified with the empty set.

When $V$ has an initial object $\emptyset$, one can regard any sequence of objects $(Z_1,...,Z_n)$ of $V$ as a $V$-graph. The object set is $\{0,...,n\}$, $(Z_1,...,Z_n)(i-1,i) = Z_i$ for $1{\leq}i{\leq}n$, and all the other homs are equal to $\emptyset$. We denote also by $0$ the $V$-graph corresponding to the empty sequence $()$. Note that $0$ is a representing object for the forgetful functor $(-)_0:\ca GV{\rightarrow}\Set$ which sends an enriched graph to its underlying set of objects. Globular pasting diagrams \cite{Batanin-MonGlobCats} may be regarded as iterated sequences, for instance $(0,0,0)$ and $((0,0),(0),(0,0,0))$ correspond respectively to
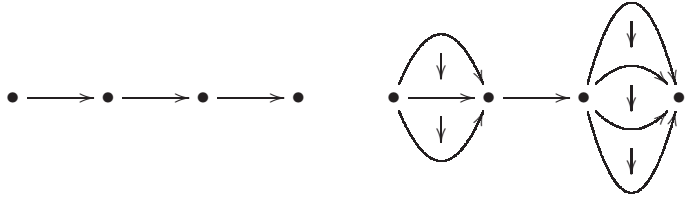
\[ \xygraph{{\xybox{\xygraph{{\bullet}="p1" [r] {\bullet}="p2" [r] {\bullet}="p3" [r] {\bullet}="p4" "p1":"p2"^-{}:"p3"^-{}:"p4"^-{}}}} [r(4)]
{\xybox{\xygraph{{\bullet}="p1" [r] {\bullet}="p2" [r] {\bullet}="p3" [r] {\bullet}="p4"
"p1":@/^{2pc}/"p2"|-{}="p11" "p1":"p2"|-{}="p12" "p1":@/^{-2pc}/"p2"|-{}="p13" "p11":@{}"p12"|(.3){}="s11"|(.7){}="f11" "s11":"f11" "p12":@{}"p13"|(.3){}="s12"|(.7){}="f12" "s12":"f12" "p2":"p3"
"p3":@/^{3pc}/"p4"|-{}="p31" "p3":@/^{1pc}/"p4"|-{}="p32" "p3":@/^{-1pc}/"p4"|-{}="p33" "p3":@/^{-3pc}/"p4"|-{}="p34" "p31":@{}"p32"|(.3){}="s31"|(.7){}="f31" "s31":"f31" "p32":@{}"p33"|(.3){}="s32"|(.7){}="f32" "s32":"f32" "p33":@{}"p34"|(.3){}="s33"|(.7){}="f33" "s33":"f33"}}}} \]
when one starts with $V=\Set$. We denote by ``$n$'' the free-living $n$-cell, defined inductively by $n+1=(n)$.

It is often better to think of $\ca G$ as taking values in $\CAT/\Set$. By applying the endofunctor $\ca G$ to the unique functor $t_V:V{\rightarrow}1$ for each $V$, produces $(-)_0$ which sends an enriched graph to its underlying set of objects. This 2-functor
\[ \ca G_1 : \CAT \rightarrow \CAT/\Set \]
has a left adjoint described as follows. First to a given functor $f:A \to \Set$ we denote by $f^{\times 2}:A \to \Set$ the functor with object map $a \mapsto f(a) \times f(a)$. Then to a given functor $g:A \to \Set$ we denote by $g_{\bullet}$ the domain of the discrete opfibration corresponding to $g$ via the Grothendieck construction. That is, $g_{\bullet}$ can be defined via the pullback
\[ \xygraph{{g_{\bullet}}="tl" [r] {\Set_{\bullet}}="tr" [d] {\Set}="br" [l] {A}="bl" "tl":"tr"^-{}:"br"^-{U}:@{<-}"bl"^-{g}:@{<-}"tl"^-{} "tl":@{}"br"|-{pb}} \]
where $U$ is the forgetful functor from the category of pointed sets and base point preserving maps. The left adjoint to $\ca G_1$ is then described on objects by $f \mapsto f^{\times 2}_{\bullet}$. Explicitly $f^{\times 2}_{\bullet}$ has as objects triples $(a,x,y)$ where $a$ is an object of $A$, and $(x,y)$ is an ordered pair of objects of $fa$. Maps in $f^{\times 2}_{\bullet}$ are maps in $A$ which preserve these base points.

It is interesting to look at the unit and counit of this 2-adjunction. Given a category $V$, $(\ca Gt_V)_{\bullet}^{\times 2}$ is the category of bipointed enriched graphs in $V$. The counit $\varepsilon_V:(\ca Gt_V)_{\bullet}^{\times 2} \to V$ sends $(X,a,b)$ to the hom $X(a,b)$. When $V$ has an initial object $\varepsilon_V$ has a left adjoint $L_V$ given by $X \mapsto ((X),0,1)$. Given a functor $f:A{\rightarrow}\Set$ the unit $\eta_f:A{\rightarrow}\ca G(f_{\bullet}^{\times 2})$ sends $a \in A$ to the enriched graph whose objects are elements of $fa$, and the hom $\eta_f(a)(x,y)$ is given by the bipointed object $(a,x,y)$.

A given functor $f:A \to \ca GV$ thus admits a canonical factorisation
\[ \xygraph{!{0;(2,0):} {A}="l" [r] {\ca G(f_{0 \bullet}^{\times 2})}="m" [r] {\ca GV}="r" "l":"m"^-{\eta_{f_0}}:"r"^-{\ca G\ca H_V(f)}} \]
where on objects one has $\ca H_V(a,x,y) = f(a)(x,y)$. This is the generic factorisation of $f$ in the sense of \cite{Weber-Fam2fun}. The adjointness $(-)_{\bullet}^{\times 2} \ladj \ca G_1$ says that $f$ is uniquely determined by its object part $f_0 := (-)_0f$ and its hom data $\ca H_V(f)$. For the sake of computing colimits in $\ca GV$, as we will in section(\ref{ssec:properties-of-GV}), it is worth noting that one can reorganise the data of a lax triangle as on the left in
\[ \xygraph{ {\xybox{\xygraph{{A}="l" [r(2)] {B}="r" [dl] {\ca GV}="b" "l":"r"^-{k}:"b"^-{h}:@{<-}"l"^-{f} [d(.5)r(.85)] :@{=>}[r(.3)]^-{\phi}}}} [r(3)]
{\xybox{\xygraph{{A}="l" [r(2)] {B}="r" [dl] {\Set}="b" "l":"r"^-{k}:"b"^-{h_0}:@{<-}"l"^-{f_0} [d(.5)r(.85)] :@{=>}[r(.3)]^-{\phi_0}}}} [r(3)]
{\xybox{\xygraph{{f_{0 \bullet}^{\times 2}}="l" [r(2)] {h_{0 \bullet}^{\times 2}}="r" [dl] {V}="b" "l":"r"^-{\phi_{0 \bullet }^{\times 2}}:"b"^-{\ca H_V(h)}:@{<-}"l"^-{\ca H_V(f)} [d(.5)r(.85)] :@{=>}[r(.3)]^-{\ca H_{\phi}}}}}} \]
into $\ca GV$ in the same way. The middle triangle is just $(-)_0 \phi$. In the right hand triangle, $\phi_{0 \bullet}^{\times 2}$ is the evident functor with object map $(a,x,y) \mapsto (ka,\phi_a(x),\phi_a(y))$ which is determined by $\phi_0$. The natural transformation $\ca H_{\phi}$ has components given by the hom maps of the components of $\phi$, that is $(\ca H_{\phi})_{(a,x,y)}$ is the map $(\phi_a)_{x,y}:f(a)(x,y) \to hk(a)(\phi_a(x),\phi_a(y))$. It then follows easily from unpacking the data involved that
\begin{lemma}\label{lem:lax-triangles-into-GV}
Given $f:A \to \ca GV$, $k:A \to B$ and $h:B \to \ca GV$, the assignment $\phi \mapsto (\phi_0,\ca H_{\phi})$ is a bijection which is natural in $h$.
\end{lemma}
Suppose that one has an object $0$ in a category $A$, and $f$ is the representable $f=A(0,-)$. Then $f_{\bullet}^{\times 2}$ may be regarded as the category of endo-cospans of the object $0$, that is to say the category of diagrams
\[ 0 \rightarrow a \leftarrow 0 \]
and a point of $a \in A$ is now just a map $0{\rightarrow}a$. When $A$ is also cocomplete one can compute a left adjoint to $\eta_f$. To do this note that a graph $X$ enriched in $f_{\bullet}^{\times 2}$ gives rise to a functor
\[ \overline{X} : X^{(2)}_0 \rightarrow A \]
where $X_0$ is the set of objects of $X$. For any set $Z$, $Z^{(2)}$ is defined as the following category. It has two kinds of objects: an object being either an element of $Z$, or an ordered pair of elements of $Z$. There are two kinds of non-identity maps
\[ x \rightarrow (x,y) \leftarrow y \]
where $(x,y)$ is an ordered pair from $Z$, and $Z^{(2)}$ is free on the graph just described. A more conceptual way to see this category is as the category of elements of the graph
\[ \xygraph{{Z{\times}Z}="l" [r] {Z}="r" "l":@<1ex>"r" "l":@<-1ex>"r"} \]
where the source and target maps are the product projections, as a presheaf on the category
\[ \xygraph{{\G_{\leq{1}}} [r(.75)] {=} [r(1.25)] {\xybox{\xygraph{0 [r] 1 "0":@<1ex>"1":@<1ex>@{<-}"0"}}} *\frm{-}}  \]
and so there is a discrete fibration $Z^{(2)}{\rightarrow}\G_{{\leq}1}$. The functor $\overline{X}$ sends singletons to $0 \in A$, and a pair $(x,y)$ to the head of the hom $X(x,y)$. The arrow map of $\overline{X}$ encodes the bipointings of the homs. One may then easily verify
\begin{proposition}\label{prop:unit-G1}
Let $0 \in A$, $f=A(0,-)$ and $A$ be cocomplete. Then $\eta_f$ has left adjoint given on objects by $X \mapsto \colim(\overline{X})$.
\end{proposition}
There is a close connection between $\ca G$ and the $\Fam$ construction. A very mild reformulation of the notion of $V$-graph is the following: a $V$-graph $X$ consists of a set $X_0$ together with an $(X_0{\times}X_0)$-indexed family of objects of $V$. Together with the analogous reformulation of the maps of $\ca GV$, this means that we have a pullback square
\[ \PbSq {\ca GV} {\Fam{V}} {\Set} {\Set} {(-)_0=\ca Gt_V} {} {\Fam(t_V)} {(-)^2} \]
in $\CAT$, and thus a cartesian 2-natural transformation $\ca G \to \Fam$. From \cite{Weber-Fam2fun} theorem(7.4) we conclude
\begin{proposition}\label{prop:GFam2fun}
$\ca G$ is a familial 2-functor.
\end{proposition}
In particular notice that for all $V$, the functor $(-)_0:\ca GV \to \Set$ has the structure of a split fibration. The cartesian morphisms are exactly those which are \emph{fully faithful}, which are those morphisms of $V$-graphs whose hom maps are isomorphisms. The vertical-cartesian factorisation of a given $f:X \to Y$ corresponds to its factorisation as an identity on objects map followed by a fully-faithful map. Moreover it follows from the theory of \cite{Weber-Fam2fun} that $\ca G$ preserves conical connected limits as well as all the notions of ``Grothendieck fibration'' which one can define internal to a finitely complete 2-category, and that the obstruction maps for comma objects are right adjoints. In addition to this we have
\begin{lemma}\label{lem:G-EM-object}
$\ca G$ preserves Eilenberg-Moore objects.
\end{lemma}
\noindent Given a monad $T$ on a category $V$, we shall write $V^T$ for the category of $T$-algebras and morphisms thereof, and $U^T:V^T{\rightarrow}V$ for the forgetful functor. We shall denote a typical object of $V^T$ as a pair $(X,x)$, where $X$ is the underlying object in $V$ and $x:TX{\rightarrow}X$ is the $T$-algebra structure. From \cite{Street-FTM} the 2-cell $TU^T \to U^T$, whose component at $(X,x)$ is $x$ itself has a universal property exhibiting $V^T$ as a kind of 2-categorical limit called an \emph{Eilenberg-Moore object}. See \cite{Street-FTM} or \cite{LackStreet-FTM2} for more details on this general notion. The direct proof that for any monad $T$ on a category $V$, the obstruction map $\ca G(V^T){\rightarrow}\ca G(V)^{\ca G(T)}$ is an isomorphism comes down to the obvious fact that for any $V$-graph $B$, a $\ca GT$-algebra structure on $B$ is the same thing as a $T$-algebra structure on the homs of $B$, and similarly for algebra morphisms.

Returning to the consideration of $\ca G_1$, our final observation for this section is
\begin{proposition}\label{prop:G1-locally-ff}
$\ca G_1 : \CAT \to \CAT/\Set$ is locally fully faithful.
\end{proposition}
\begin{proof}
Given functors $F,G : V \to W$, the data of a natural transformation $\phi:\ca GF \to \ca GG$ over $\Set$ amounts to giving for each $X \in \ca GV$ and $a,b \in X_0$, maps $\phi_{X,a,b} : FX(a,b) \to GX(a,b)$, such that for $f:X \to Y$ one has the naturality condition for $f$ between $a$ and $b$:
\[ \xygraph{!{0;(2.5,0):(0,.4)::} {FX(a,b)}="tl" [r] {GX(a,b)}="tr" [d] {GY(fa,fb)}="br" [l] {FY(fa,fb)}="bl" "tl":"tr"^-{\phi_{X,a,b}}:"br"^-{Gf_{a,b}}:@{<-}"bl"^-{\phi_{Y,fa,fb}}:@{<-}"tl"^-{Ff_{a,b}} "tl":@{}"br"|-{=}} \]
So we define $\phi':F \to G$ by $\phi'_Z = \phi_{(Z),0,1}$. One has $c:(X(a,b)) \to X$ in $\ca GV$ with object map $(0,1) \mapsto (a,b)$ and hom map $c_{0,1}$ the identity. The naturality condition for $c$ between $0$ and $1$ yields
$\phi_{(X(a,b)),0,1} = \phi_{X,a,b}$
from which it follows that $\phi = \ca G\phi'$. Conversely
$(\ca G\phi)'_Z = (\ca G\phi)_{(Z),0,1} = \phi_Z$
and so $\phi \mapsto \phi'$ is the inverse of
\[ \begin{array}{lccr} {\CAT(V,W)(F,G) \to \CAT/\Set(\ca GV,\ca GW)(\ca GF,\ca GG)} &&& {\psi \mapsto \ca G\psi.} \end{array} \qedhere \]
\end{proof}

\subsection{Properties of $\ca GV$}
\label{ssec:properties-of-GV}
This section contains a variety of results from which it is clear that as a category, $\ca GV$ is at least as good as $V$. To begin with, any limit that $V$ possesses is also possessed by $\ca GV$.
\begin{proposition}\label{prop:limits-in-GV}
Let $I$ be a small category. If $V$ admits limits of functors out of $I$, then so does $\ca GV$ and these are preserved by $(-)_0$.
\end{proposition}
\begin{proof}
Let $F:I \to \ca GV$ be a functor. We construct its limit $L$ directly as follows. First we take the set $L_0$ to be the limit of the functor $F(-)_0$, writing $\lambda_{i,0}:L_0 \to F(i)_0$ for a typical component of the limit cone. Without loss of generality one can represent the elements of $L_0$ explicitly as matching families of elements of the $F(i)_0$. That is, any such element is a family
\[ x := (x_i \in F(i)_0 \,\, : \,\, i \in I) \]
such that for all $f:i \to j$ in $I$, one has $Ff(x_i)=x_j$. Given an ordered pair $(x,y)$ of such families, one has a functor
\[ \begin{array}{lccr} {F_{x,y}:I \to V} &&& {i \mapsto Fi(x_i,y_i)} \end{array} \]
with indicated object map. One then defines the hom $L(x,y)$ to be the limit in $V$ of $F_{x,y}$, and we write $\lambda_{i,x,y}:Fi(x_i,y_i) \to L(x,y)$ for the components of the limit cone. These provide the hom maps, and $\lambda_{i,0}$ the object functions, of morphisms $\lambda_i:L \to Fi$. It is easily verified that these exhibit $L$ as a limit of $F$.
\end{proof}
\noindent In particular from the explicit construction of limits just described, it is clear that $\ca GV$ possesses some pullbacks under no conditions on $V$.
\begin{corollary}\label{cor:pbs-in-GV-along-ff}
For any category $V$, $\ca GV$ admits all pullbacks along fully faithful maps, and these are preserved by $(-)_0$. Moreover the pullback of a fully faithful map is itself fully faithful.
\end{corollary}
\begin{proof}
In this case the construction of proposition(\ref{prop:limits-in-GV}) goes through because the pullbacks in $V$ that arise in the construction are all along an isomorphism, and such clearly exist in any $V$. The last statement follows from this explicit construction since isomorphisms in any $V$ are pullback stable.
\end{proof}
By lemma(\ref{lem:lax-triangles-into-GV}) one can compute the left kan extension of $F:I \to \ca GV$, along any functor $G:I \to J$ between small categories, in the following way. First compute the left extension $K_0:J \to \Set$ of $F_0$ along $G$ denoting the universal 2-cell as $\kappa_0:F_0 \to K_0G$. Then given sufficient colimits in $V$, compute the left extension $\ca H_V(K):K_{0 \bullet}^{\times 2} \to V$ of $\ca H_V(F)$ along $\kappa_{0 \bullet}^{\times 2}$, denoting the universal 2-cell as $\ca H_{\kappa}:\ca H_V(F) \to \ca H_V(K)$. Thus we have the object part $K_0$ and hom data $\ca H_V(K)$ of a functor $K:J \to \ca GV$. The natural transformation $\kappa:F \to KG$ corresponding to $(\kappa_0,\ca H_{\kappa})$ by lemma(\ref{lem:lax-triangles-into-GV}) clearly exhibits $K$ as the left extension of $F$ along $G$, by a straight forward application of lemma(\ref{lem:lax-triangles-into-GV}) and the definition of ``left kan extension''.

When $J=1$ note that $K_{0 \bullet}^{\times 2}$ is just the discrete category $K_0 \times K_0$, and so for $x,y \in K_0$ one may compute $\ca H_{V}(K)(x,y)$ as the colimit of $\ca H_V(F)$ restricted to the fibre of $\kappa_{0 \bullet}^{\times 2}$ over $(x,y)$.
\begin{proposition}\label{prop:colimits-GV}
\begin{enumerate}
\item For any category $V$, $\ca GV$ has a strict initial object.\label{pcase:initial-GV}
\item If $V$ has an initial object, then $\ca GV$ has coproducts and pullbacks along coproduct inclusions.\label{pcase:coproducts-GV}
\item If $V$ has a strict initial object, then every $X \in \ca GV$ decomposes as a coproduct of connected objects.
\label{pcase:GV-locally-connected}
\item If $\lambda$ is a regular cardinal and $V$ has $\lambda$-filtered colimits, then so does $\ca GV$.\label{pcase:filtered-colims-GV}
\item If $V$ has all small colimits, then so does $\ca GV$.\label{pcase:all-colimits-GV}
\end{enumerate}
In each case the colimits in $\ca GV$ under discussion are preserved by $(-)_0$.
\end{proposition}
\begin{remark}\label{rem:explain-loc-conn-GV}
In appendix(\ref{sec:lcp}) we recall some of the general theory of connected objects. An alternative formulation of (\ref{pcase:GV-locally-connected}) is that $V$ having a strict initial object implies that $\ca GV$ is locally connected in the sense of definition(\ref{def:locally-connected-category}), which by lemma(\ref{lem:easy-ext}), implies that $\ca GV$ is extensive. If moreover one has finite limits in $V$, and thus also in $\ca GV$, then by proposition(\ref{prop:lext-decompose-charn}), this coproduct decomposition into connected objects is essentially unique, and the assignation $X \mapsto \pi_0(X)$ is the object map of a left adjoint to the functor $(-) \cdot 1$ given by taking copowers with $1$.
\end{remark}
\begin{proof}
(\emph{of proposition(\ref{prop:colimits-GV})}).
By the above uniform construction of colimits one has (\ref{pcase:all-colimits-GV}), and the preservation of any colimit by $(-)_0$ when it is constructed in this way. The empty $V$-enriched graph is clearly strictly initial in $\ca GV$ and so (\ref{pcase:initial-GV}) follows.

(\ref{pcase:coproducts-GV}): In the case where $I$ is discrete, $\kappa_{0 \bullet}^{\times 2}$ is the inclusion
\[ \coprod_{i \in I} F(i)_0 \times F(i)_0 \to \coprod_i F(i)_0 \times \coprod_i F(i)_0 \]
(between discrete categories) which picks out pairs $(x,y)$ which live in the same component. Thus the coproduct $X:= \coprod_i X_i$ in $\ca GV$ is defined to have objects those of the disjoint union of the objects sets of the $X_i$, and homs given by $X(x,y) = X_i(x,y)$ when $x$ and $y$ are both in $X_i$, and $\emptyset$ otherwise. The required pullbacks exist by corollary(\ref{cor:pbs-in-GV-along-ff}) since coproduct inclusions are clearly fully faithful.

(\ref{pcase:GV-locally-connected}): Let $X \in \ca GV$. We define the relation on $X_0$ as 
\[ \{(a,b) \,\, : \,\, V(X(a,b),\emptyset) = \emptyset \} \subseteq X_0 \times X_0 \]
and say that $a$ and $b$ are \emph{in the same component} of $X$ when they are identified by the equivalence relation generated by the above relation. Denote by $\pi_0(X)$ the set of equivalence classes, which themselves are called connected components. For $i \in \pi_0(X_0)$ we denote by $X_i$ the full sub-$V$-graph of $X$ whose objects are those of $X$'s $i$-th component, and by $c_{X,i}:X_i \to X$ the evident inclusion.

Suppose $\pi_0(X)=1$ and $f:X \to \coprod_j Y_j$ is a graph morphism into some coproduct of $V$-graphs. We will now show that such an $f$ factors uniquely through a unique summand, so that $X$ is connected. From the explicit construction of coproducts in $\ca GV$ in proposition(\ref{prop:colimits-GV}), it is clear that coproduct inclusions in $\ca GV$ are mono. Thus it suffices to show that $f$ factors through a unique summand. Since $X$ is non-empty it suffices to show that $f$ sends any pair $(a,b)$ of elements of $X_0$ to the same summand. Since $\pi_0(X)=1$ there is a sequence
\[ (x_j \,\, : \,\, 0 \leq j \leq n) \]
of elements of $X_0$, such that $x_0=a$, $x_n=b$ and for all $1 \leq j \leq n$, the set
\[ S_{i,j} := V(X(x_{j{-}1},x_j),\emptyset) \times V(X(x_j,x_{j{-}1}),\emptyset) \]
is empty. But if $fx_{j{-}1}$ and $fx_j$ are in different components of the coproduct, then both $Y(fx_{j{-}1},fx_j)$ and $Y(fx_j,fx_{j{-}1})$ would be $\emptyset$, and so the hom maps $(f_{x_{j{-}1},x_j},f_{x_j,x_{j{-}1}})$ would give an element of $S_{i,j}$. Thus all the elements $(x_0,...,x_n)$ are sent to the same component of the coproduct by $f$.

Thus for general $X$, the $X_i$ for $i \in \pi_0(X)$ are connected. Since $V$'s initial object is strict, $a$ and $b$ in $X_0$ will be in a different component iff $X(a,b) \iso \emptyset$. Thus by the explicit construction of coproducts in $\ca GV$, the $c_{X,i}$ are the components of a coproduct cocone.

(\ref{pcase:filtered-colims-GV}): When $I$ is $\lambda$-filtered we note that since $\lambda$-filtered colimits in $\Set$ commute with binary products, the cocone
\[ \kappa_{0,i} \times \kappa_{0,i} : F(i)_0 \times F(i)_0 \to K_0 \times K_0 \]
in $\Set$ is also a colimit cocone. Thus the functor $\kappa_{0 \bullet}^{\times 2} : F_{0 \bullet}^{\times 2} \to K_0 \times K_0$ has another interpretation. Since $F_{0 \bullet}^{\times 2}$ is the category of elements of the functor $i \mapsto F(i)_0 \times F(i)_0$, then by the above remark $K_0 \times K_0$ is the set of connected components of $F_{0 \bullet}^{\times 2}$ and $\kappa_{0 \bullet}^{\times 2}$ is the canonical projection. So the fibres of $\kappa_{0 \bullet \bullet}$ are the connected components of $F_{0 \bullet}^{\times 2}$. Since the evident forgetful functor $F_{0 \bullet}^{\times 2} \to I$ is a discrete opfibration, the connected components of $F_{0 \bullet}^{\times 2}$ are themselves $\lambda$-filtered. Thus a fibre of $\kappa_{0 \bullet}^{\times 2}$ over a given $(x,y)$ will itself be $\lambda$-filtered, and so $\lambda$-filtered colimits in $V$ will suffice for the construction of the colimit in this case.
\end{proof}
\begin{remark}\label{rem:fibrewise-limits-and-colimits}
There is one very easy to understand class of limit/colimit of $V$-graphs. These are those for functors $F:J \to \ca GV$ where $J$ is connected and $F_0:J \to \Set$ is constant at some set $X$. For then the limit or colimit of $F$ may also be taken to have object set $X$, and one computes the hom between $a$ and $b \in X$ of the limit or colimit by taking the limit or colimit in $V$ of the functor $J \to V$ with object map $j \mapsto F(j)(a,b)$.
\end{remark}
Now we describe how the 2-functor $\ca G$ preserves locally (c)-presentable categories and Grothendieck toposes. First we require a general lemma which produces a strongly generating or dense subcategory of $\ca GV$ from one in $V$ in a canonical way. Recall that a functor $i:\ca D \to V$ is strongly generating when $V(i,1):\ca D \to \PSh {\ca D}$ is conservative (ie reflects isomorphisms), and that $i$ is dense when $V(i,1)$ is fully faithful. Moreover recall that an object $D$ of a category $V$ is said to be \emph{small projective} when $V(D,-)$ preserves all small colimits.

For the following lemma we require also the endofunctor
\[ (-)_+ : \Cat \longrightarrow \Cat \]
of the 2-category $\Cat$ of small categories. For a small category $\C$, one describes $\C_+$ as follows. There is an injective on objects fully faithful functor
\[ \begin{array}{lcr} {\iota_{\C} : \C \rightarrow \C_+} && {C \mapsto C_+} \end{array} \]
and $\C_+$ has an additional object $0$ not in the image of $\iota_{\C}$. Moreover for each $C \in \C$ one has maps
\[ \begin{array}{lcr} {\sigma_C : 0 \rightarrow C_+} && {\tau_C : 0 \rightarrow C_+} \end{array} \]
and for all $f:C{\rightarrow}D$ one has the equations $f_+\sigma_C=\sigma_D$ and $f_+\tau_C=\tau_D$. Starting from the terminal category and iterating $(-)_+$ $n$ times gives the usual site $\G_{\leq n}$
\[ \xygraph{{0}="l" [r] {...}="m" [r] {n}="r" "l":@<1ex>"m"^-{\sigma}:@<1ex>"r"^-{\sigma} "l":@<-1ex>"m"_-{\tau}:@<-1ex>"r"_-{\tau}} \]
for $n$-globular sets.

Given a small category $\ca D$ and a functor $i:\ca D \to V$ where $V$ has an initial object, one has a functor $i^+:\ca D_+ \to \ca GV$ given on objects by $i^+(0)=0$ and $i^+(D_+)=(iD)$, fitting into
\[ \xygraph{{\ca D}="p0" [r] {V}="p1" [d] {\ca GV}="p2" [l] {\ca D_{+}}="p3" "p0":"p1"^-{i}:"p2"^-{(-)}:@{<-}"p3"^-{i^+}:@{<-}"p0"^-{\iota_{\ca D}} "p0":@{}"p2"|-{\tn{pb}}} \]
in $\CAT$. Note moreover that when $V$'s initial object is strict, the fully faithfulness of $i$ implies that of $i^+$.
\begin{lemma}\label{lem:GV-dense}
Suppose that $V$ has a strict initial object, $\ca D$ is a small category and $i:\ca D \to V$ is a fully faithful functor. Let $\lambda$ be a regular cardinal.
\begin{enumerate}
\item If $i$ is strongly generating then so is $i^+$. \cite{KellyLack-NiceVCat} \label{lemcase:Dpr-strong-generator}
\item If $i$ is dense then so is $i^+$.\label{lemcase:Dpr-GV-dense}
\item If the objects in the image of $i$ are connected then so are those in the image of $i^+$.\label{lemcase:Dpr-connected}
\item If $V$ has $\lambda$-filtered colimits and the objects in the image of $i$ are $\lambda$-presentable, then so are those of $i^+$. \cite{KellyLack-NiceVCat} \label{lemcase:Dpr-lampres}
\item If $V$ has small colimits and the objects in the image of $i$ are small projective, then so are those of $i^+$.\label{lemcase:Dpr-small-projective}
\end{enumerate}
\end{lemma}
\begin{proof}
For convenience we regard $i$ as an inclusion of a full subcategory. Since $V$ has a strict initial object, $i^+$ is also fully faithful, and so we regard it as an inclusion also. Let $f:X \to Y$ be in $\ca GV$. Suppose that $\ca GV(0,f)$ is a bijection, and that for all $D \in \ca D$, $\ca GV((D),f)$ are bijections. For (\ref{lemcase:Dpr-strong-generator}) we must show that $f$ is an isomorphism. To say that $\ca GV(0,f)$ is a bijection is to say that $f$ is bijective on objects, and so it suffices to show that $f$ is fully faithful. Let $a,b \in X_0$ and $D \in \ca D$. Note that the hom set $V(D,X(a,b))$ may be recovered as the pullback of the cospan
\[ \xygraph{!{0;(2.5,0):} {1}="l" [r] {\ca GV(0,X) \times \ca GV(0,X)}="m" [r(2)] {\ca GV((D),X)}="r" "l":"m"^-{(a,b)}:@{<-}"r"^-{(\ca GV(i_0,X),\ca GV(i_1,X))}} \]
where $i_0$ and $i_1$ pick out the objects $0$ and $1$ of $(D)$ respectively. Moreover the function $V(D,f_{a,b})$ is induced by the isomorphism of cospans
\[ \xygraph{!{0;(2.5,0):(0,.4)::} {1}="l1" [r] {\ca GV(0,X) \times \ca GV(0,X)}="m1" [r(2)] {\ca GV((D),X)}="r1" "l1":"m1"^-{(a,b)}:@{<-}"r1"^-{(\ca GV(i_0,X),\ca GV(i_1,X))} 
"l1" [d] {1}="l2" [r] {\ca GV(0,Y) \times \ca GV(0,Y)}="m2" [r(2)] {\ca GV((D),Y)}="r2" "l2":"m2"_-{(fa,fb)}:@{<-}"r2"_-{(\ca GV(i_0,Y),\ca GV(i_1,Y))} "l1":"l2"^-{} "m1":"m2"^-{\ca GV(0,f) \times \ca GV(0,f)} "r1":"r2"^-{\ca GV((D),f)}} \]
and so $V(D,f_{a,b})$ is also bijective. Since this is true for all $D \in \ca D$ and $\ca D$ is a strong generator, it follows that $f_{a,b}$ is an isomorphism. Thus $f$ is fully faithful and so (\ref{lemcase:Dpr-strong-generator}) follows.

Given functions
\[ f_{E} : \ca GV(E,X) \rightarrow \ca GV(E,Y) \]
natural in $E \in \ca D_+$, we must show for (\ref{lemcase:Dpr-GV-dense}) that there is a unique $f:X{\rightarrow}Y$ such that $f_{E}=\ca GV(E,f)$. The object map of $f$ is forced to be $f_0$, and naturality with respect to the maps $i_0$ and $i_1:0 \to (D)$ ensures that the functions $f_{E}$ amount to $f_0$ together with functions
\[ f_{D,a,b} : \ca G(t_V)_{\bullet}^{\times 2}((0,(D),1),(a,X,b)) \rightarrow \ca G(t_V)_{\bullet}^{\times 2}((0,(D),1),(f_0a,Y,f_0b)) \]
natural in $D \in \ca D$ for all $a,b \in X_0$. Recall from section(\ref{ssec:enriched-graphs}) that $\varepsilon_V:\ca G(t_V)_{\bullet}^{\times 2} \to V$ has a left adjoint given on objects by $Z \mapsto (0,(Z),1)$. By this adjointness the above maps are in turn in bijection with maps
\[ f'_{D,a,b} : V(D,X(a,b)) \rightarrow V(D,Y(f_0a,f_0b)) \]
natural in $D \in \ca D$ for all $a,b \in X_0$, and so by the density of $\ca D$ one has unique $f_{a,b}$ in $V$ such that $f'_{D,a,b}=V(D,f_{a,b})$. Thus $f_0$ and the $f_{a,b}$ together form the object and hom maps of the unique desired map $f$, and so (\ref{lemcase:Dpr-GV-dense}) follows.

By proposition(\ref{prop:colimits-GV}) $(-)_0$ preserves all the necessary colimits, so that in the case of $(\ref{lemcase:Dpr-connected})$ $0$ is connected, in the case of (\ref{lemcase:Dpr-lampres}) it is $\lambda$-presentable and in the case of (\ref{lemcase:Dpr-small-projective}) it is small projective.

Recall the uniform construction of a colimit of $F:I \to \ca GV$ described for proposition(\ref{prop:colimits-GV}), and write $\kappa_i:Fi \to K$ for the universal cocone. Then for $D \in V$ the cocone $\ca GV((D),\kappa_i)$ induces an obstruction map
\[ \gamma_{(D),\kappa_i} : \colim_{i \in I} \ca GV((D),Fi) \to \ca GV((D),K) \]
which measures the extent to which $\ca GV((D),-)$ preserves the colimit of $F$. We shall give an alternative description of this map in terms of the homs of $V$. First observe that any map $f:(D) \to X$ amounts to an ordered pair $(a,b)$ of objects of $X$ picked out by $f_0$, and the hom map $f_{0,1}:D \to X(a,b)$, and so one has a bijection
\[ \ca GV((D),X) \iso \coprod_{a,b{\in}X_0} V(D,X(a,b)). \]
Second for $a,b \in K_0$ recall from the construction of the colimit $K$ that one has a colimit cocone
$\kappa_{i,\alpha,\beta} : Fi(\alpha,\beta) \to K(a,b)$
in $V$ where $(\alpha,i,\beta)$ ranges over the fibre of $\kappa_{0 \bullet}^{\times 2}:F_{0 \bullet}^{\times 2} \to K_0 \times K_0$ over $(a,b)$. Thus one has an obstruction map
\[ \gamma_{D,\kappa_{i,\alpha,\beta}} : \colim_{(i,\alpha,\beta) \in (\kappa_{0 \bullet}^{\times 2})^{-1}(a,b)} V(D,Fi(\alpha,\beta)) \to V(D,K(a,b)) \]
measuring the extent to which $V(D,-)$ preserves the defining colimit of $K(a,b)$. The above isomorphisms exhibit $\gamma_{(D),\kappa_i}$ as isomorphic in $\Set^{\to}$ to the function
\[ \coprod_{a,b} \gamma_{D,\kappa_{i,a,b}} : \coprod_{a,b} \colim_{(\alpha,i,\beta)} V(D,Fi(\alpha,\beta)) \to  \coprod_{a,b} V(D,K(a,b)). \]

Let $D$ be small projective and $I$ small. Then the colimit in the definition of $\gamma_{D,\kappa_{i,\alpha,\beta}}$ is preserved since $D$ is small projective. Thus the functions $\gamma_{D,\kappa_{i,\alpha,\beta}}$, and thus $\gamma_{(D),\kappa_i}$ are bijective, whence $(D)$ is also small projective, and so (\ref{lemcase:Dpr-small-projective}) follows. Similar arguments prove (\ref{lemcase:Dpr-lampres}) and (\ref{lemcase:Dpr-connected}). In the case of (\ref{lemcase:Dpr-connected}) when $D$ is connected and $I$ discrete, $(\kappa_{0 \bullet}^{\times 2})^{-1}(a,b)$ is either the empty or the terminal category. In the former case the colimit in the definition of $\gamma_{D,\kappa_{i,\alpha,\beta}}$ is preserved since $D$ is connected, and in the latter case this is so since the colimit in question is absolute. As for (\ref{lemcase:Dpr-lampres}) where $D$ is now $\lambda$-presentable and $I$ is $\lambda$-filtered, the result follows because as we saw in the proof of proposition(\ref{prop:colimits-GV}), the 
categories $(\kappa_{0 \bullet}^{\times 2})^{-1}(a,b)$ are also $\lambda$-filtered.
\end{proof}
\begin{theorem}\label{thm:GV-locally-c-presentable}
Let $\lambda$ be a regular cardinal.
\begin{enumerate}
\item If $V$ is locally $\lambda$-presentable then so is $\ca GV$.\cite{KellyLack-NiceVCat} \label{thmcase:loc-pres}
\item If $V$ is locally $\lambda$-presentable and has a strict initial object, then $\ca GV$ is locally $\lambda$-c-presentable.\label{thmcase:loc-c-pres}
\item If $V$ is locally $\lambda$-presentable then $\ca G^2V$ is locally $\lambda$-c-presentable.\label{thmcase:G2-lcp}
\item If $V$ is locally $\lambda$-c-presentable then so is $\ca GV$.\label{thmcase:lcp}
\item If $V$ is a presheaf topos then so is $\ca GV$.\label{thmcase:presheaf-topos}
\item If $V$ is a Grothendieck topos then $\ca GV$ is a locally connected Grothendieck topos.\label{thmcase:Groth-topos}
\end{enumerate}
\end{theorem}
\begin{proof}
If $V$ is locally $\lambda$-presentable then $\ca GV$ is cocomplete by proposition(\ref{prop:colimits-GV}), and one can build a strong generator for $\ca GV$ consisting of $\lambda$-presentable objects from one in $V$ using lemma(\ref{lem:GV-dense}) to exhibit $\ca GV$ as locally $\lambda$-presentable, giving us (\ref{thmcase:loc-pres}). If in addition $V$ has a strict initial object, then in $\ca GV$ every object decomposes as a sum connected objects by proposition(\ref{prop:colimits-GV})(\ref{pcase:GV-locally-connected}), and so (\ref{thmcase:loc-c-pres}) follows by theorem(\ref{thm:conn-GabUlm})(\ref{lc6}). (\ref{thmcase:G2-lcp}) now follows since $\ca GV$ has a strict initial object by proposition(\ref{prop:colimits-GV}). (\ref{thmcase:lcp}) is immediate from (\ref{thmcase:loc-c-pres}) and theorem(\ref{thm:conn-GabUlm})(\ref{lc7}).

Recall that a category $V$ is a presheaf topos iff it has a small dense full subcategory $i:\ca D \hookrightarrow V$ consisting of small projective objects. Clearly the representables in a presheaf topos provide such a dense subcategory. Conversely $V(i,1):V \to \PSh{\ca D}$ is fully faithful by density. Since the objects of $\ca D$ are small projective, $V(i,1)$ is cocontinuous, and since every presheaf is a colimit of representables, it then follows that $V(i,1)$ essentially surjective on objects, giving the desired equivalence $V \catequiv \PSh{\ca D}$. In this situation $\ca D_+$ provides a small dense subcategory of $\ca GV$ consisting of small projectives by lemma(\ref{lem:GV-dense}), whence $\ca GV \catequiv \PSh{\ca D_+}$, and so (\ref{thmcase:presheaf-topos}) follows.

Since a Grothendieck topos is a left exact localisation of a presheaf category, the 2-functoriality of $\ca G$ together with (\ref{thmcase:presheaf-topos}), (\ref{thmcase:loc-c-pres}) and example(\ref{ex:lc-groth-toposes}) implies that to establish (\ref{thmcase:Groth-topos}), it suffices to show that $\ca G$ preserves left exact functors between categories with finite limits. This follows immediately from the explicit description of limits in $\ca GV$ given in the proof of proposition(\ref{prop:limits-in-GV}).
\end{proof}
In particular from theorem(\ref{thm:GV-locally-c-presentable})(\ref{thmcase:presheaf-topos}) and the 2-functor $(-)_+$, we obtain
\[ \ca G^n\Set \catequiv \PSh {\G_{\leq n}} \]
reconciling the two ways of looking at the category of $n$-globular sets. Note however that this is a genuine equivalence and not an isomorphism.

\section{Constructing a monad from a distributive multitensor}
\label{sec:multitensor-to-monad-construction}
The passage $V \mapsto \ca GV$ discussed in the previous section will now be extended to the construction $(V,E) \mapsto (\ca GV,\Gamma E)$, which takes a category $V$ equipped with a multitensor $E$, and produces a monad $\Gamma E$ on $\ca GV$. The construction itself is very simple and not at all original. What is perhaps novel is the recognition that this construction is so well-behaved formally, and that taking it as fundamental leads to considerable efficiencies in our ability to describe many constructions later on (both in this paper and subsequent works).

A multitensor $E$ and its associated monad $\Gamma E$ describe the same structure, but in different ways. Multitensors, just like the monoidal structures they generalise, come with an attendant notion of enriched category, whereas monads come with a notion of algebra. Proposition(\ref{prop:Gamma-E-algebras-as-E-categories}) says that a category enriched in $E$ is the same thing as an algebra for $\Gamma E$. In the technical aspects of operad/monad theory one is often interested in the formal properties enjoyed by the operads/monads one is considering. Thus it is of interest to know how the formal properties of $E$ correspond those of $\Gamma E$, which is what the main result of this section, theorem(\ref{thm:preservation-by-Gamma-E}), tells us.

\subsection{Recalling multitensors}
\label{sec:recalling-multitensors}
We begin by recalling some definitions and notation from \cite{BataninWeber-EnHop}. For a category $V$, the free strict monoidal category $MV$ on $V$ is described as follows. An object of $MV$ is a finite sequence $(Z_1,...,Z_n)$ of objects of $V$. A map is a sequence of maps of $V$ -- there are no maps between sequences of objects of different lengths. The unit $\eta_V:V{\rightarrow}MV$ of the 2-monad $M$ is the inclusion of sequences of length $1$. The multiplication $\mu_V:M^2V{\rightarrow}MV$ is given by concatenation. A \emph{lax monoidal category} is a lax algebra for the 2-monad $M$, and a \emph{multitensor} on a category $V$ is by definition a lax monoidal structure on $V$. 

Explicitly a multitensor on a category $V$ consists of a functor $E:MV{\rightarrow}V$, and maps
\[ \begin{array}{lcr} {u_Z:Z \rightarrow E(Z)} && {\sigma_{Z_{ij}}:\opE\limits_i\opE\limits_j Z_{ij} \rightarrow \opE\limits_{ij} Z_{ij}} \end{array} \]
for all $Z$, $Z_{ij}$ from $V$ which are natural in their arguments, and such that
\[ \xy (0,0); (10,0):
(0,0)*{\xybox{\xymatrix @C=1em {{\opE\limits_iZ_i} \ar[r]^-{u\opE\limits_i} \ar[d]_{1} \save \POS?(.4)="domeq" \restore & {E_1\opE\limits_iZ_i} \ar[dl]^{\sigma} \save \POS?(.4)="codeq" \restore \\ {\opE\limits_iZ_i} \POS "domeq"; "codeq" **@{}; ?*{=}}}};
(4,0)*{\xybox{\xymatrix @C=1.5em {{\opE\limits_i\opE\limits_j\opE\limits_kZ_{ijk}} \ar[r]^-{{\sigma}\opE\limits_k} \ar[d]_{\opE\limits_i\sigma} \save \POS?="domeq" \restore & {\opE\limits_{ij}\opE\limits_kZ_{ijk}} \ar[d]^{\sigma} \save \POS?="codeq" \restore \\
{\opE\limits_i\opE\limits_{jk}Z_{ijk}} \ar[r]_-{\sigma} & {\opE\limits_{ijk}Z_{ijk}} \POS "domeq"; "codeq" **@{}; ?*{=}}}};
(8,0)*{\xybox{{\xymatrix @C=1em {{\opE\limits_iE_1Z_i} \ar[dr]_{\sigma} \save \POS?(.4)="domeq" \restore & {\opE\limits_iZ_i} \ar[d]^{1} \save \POS?(.4)="codeq" \restore  \ar[l]_-{\opE\limits_iu} \\ & {\opE\limits_iZ_i}} \POS "domeq"; "codeq" **@{}; ?*{=}}}}
\endxy \]
in $V$. As in \cite{BataninWeber-EnHop} the expressions
\[ \begin{array}{lcccr} {E(X_1,...,X_n)} && {\opE\limits_{1{\leq}i{\leq}n} X_i} && {\opE\limits_i X_i} \end{array} \]
are alternative notation for the n-ary tensor product of the objects $X_1,...,X_n$, and we refer to the endofunctor of $V$ obtained by observing the effect of $E$ on singleton sequences as $E_1$. The data $(E,u,\sigma)$ is called a \emph{multitensor} on $V$, and $u$ and $\sigma$ are referred to as the \emph{unit} and \emph{substitution} of the multitensor respectively.

Given a multitensor $(E,u,\sigma)$ on $V$, an \emph{$E$-category} consists of $X \in \ca GV$ together with maps
\[ \kappa_{x_i} : \opE\limits_i X(x_{i-1},x_i) \rightarrow X(x_0,x_n) \]
for all $n \in \N$ and sequences $(x_0,...,x_n)$ of objects of $X$, such that
\[ \xygraph{
{\xybox{\xygraph{!{0;(2,0):(0,.6)::}
{X(x_0,x_1)} (:[r]{E_1X(x_0,x_1)}^-{u} :[d]{X(x_0,x_1)}="bot"^{\kappa}, :"bot"_{\id})}}}
[r(5)][d(.15)]
{\xybox{\xygraph{!{0;(2.75,0):(0,.5)::}
{\opE\limits_i\opE\limits_jX(x_{(ij)-1},x_{ij})} (:[r]{\opE\limits_{ij}X(x_{(ij)-1},x_{ij})}^-{\sigma} :[d]{X(x_0,x_{mn_m})}="bot"^{\kappa},:[d]{\opE\limits_iX(x_{(i1)-1},x_{in_i})}_{\opE\limits_i\kappa} :"bot"_-{\kappa})}}}} \]
commute, where $1{\leq}i{\leq}m$, $1{\leq}j{\leq}n_i$ and $x_{(11)-1}{=}x_0$. Since a choice of $i$ and $j$ references an element of the ordinal $n_{\bullet}$, the predecessor $(ij){-}1$ of the pair $(ij)$ is well-defined when $i$ and $j$ are not both $1$. With the obvious notion of $E$-functor (see \cite{BataninWeber-EnHop}), one has a category $\Enrich E$ of $E$-categories and $E$-functors together with a forgetful functor
\[ U^E : \Enrich E \rightarrow \ca GV. \]

A multitensor $(E,u,\sigma)$ is \emph{distributive} when for all $n$ the functor
\[ \begin{array}{lccr} {V^n \to V} &&& {(X_1,...,X_n) \mapsto E(X_1,...,X_n)} \end{array} \]
preserves coproducts in each variable.

Multitensors generalise non-symmetric operads. For given a non-symmetric operad
\[ \begin{array}{lcr} {(A_n \, \, : \, \, n \in \N)} & {u:I \rightarrow A_1} &
{\sigma:A_k \tensor A_{n_1} \tensor ... \tensor A_{n_k} \rightarrow A_{n_{\bullet}}} \end{array} \]
in a braided monoidal category $V$, one can define a multitensor $E$ on $V$ via the formula
\[ \opE\limits_{1{\leq}i{\leq}n} X_i = A_n \tensor X_1 \tensor ... \tensor X_n \]
as observed in \cite{BataninWeber-EnHop} example(2.6), and when the tensor product for $V$ is distributive, so is $E$. A category enriched in $E$ with one object is precisely an $A$-algebra in the usual sense. In the case where $V$ is $\Set$ with tensor product given by cartesian product, this construction is part of an equivalence between the evident category of distributive multitensors on $\Set$ and that of non-symmetric operads in $\Set$. This equivalence is easily established using the fact that every set is a coproduct of singletons.

\subsection{Monads from multitensors}
\label{sec:monads-from-multitensors}
Let $(E,u,\sigma)$ be a distributive multitensor on a category $V$ with coproducts. Then we define a monad $\Gamma E$ on the category $\ca GV$ of graphs enriched in $V$ as follows. We ask that the monad $\Gamma E$ actually live over $\Set$, that is to say, in the 2-category $\CAT/\Set$. Thus for $X \in \ca GV$, $\Gamma E(X)$ has the same objects as $X$. The homs of $\Gamma E(X)$ are defined by the equation
\begin{equation}\label{eq:monad-from-multitensor-on-homs}
\Gamma E(X)(a,b) = \coprod_{a=x_0,...,x_n=b} \opE\limits_i X(x_{i-1},x_i)
\end{equation}
for all $a,b \in X_0$. The above coproduct is taken over all finite sequences of objects of $X$ starting at $a$ and finishing at $b$. Let us write $k_{E,X,(x_i)_i}$ for a given coproduct inclusion for the above sum. Since the monad we are defining is over $\Set$, the object maps of the components of the unit $\eta$ and multiplication $\mu$ are identities, and so it suffices to define their hom maps. For the unit these are the composites
\[ \xygraph{!{0;(2,0):} {X(a,b)}="l" [r] {E_1X(a,b)}="m" [r(1.25)] {\Gamma E(X)(a,b).}="r" "l":"m"^-{u_{X(a,b)}}:"r"^-{k_{E,X,a,b}}} \]
In order to define the multiplication, observe that the composites
\[ \xygraph{!{0;(3,0):} {\opE\limits_i\opE\limits_j X(x_{ij-1},x_{ij})}="l" [r] {\opE\limits_i \Gamma E(X)(x_{i-1},x_i)}="m" [r] {(\Gamma E)^2(X)(a,b)}="r" "l":"m"^-{\opE\limits_i k}:"r"^-{k}} \]
ranging of all doubly-nested sequences $(x_{ij})_{ij}$ of objects of $X$ starting from $a$ and finishing at $b$, exhibit the hom $(\Gamma E)^2(X)(a,b)$ as a coproduct, since $E$ preserves coproducts in each variable. Let us write $k^{(2)}_{E,X,(x_{ij})_{ij}}$ for such a coproduct inclusion. We can now define the hom map of the components of the multiplication $\mu$ as unique such that
\[ \xygraph{!{0;(2.5,0):(0,.4)::} {\opE\limits_i\opE\limits_j X(x_{ij-1},x_{ij})}="tl" [r] {\opE\limits_{ij} X(x_{ij-1},x_{ij})}="tr" [d] {\Gamma E(X)(a,b)}="br" [l] {(\Gamma E)^2(X)(a,b)}="bl" "tl":"tr"^-{\sigma}:"br"^-{k}:@{<-}"bl"^-{\mu_{X,a,b}}:@{<-}"tl"^-{k^{(2)}}} \]
commutes for all doubly-nested sequences $(x_{ij})_{ij}$ starting at $a$ and finishing at $b$.
\begin{proposition}\label{prop:Gamma-E-algebras-as-E-categories}
Let $V$ be a category with coproducts and $(E,u,\sigma)$ be a distributive multitensor on $V$. Then $(\Gamma E,\eta,\mu)$ as defined above is a monad on $\ca GV$ and one has an isomorphism
$\Enrich{E} \iso (\ca GV)^{\Gamma E}$
commuting with the forgetful functors into $\ca GV$.
\end{proposition}
\begin{proof}
Since $(\Gamma E,\eta,\mu)$ are defined over $\Set$ it suffices to check the monad axioms on the homs. For the unit laws we must verify the commutativity of
\[ \xygraph{{\xybox{\xygraph{!{0;(2.5,0):(0,.4)::} {\Gamma E(X)(a,b)}="l" [r] {(\Gamma E)^2(X)(a,b)}="r" [d] {\Gamma E(X)(a,b)}="b" "l":"r"^-{\eta_{\Gamma E}}:"b"^-{\mu}:@{<-}"l"^-{1}}}} [r(5)]
{\xybox{\xygraph{!{0;(2.5,0):(0,.4)::} {\Gamma E(X)(a,b)}="r" [l] {(\Gamma E)^2(X)(a,b)}="l" [d] {\Gamma E(X)(a,b)}="b" "r":"l"_-{\Gamma E(\eta)}:"b"_-{\mu}:@{<-}"r"_-{1}}}}} \]
and precomposing each of these by each of the injections $k_{E,X,(x_i)_i}$ gives the unit laws for the multitensor $E$. Given a triply-nested sequence of objects of $X$ starting at $a$ and finishing at $b$, let us denote by $k^{(3)}_{E,X,(x_{ijk})_{ijk}}$ the composite
\[ \xygraph{!{0;(3,0):} {\opE\limits_i\opE\limits_j\opE\limits_k X(x_{ijk-1},x_{ijk})}="l" [r(1.2)] {\opE\limits_i(\Gamma E)^2(X)(x_{i-1},x_i)}="m" [r] {(\Gamma E)^3(X)(a,b)}="r" "l":"m"^-{\opE\limits_i k^{(2)}}:"r"^-{k}} \]
and note that since $E$ is distributive, the family of maps so determined exhibits the hom $(\Gamma E)^3(X)(a,b)$ as a coproduct. The associative law on the homs then follows because precomposing the diagrams that express it with such coproduct injections gives back the associativity diagrams for the multitensor $E$. Thus $(\Gamma E,\eta,\mu)$ is indeed a monad on $\ca GV$.

For $X \in \ca GV$ a morphism $a:\Gamma E(X) \to X$ may be identified, by precomposing with the appropriate coproduct inclusions, with morphisms $\opE\limits_i X(x_{i-1},x_i) \rightarrow X(x_0,x_n)$ for all sequences $(x_0,...x_n)$ of objects of $X$, and under this identification the unit and associative laws for a $\Gamma E$-algebra correspond exactly to those for an $E$-category. To say that $f:X \to Y$ in $\ca GV$ underlies a given morphism $(X,a) \to (Y,b)$ of $\Gamma E$-algebras is clearly equivalent to saying that $f$ underlies an $E$-functor. Thus one has the required canonical isomorphism over $\ca GV$.
\end{proof}
\begin{example}\label{ex:cat-monad-from-multitensor}
In the case where $V$ is $\Set$ and $E$ is cartesian product, $\Gamma E$ is the monad for categories on $\Graph$. The summand of equation(\ref{eq:monad-from-multitensor-on-homs}) corresponding to a given sequence $(x_0,...,x_n)$ is the set of paths in $X$ of length $n$, starting at $a=x_0$ and finishing at $b=x_n$, which visits successively the intermediate vertices $(x_1,...,x_{n-1})$.
\end{example}
\begin{remark}\label{rem:EnHopI}
Given a monad $T$ on $\ca GV$ over $\Set$, and a set $Z$, one obtains by restriction a monad $T_Z$ on the category $\ca GV_Z$ of $V$-graphs with fixed object set $Z$. Let us write $\Gamma^{\textnormal{old}}$ for the functor labelled as $\Gamma$ in \cite{BataninWeber-EnHop}. Then for a given distributive multitensor $E$, our present $\Gamma$ and $\Gamma^{\textnormal{old}}$ are related by the formula
\[ \Gamma^{\textnormal{old}}(E) = \Gamma(E)_1 \]
where the $1$ on the right hand side of this equation indicates a singleton. In other words in this paper we are describing the ``many-objects version'' of the theory presented in \cite{BataninWeber-EnHop} section(4).
\end{remark}

\subsection{Properties of $\Gamma E$}
\label{sec:properties-induced-monad}
For a functor
\[ F : A_1 \times ... \times A_n \to B \]
of $n$ variables, the preservation by $F$ of a given connected limit or colimit implies that this limit or colimit is preserved in each variable separately. To see this one considers diagrams which are constant in all but the variable of interest, and use the fact that the limit/colimit of a connected diagram constant at an object $X$, is $X$, as witnessed by a universal cone/cocone all of whose components are $1_X$.

However the converse of this is false in general. For instance to say that $F$ preserves pullbacks is to say that it does so in each variable, and moreover, that all squares of the form
\begin{equation}\label{eq:elementary-pbs-in-cartesian-product}
\xygraph{!{0;(6,0):(0,.1667)::} {(a_1,...a_{i-1},a_i,...,a_j,a_{j+1}...,a_n)}="tl" [r] {(a_1,...a_{i-1},b_i,...,a_j,a_{j+1}...,a_n)}="tr" [d] {(a_1,...a_{i-1},b_i,...,b_j,a_{j+1}...,a_n)}="br" [l] {(a_1,...a_{i-1},a_i,...,b_j,a_{j+1}...,a_n)}="bl" "tl":"tr"^-{(1,...,1,f,1,...,1)}:"br"^-{(1,...,1,g,1,...,1)}:@{<-}"bl"^-{(1,...,1,f,1,...,1)}:@{<-}"tl"^-{(1,...,1,g,1,...,1)}}
\end{equation}
are sent to pullbacks in $B$, where $1 \leq i < j \leq n$, $f:a_i \to b_i$ and $g:a_j \to b_j$. That this extra condition follows from $F$ preserving all pullbacks follows since these squares are obviously pullbacks in $A_1 \times ... \times A_n$. Conversely note that any general map $(f_1,...,f_n):(a_1,...,a_n) \to (b_1,...,b_n)$ in $A_1 \times ... \times A_n$ can be factored in the following manner
\[ \xygraph{!{0;(2.5,0):} {(a_1,...,a_n)}="p1" [r] {(b_1,a_2,...,a_n)}="p2" [r] {...}="p3" [r] {(b_1,...,b_n)}="p4" "p1":"p2"^-{(f_1,1,...,1)}:"p3"^-{(1,f_2,...,1)}:"p4"^-{(1,...,1,f_n)}} \]
and doing so to each of the maps in a general pullback in $A_1 \times ... \times A_n$, produces an $n \times n$ lattice diagram in which each inner square is either of the form (\ref{eq:elementary-pbs-in-cartesian-product}), or a pullback in a single variable.

An important case where such distinctions can be ignored is with $\lambda$-filtered colimits for some regular cardinal $\lambda$. For suppose that $F$ preserves $\lambda$-filtered colimits in each variable. By \cite{AdamekRosicky-LFP} corollary(1.7) it suffices to show that $F$ preserves colimits of chains of length $\lambda$. Given such a chain
\[ \begin{array}{lccr} {X:\lambda \to A_1 \times ... \times A_n} &&& {i \mapsto (X_{i1},...,X_{in})} \end{array} \]
with object map denoted on the right, one obtains the functor
\[ \begin{array}{lccr} {X':\lambda^n \to A_1 \times ... \times A_n} &&& {(i_1,...,i_n) \mapsto (X_{i_11},...,X_{i_nn})} \end{array} \]
which one may readily verify has the same colimit as $X$. But the colimit of $X'$ may be taken one variable at a time and so
\[ \colim(X) \iso \colim_{i_1} \colim_{i_2} ... \colim_{i_n} (X_{i_11},...,X_{i_nn}) \]
from which it follows that $F$ preserves $\colim(X)$. We say that a multitensor $(E,u,\sigma)$ is \emph{$\lambda$-accessible} when the functor $E:MV \to V$ preserves $\lambda$-filtered colimits, which is clearly equivalent to the condition that each of the associated $n$-ary functors $E_n:V^n \to V$ does so, which as we have seen, is equivalent to the condition that each of the $E_n$'s preserve $\lambda$-filtered colimits in each variable.

Cartesian monads play a fundamental role in higher category theory \cite{Leinster-HDA-book}. Recall that a monad $(T,\eta,\mu)$ on a category $V$ with pullbacks is said to be \emph{cartesian} when $T$ preserves pullbacks, and $\eta$ and $\mu$ are cartesian transformations (meaning that their naturality squares are pullbacks). Similarly one has the notion of a cartesian multitensor, with a multitensor $(E,u,\sigma)$ on a category $V$ with pullbacks being \emph{cartesian} when $E$ preserves pullbacks, and $u$ and $\sigma$ are cartesian transformations.

Recall that a functor $F:V \to W$ is a \emph{local right adjoint} (local right adjoint) when for all $X \in V$ the induced functor
\[ F_X : V/X \to V/FX \]
between slice categories is a right adjoint. When $V$ has a terminal object $1$, it suffices for local right adjoint-ness that $F_1$ be a right adjoint. Recall moreover that local right adjoint functors preserve all connected limits, and thus in particular pullbacks. A monad $(T,\eta,\mu)$ on a category $V$ is local right adjoint (as a monad) when $T$ is local right adjoint as a functor and $\eta$ and $\mu$ are cartesian. Thus this is a slightly stronger condition on a monad than being cartesian. Local right adjoint monads, especially defined on presheaf categories, are fundamental to higher category theory. Indeed a deeper understanding of such monads is the key to understanding the relationship between the operadic and homotopical approaches to the subject \cite{Weber-Fam2fun}. Similarly one has the notion of an local right adjoint multitensor, with a multitensor $(E,u,\sigma)$ on a category $V$ being \emph{local right adjoint} when the functor $E:MV \to V$ is local right adjoint, and $u$ and $\sigma$ are cartesian transformations.

For a functor $F:\coprod_{i \in I} V_i \to W$ to be local right adjoint is equivalent to each of the induced $F_i:V_i \to W$ being local right adjoint, because for $X \in V_i$, $F_X = (F_i)_X$. Thus the condition that $E:MV \to V$ be local right adjoint is equivalent to the condition that each of the $E_n$'s is local right adjoint. The condition that a functor $F:V_1 \times ... \times V_n \to W$ to be local right adjoint is equivalent to the condition that it be local right adjoint in each variable, and moreover that it send the basic pullbacks (\ref{eq:elementary-pbs-in-cartesian-product}) in $V_1 \times ... \times V_n$ to pullbacks in $W$. For suppose $F$ is local right adjoint. Then since it preserves all pullbacks it preserves those of the form (\ref{eq:elementary-pbs-in-cartesian-product}). Moreover for $1 \leq i \leq n$ the functor
\[ F(X_1,...,X_{i-1},-,X_{i+1},...,X_n)_{X_i} : V_i/X_i \to W/F(X_1,...,X_n) \]
can be written as the composite
\[ \xygraph{!{0;(3,0):} {V_i/X_i}="l" [r] {V_1/X_1 \times ... \times V_n/X_n}="m" [r(1.5)] {W/F(X_1,...,X_n)}="r" "l":"m"^-{}:"r"^-{F_{X_1,...,X_n}}} \]
in which the first functor has object map $f \mapsto F(1_{X_1},...,1_{X_{i-1}},f,1_{X_{i+1}},...,1_{X_n})$. Since for all $i$ both these functors are clearly right adjoints, $F$ is local right adjoint in each variable. Conversely, supposing $F$ to be local right adjoint in each variable and preserving the pullbacks (\ref{eq:elementary-pbs-in-cartesian-product}), $F$'s effect on the slice over $(X_1,...,X_n)$ is isomorphic to the composite
\[ \xygraph{!{0;(2.65,0):} {\prod_i V/X_i}="l" [r(2)] {\left(W/F(X_1,...,X_n)\right)^n}="m" [r] {W/F(X_1,...,X_n)}="r" "l":"m"^-{\prod_i F(X_1,...,X_{i-1},-,X_{i+1},...,X_n)_{X_i}}:"r"^-{\times}} \]
and both these functors are clearly right adjoints. Thus the condition that $E:MV \to V$ be local right adjoint is equivalent to the condition that each $E_n$ is local right adjoint in each variable and preserve the pullbacks of the form (\ref{eq:elementary-pbs-in-cartesian-product}).

The following result expresses how the assignment $E \mapsto \Gamma E$ is compatible with the various categorical properties we have been discussing.
\begin{theorem}\label{thm:preservation-by-Gamma-E}
Let $V$ be a category with coproducts and $(E,u,\sigma)$ be a distributive multitensor on $V$, and let $(\Gamma E,\eta,\mu)$ be the corresponding monad on $\ca GV$. Let $\lambda$ be a regular cardinal.
\begin{enumerate}
\item $\Gamma E$ preserves coproducts.\label{thmcase:Gamma-E-pres-coproducts}
\item Suppose $V$ has $\lambda$-filtered colimits. Then $E$ is $\lambda$-accessible iff $\Gamma E$ is.\label{thmcase:Gamma-E-pres-filtered-colims}
\item Suppose $V$ has pullbacks and every object of $V$ is a coproduct of connected objects. Then $(E,u,\sigma)$ is a cartesian multitensor iff $(\Gamma E,\eta,\mu)$ is a cartesian monad.\label{thmcase:Gamma-E-cartesian}
\item Suppose $V$ is locally $\lambda$-c-presentable. Then $(E,u,\sigma)$ is an local right adjoint multitensor iff $(\Gamma E,\eta,\mu)$ is an local right adjoint monad.\label{thmcase:Gamma-E-lra}
\end{enumerate}
\end{theorem}
\noindent The proof of this result will occupy the remainder of section(\ref{sec:multitensor-to-monad-construction}).

\subsection{Coproducts and filtered colimits}
\label{sec:GammaE-preserves-coproducts-filtered-colimits}
In lemma(\ref{lem:GammaE-colim-in-terms-of-E}) below we formulate the preservation by $\Gamma E$ of a given colimit in terms of the underlying multitensor $E$. We require some further notation. For a functor $f:J \to \Set$ and $n \in \N$ we denote by $f^{\times n}:J \to \Set$ the functor with object map $j \mapsto f(j)^n$, and if $\kappa_j:fj \to K$ form a colimit cocone, then we denote by $\kappa_{\bullet}^{\times n}:f^{\times n}_{\bullet} \to K^n$ the evident induced functor. We have been using this notation already, for instance in proposition(\ref{prop:colimits-GV}), in the case $n=2$.
\begin{lemma}\label{lem:GammaE-colim-in-terms-of-E}
Let $J$ be a small category, $F:J \to \ca GV$ and $V$ has sufficient colimits so that the colimit $K$ of $F$ may be constructed as in the discussion preceeding proposition(\ref{prop:colimits-GV}). Let $\kappa_{j,0}:F(j)_0 \to K_0$ be a colimit cocone in $\Set$ at the level of objects, and for $a,b \in K_0$ let
\[ \kappa_{j,\alpha,\beta} : F(j)(\alpha,\beta) \to K(a,b) \]
be the colimit cocone in $V$, where $(j,\alpha,\beta) \in (\kappa_{0 \bullet}^{\times 2})^{-1}(a,b)$. If for all sequences $(x_0,...,x_n)$ of objects of $K$, the morphisms
\[ \opE\limits_i \kappa_{j,\gamma_{i-1},\gamma_i} : \opE\limits_i F(j)(\gamma_{i-1},\gamma_i) \to \opE\limits_i K(x_{i-1},x_i) \]
ranging over $(j,\gamma_0,...,\gamma_n) \in (\kappa_{0 \bullet}^{\times n})^{-1}(x_0,...,x_n)$ form a colimit cocone in $V$, then $\Gamma E$ preserves the colimit of $F$.
\end{lemma}
\begin{proof}
The obstruction map $k$ measuring whether $\Gamma E$ preserves the colimit $K$ is bijective on objects since $\Gamma E$ is over $\Set$. By definition of $\Gamma E$ and the construction of colimits in $\ca GV$ one has
\[ \left(\colim\limits_{j \in J} \Gamma E(Fj)\right)(a,b) = \coprod_{a=x_0,...,x_n=b} \colim\limits_{j,\gamma_0,...,\gamma_n} \opE\limits_i F(j)(\gamma_{i-1},\gamma_i) \]
where in the summand $(j,\gamma_0,...,\gamma_n) \in (\kappa_{0 \bullet}^{\times n})^{-1}(x_0,...,x_n)$. Thus if the obstruction maps measuring whether the $\opE\limits_i \kappa_{j,\gamma_{i-1},\gamma_i}$ are colimit cocones are invertible, then the hom maps of $k$ are invertible, and so $k$ is also fully faithful.
\end{proof}
In order to understand how the preservation by $E$ of $\lambda$-filtered colimits gives rise to the same property for $\Gamma E$, we require
\begin{lemma}\label{lem:filtered-finality}
Let $J$ be a filtered category, $F:J \to \Set$ and $\kappa_j:F(j) \to K$ be a colimit cocone. Then for $n > 0$ and $1 \leq i \leq n$ the functor
\[ \begin{array}{lccr} {\tn{pr}_i : (\kappa_{\bullet}^{\times n})^{-1}(x_0,...,x_n) \to (\kappa_{\bullet}^{\times 2})^{-1}(x_{i-1},x_i)} &&& {(j,\gamma_0,...,\gamma_n) \mapsto (j,\gamma_{i-1},\gamma_i)} \end{array} \]
is final.
\end{lemma}
\begin{proof}
For a given $(j,\alpha,\beta) \in (\kappa_{\bullet}^{\times n})^{-1}(x_{i-1},x_i)$ we must show that the comma category $(j,\alpha,\beta)/\tn{pr}_i$ is connected. Explicitly the objects of this comma category consist of the data
\[ \begin{array}{lccr} {f : j \to j'} &&& {(j',\gamma_0,...,\gamma_n)} \end{array} \]
where $\gamma_i \in Fj$, $F(f)(\alpha) = \gamma_{i-1}$ and $F(f)(\beta) = \gamma_i$. A morphism
\[ (f,j',\gamma_0,...,\gamma_n) \to (f',j'',\gamma'_0,...,\gamma'_n) \]
is a map $g:j' \to j''$ in $J$ such that $gf=f'$ and $F(g)(\gamma_k) = \gamma'_k$ for $1 \leq k \leq n$.

For $k \notin \{i-1,i\}$ one can find $(j_k,\gamma_k)$ where $j_k \in J$, $\gamma_k \in F(j_k)$ and $\kappa_{j_k}(\gamma_k) = x_k$ since the cocone $\kappa$ is jointly epic. By the filteredness of $J$ one has maps $\delta:j \to j'$ and $\delta_k:j_k \to j'$, and thus $(\delta,\varepsilon_0,...,\varepsilon_n)$ with $\varepsilon_{i-1}=F(\delta)(\alpha)$, $\varepsilon_{i}=F(\delta)(\beta)$ and $\varepsilon_k=F(\delta)(\gamma_k)$ for $k \notin \{i-1,i\}$, exhibits $(j,\alpha,\beta)/\tn{pr}_i$ as non-empty.

Note that if $y,z \in Fj$ satisfy $\kappa_j(x)=\kappa_j(y)$, then since $K$ may be identified as the connected components of $F_{\bullet}$, there is an undirected path
\[ (j,x) \to (j_1,z_1) \from ... \to (j_n,z_n) \to (j,y) \]
in $F_{\bullet}$. Consider the underlying diagram in $J$ with endpoints (ie the two instances of $j$) identified. Using the filteredness of $J$ one has a cocone for this diagram, and we write $j'$ for the vertex of this cocone. Thus we have $f:j \to j'$ such that $F(f)(y)=F(f)(z)$.

Now let $(f,j',\gamma_0,...,\gamma_n)$ and $(f',j'',\gamma'_0,...,\gamma'_n)$ be any two objects of $(j,\alpha,\beta)/\tn{pr}_i$. First we use the filteredness of $J$ to produce a commutative square
\[ \xygraph{{j}="tl" [r] {j'}="tr" [d] {v_1}="br" [l] {j''}="bl" "tl":"tr"^-{f}:"br"^-{h_1}:@{<-}"bl"^-{g_1}:@{<-}"tl"^-{f'}} \]
whose diagonal we denote as $d_1$. Note that by definition $F(h_1)(\gamma_{i-1}) = F(d_1)(\alpha) = F(g_1)(\gamma'_{i-1})$ and $F(h_1)(\gamma_i) = F(d_1)(\beta) = F(g_1)(\gamma'_i)$, but we have no reason to suppose that $F(h_1)(\gamma_k) = F(g_1)(\gamma'_k)$ for $k \notin \{i-1,i\}$. However $F(h_1)(\gamma_k)$ and $F(g_1)(\gamma_k)$ are by definition identified by $\kappa_{v_1}$. Choosing one value of $k$ and using the observation of the previous paragraph, we can find $r_1:v_1 \to v_2$ such that $F(h_2)(\gamma_k)=F(g_2)(\gamma'_k)$ where $h_2=r_1h_1$ and $g_2=r_1g_1$. Do the same successively for all other $k \notin \{i-1,i\}$, so that in the end one has $h:j' \to v$ and $g:j'' \to v$ such that $hf=gf'$ whose common value we denote as $d$, and $F(h)(\gamma_k)=F(g)(\gamma'_k)$ for all $1 \leq k \leq n$. Denote by $\psi_k \in F(v)$ for the common value of $F(h)(\gamma_k)=F(g)(\gamma'_k)$. Thus one has
\[ \xygraph{!{0;(3,0):} {(f,j',\gamma_0,...,\gamma_n)}="l" [r] {(d,v,\psi_1,...,\psi_k)}="m" [r] {(f',j'',\gamma'_0,...,\gamma'_n)}="r" "l":"m"^-{h}:@{<-}"r"^-{g}} \]
in $(j,\alpha,\beta)/\tn{pr}_i$. Thus $(j,\alpha,\beta)/\tn{pr}_i$ is indeed connected.
\end{proof}
With these preliminary results in hand we can now proceed to
\begin{proof}
(\emph{of theorem(\ref{thm:preservation-by-Gamma-E})(\ref{thmcase:Gamma-E-pres-coproducts}) and (\ref{thmcase:Gamma-E-pres-filtered-colims}}))

(\ref{thmcase:Gamma-E-pres-coproducts}): Let $J$ be small and discrete and $F:J \to \ca GV$. In the situation of lemma(\ref{lem:GammaE-colim-in-terms-of-E}) with a given sequence $(x_0,...,x_n)$ from $K_0$, if that sequence contains elements from different $F(j)$'s then the category $(\kappa_{0 \bullet}^{\times n})^{-1}(x_0,...,x_n)$ will be empty, but by distributivity in this case $\opE\limits_i K(x_{i-1},x_i)$ will also be initial. On the other hand when the $x_i$ all come from the same $F(j)$, one has
\[ (\kappa_{0 \bullet}^{\times n})^{-1}(x_0,...,x_n) = \prod\limits_{1 \leq i \leq n} \kappa_{j0}^{-1}(x_i) \]
and then the universality of the cocone $\opE\limits_i \kappa_{j,\gamma_{i-1},\gamma_i}$ follows again from the distributivity of $E$.

(\ref{thmcase:Gamma-E-pres-filtered-colims}): Suppose $E$ is $\lambda$-accessible. Let $J$ be $\lambda$-filtered, $F:J \to \ca GV$ and $(x_0,...,x_n)$ be a sequence from $K_0$, where as in lemma(\ref{lem:GammaE-colim-in-terms-of-E}), $K$ is the colimit of $F$. Then one has a functor
\[ \begin{array}{lccr} {(\kappa^{\times n}_{0 \bullet})^{-1}(x_0,...,x_n) \to V^n} &&& {(j,\gamma_0,...,\gamma_n) \mapsto \left(F(j)(\gamma_{i-1},\gamma_i)\right)_{1 \leq i \leq n}} \end{array} \]
and we claim that
\[ \left(\kappa_{j,\gamma_{i-1},\gamma_i} : F(j)(\gamma_{i-1},\gamma_i) \to K(x_{i-1},x_i) \,\,\, : \,\,\, 1 \leq i \leq n\right) \]                                                                                                                                                                    
is a colimit cocone in $V^n$ for this functor. In the $i$-th variable $\kappa_{j,\gamma_{i-1},\gamma_i}$ is a cocone for the composite functor
\[ \xygraph{!{0;(2.5,0):} {(\kappa^{\times n}_{0 \bullet})^{-1}(x_0,...,x_n)}="l" [r(1.25)] {(\kappa^{\times 2}_{0 \bullet})^{-1}(x_{i-1},x_i)}="m" [r] {V}="r" "l":"m"^-{\tn{pr}_i}:"r"^-{}} \]
in which the second leg has colimit cocone given by the components $\kappa_{j,\gamma_{i-1},\gamma_i}$. Since $\tn{pr}_i$ is final by lemma(\ref{lem:filtered-finality}), the cocone $(\kappa_{j,\gamma_{i-1},\gamma_i} \,\,\, : \,\,\, 1 \leq i \leq n)$ is indeed universal as claimed.

Now the category $(F^{\times n}_{0 \bullet})$ comes with a discrete opfibration into $J$, and so its connected components are $\lambda$-filtered. But since $\lambda$-filtered colimits commute with finite products in $\Set$, these connected components are exactly the fibres of $(\kappa^{\times n}_{0 \bullet})$, and so for each sequence $(x_0,...,x_n)$, $(\kappa^{\times n}_{0 \bullet})^{-1}(x_0,...,x_n)$ is $\lambda$-filtered. Thus by lemma(\ref{lem:GammaE-colim-in-terms-of-E}) $\Gamma E$ is $\lambda$-accessible.

Conversely suppose that $\Gamma E$ is $\lambda$-accessible. For $F:J \to V$ with $J$ where is $\lambda$-filtered, with colimit cocone $\kappa_j:Fj \to K$ we must show that the induced cocone
\begin{equation}\label{eq:E-of-filtered-cocone}
E(X_1,...,X_{i-1},Fj,X_{i+1},...X_n) \to E(X_1,...,X_{i-1},K,X_{i+1},...X_n)
\end{equation}
is universal, for all $N \in \N$, $1 \leq i \leq n$ and $X_1,...,X_{i-1},X_{i+1},...X_n \in V$. By remark(\ref{rem:fibrewise-limits-and-colimits}) the cocone
\begin{equation}\label{eq:filtered-cocone-in-GV}
(X_1,...,X_{i-1},Fj,X_{i+1},...X_n) \to (X_1,...,X_{i-1},K,X_{i+1},...X_n) 
\end{equation}
in $\ca GV$ is universal, and moreover that for any sequence $(Y_1,...,Y_n)$ of objects of $V$ and $1 \leq a,b \leq n$ one has
\begin{eqnarray*}
\Gamma E(Y_1,...,Y_n)(0,n) & = & \coprod\limits_{0=x_0,...,x_n=n} \opE\limits_i \left((Y_1,...,Y_n)(x_{i-1},x_i)\right) \\
& \iso & \opE\limits_i Y_i
\end{eqnarray*}
by the distributivity of $E$. Thus applying $\Gamma E$ to the cocone (\ref{eq:filtered-cocone-in-GV}) and looking at the hom between $0$ and $n$ gives the cocone (\ref{eq:E-of-filtered-cocone}), and so by remark(\ref{rem:fibrewise-limits-and-colimits}), the result follows.
\end{proof}

\subsection{Cartesianness of $\Gamma E$}
\label{sec:GammaE-cartesian}
Let $V$ be a category with coproducts and pullbacks, in which every object is a coproduct of connected objects, and suppose that $(E,u,\sigma)$ is a cartesian multitensor. We will now show that $(\Gamma E,\eta,\mu)$ is a cartesian monad. Note that by lemma(\ref{lem:easy-ext}) such a $V$ is in fact extensive.
\begin{lemma}\label{lem:pb-decomp}
Let $V$ be a category with coproducts and pullbacks in which every object is a coproduct of connected objects. Suppose that we are given square
\[ \xygraph{{P}="tl" [r] {B}="tr" [d] {C}="br" [l] {A}="bl" "tl":"tr"^-{q}:"br"^-{g}:@{<-}"bl"^-{f}:@{<-}"tl"^-{p}} \]
in $V$ which admits a description as on the left in
\[ \xygraph{!{0;(2.5,0):}
{\xybox{\xygraph{!{0;(2,0):(0,.5)::} {\coprod\limits_{(i,j) \in L} P_{ij}}="tl" [r] {\coprod\limits_{j \in J} B_j}="tr" [d] {\coprod\limits_{k \in K} C}="br" [l] {\coprod\limits_{i \in I} A_i}="bl" "tl":"tr"^-{(q_{ij})_{ij}}:"br"^-{(g_j)_j}:@{<-}"bl"^-{(f_i)_i}:@{<-}"tl"^-{(p_{ij})_{ij}}}}} [r(1.2)]
{\xybox{\xygraph{{L}="tl" [r] {J}="tr" [d] {K}="br" [l] {I}="bl" "tl":"tr"^-{\nu}:"br"^-{\gamma}:@{<-}"bl"^-{\phi}:@{<-}"tl"^-{\pi} "tl":@{}"br"|-{pb}}}} [r]
{\xybox{\xygraph{!{0;(1.5,0):(0,.667)::} {P_{(i,j)}}="tl" [r] {B_j}="tr" [d] {C_{\phi i = \gamma j}}="br" [l] {A_i}="bl" "tl":"tr"^-{q_{ij}}:"br"^-{g_j}:@{<-}"bl"^-{f_i}:@{<-}"tl"^-{p_{ij}}}}}} \]
in which the indexing sets of the coproduct decompositions fit into a pullback square as shown in the middle, with elements of $L$ represented explicitly as pairs $(i,j)$ such that $\phi(i)=\gamma(j)$. Suppose moreover that for all such $(i,j)$ the squares as indicated on the right in the previous display are pullbacks. Then it follows that the original square is itself a pullback.
\end{lemma}
\begin{proof}
To see this is a pullback it suffices just for connected $X$, $h:X \to A$ and $k:X \to B$ with $fh=gk$, that there is a unique filler $d:X \to P$ such that $pd=h$ and $qd=k$, since every object of $V$ is a coproduct of connected ones. But then using the connectedness one can factor $h$ and $k$ through unique summands say $i \in I$ and $j \in J$ related by $\phi(i)=\gamma(j)$, and so use the defining pullback of $P_{ij}$ to induce the desired unique $d$.
\end{proof}
One application of lemma(\ref{lem:pb-decomp}) is the componentwise construction of pullbacks in such a $V$. For given a cospan
\[ \xygraph{{A}="l" [r] {C}="m" [r] {B}="r" "l":"m"^-{f}:@{<-}"r"^-{g}} \]
in $V$, one can compute its pullback one component at a time by decomposing $A$, $B$ and $C$ into coproducts of connected objects, then pulling back the indexing sets, then taking the pullbacks componentwise, and finally re-amalgamating (by taking coproducts). Note however that the summands $P_{ij}$ of the pullback so obtained are not neccessarily themselves connected. We are now ready to exhibit
\begin{proof}
(\emph{of theorem(\ref{thm:preservation-by-Gamma-E})(\ref{thmcase:Gamma-E-cartesian})})

Let $(E,u,\sigma)$ a cartesian multitensor on $V$ a category with coproducts and pullbacks in which every object decomposes as sum of connected ones. Let $\mathbb{P}$ be the pullback
\[ \xygraph{{P}="tl" [r] {Y}="tr" [d] {Z}="br" [l] {X}="bl" "tl":"tr"^-{q}:"br"^-{g}:@{<-}"bl"^-{f}:@{<-}"tl"^-{p} "tl":@{}"br"|-{pb}} \]
in $\ca GV$ and denote by $d:P \to Z$ the diagonal. Then $\Gamma E(\mathbb{P})$ is certainly a pullback at the level of object sets, since $\Gamma E$ is over $\Set$. So it suffices, by the construction of pullbacks in $\ca GV$, to check that for each $w,w' \in P_0$ the corresponding hom square of $\Gamma E(\mathbb{P})$ is a pullback in $V$. This hom square is a square in $V$ of the form
\[ \xygraph{!{0;(5,0):(0,.25)::} {\coprod\limits_{w=w_0,...,w_n=w'} \opE\limits_i P(w_{i-1},w_i)}="tl" [r] {\coprod\limits_{qw=y_0,...,y_n=qw'} \opE\limits_i Y(y_{i-1},y_i)}="tr" [d] {\coprod\limits_{dw=z_0,...,z_n=dw'} \opE\limits_i Z(z_{i-1},z_i)}="br" [l] {\coprod\limits_{pw=x_0,...,x_n=pw'} \opE\limits_i X(x_{i-1},x_i)}="bl" "tl":"tr"^-{}:"br"^-{}:@{<-}"bl"^-{}:@{<-}"tl"^-{}} \]
and the induced square at the level of summand indexing sets is a pullback since $\mathbb{P}_0$ is a pullback in $\Set$. For each sequence $(w_0,...,w_n)$ in $P_0$ from $w$ to $w'$, the corresponding component is
\[ \xygraph{!{0;(4,0):(0,.25)::} {\opE\limits_i P(w_{i-1},w_i)}="tl" [r] {\opE\limits_i Y(qw_{i-1},qw_i)}="tr" [d] {\opE\limits_i Z(dw_{i-1},dw_i)}="br" [l] {\opE\limits_i X(pw_{i-1},pw_i)}="bl" "tl":"tr"^-{\opE\limits_i q_{w_{i-1},w_i}}:"br"^-{\opE\limits_i g_{qw_{i-1},qw_i}}:@{<-}"bl"^-{\opE\limits_i f_{pw_{i-1},pw_i}}:@{<-}"tl"^-{\opE\limits_i p_{w_{i-1},w_i}}} \]
which is a pullback since $\mathbb{P}$ is. Thus by lemma(\ref{lem:pb-decomp}) $\Gamma E(\mathbb{P})$ is a pullback.

We must show that for $f:X \to Y$ in $\ca GV$ the corresponding naturality squares of $\eta$ and $\mu$ are cartesian. Since they are over $\Set$ this is clearly so at the level of objects. The hom at $(a,b)$ of the naturality of square of $\eta$ has underlying square of summand indexing sets given by
\[ \xygraph{!{0;(4,0):(0,.25)::} {1}="tl" [r] {\{(x_0,...,x_n) \,\, : \,\, n \in \N, \, x_0=a, \, x_n=b\}}="tr" [d] {\{(y_0,...,y_n) \,\, : \,\, n \in \N, \, y_0=fa, \, y_n=fb\}}="br" [l] {1}="bl" "tl":"tr"^-{(a,b)}:"br"^-{\tn{apply $f_0$}}:@{<-}"bl"^-{(fa,fb)}:@{<-}"tl"^-{}} \]
and the components are naturality squares for $u$. Thus by lemma(\ref{lem:pb-decomp}) $\eta$ is cartesian. Note that using the distributivity of $E$ one has a canonical isomorphism
\[ (\Gamma E)^2(X)(a,b) \iso \coprod\limits_{(x_{ij})_{ij}}  \opE\limits_i\opE\limits_j X(x_{ij-1},x_{ij}) \]
where the coproduct is taken over the set of composable doubly-indexed sequences starting at $a$ and finishing at $b$.
Unpacking in these terms one can see that in the case of $\mu$'s hom naturality square, the underlying square of summand indexing sets is
\[ \xygraph{!{0;(5.5,0):(0,.1667)::} {\{(x_{ij})_{ij} \,\, : \,\, x_0=a, \, x_n=b\}}="tl" [r] {\{(x_0,...,x_n) \,\, : \,\, n \in \N, \, x_0=a, \, x_n=b\}}="tr" [d] {\{(y_0,...,y_n) \,\, : \,\, n \in \N, \, y_0=fa, \, y_n=fb\}}="br" [l] {\{(y_{ij})_{ij} \,\, : \,\, y_0=fa, \, y_n=fb\}}="bl" "tl":"tr"^-{\tn{concatenate}}:"br"^-{\tn{apply $f_0$}}:@{<-}"bl"^-{\tn{concatenate}}:@{<-}"tl"^-{\tn{apply $f_0$}}} \]
in which concatenation is that of \emph{composable} sequences, that is, one identifies the last point of the $i$-th subsequence with the first point of the $(i+1)$-th, which by definition of ``composable doubly-indexed sequence'' are equal as elements of $X_0$ or $Y_0$. This square is easily seen to be a pullback. The components of $\mu$'s hom naturality square are naturality squares for $\sigma$. Thus by lemma(\ref{lem:pb-decomp}) $\mu$ is cartesian.

Conversely suppose that $(\Gamma,\eta,\mu)$ is a cartesian monad. Then by the same argument as for the converse direction of (\ref{thmcase:Gamma-E-pres-filtered-colims}), except with pullbacks in place of $\lambda$-filtered colimits, one may conclude that $E$ preserves pullbacks. Note that for $X \in V$ the hom between $0$ and $1$ of the naturality component of $\eta_{(X)}$ is, modulo the canonical isomorphism $E_1X \iso \Gamma E(X)(0,1)$, just $E_1X$, and so $u$'s cartesianness follows from that of $\eta$ by remark(\ref{rem:fibrewise-limits-and-colimits}). Suppose that $(X_1,...,X_n)$ is a sequence of objects of $V$. Denote by $\tn{sd}(X_i)_i$ the set of subdivisions of $(X_i)_i$ into a sequence of sequences. A typical element is a sequence of sequences $(X_{ij})$ where $1 \leq i \leq k$, $1 \leq j \leq n_i$ and $n_1 + ... n_k = n$, such that sequence obtained by concatenation is $(X_1,...,X_n)$. Then modulo the canonical isomorphism
\[ (\Gamma E)^2(X_1,...,X_n) \iso \coprod\limits_{\tn{sd}(X_i)_i} \opE\limits_i\opE\limits_j X_{ij} \]
the hom of the naturality component of $\mu_{(X_1,...,X_n)}$ between $0$ and $n$ is the map
\[ (\sigma_{X_{ij}}) : \coprod\limits_{\tn{sd}(X_i)_i} \opE\limits_i\opE\limits_j X_{ij} \to \opE_i X_i \]
and thus by remark(\ref{rem:fibrewise-limits-and-colimits}), these maps are cartesian natural in the $X_i$. By lemma(\ref{lem:easy-ext}) $V$ is extensive, and so the $\sigma_{X_{ij}}$ are cartesian natural in the $X_{ij}$ as required.
\end{proof}

\subsection{Local right adjointness}
\label{sec:GammaE-lra}
We now proceed to the task of proving that the construction $\Gamma$ is compatible with local right adjoint-ness. For this we first require two lemmas. We assume familiarity with the notion of ``generic morphism'' and the alternative formulation of local right adjoint-ness in terms of generics as described in \cite{Weber-Fam2fun} proposition(2.6).
\begin{lemma}\label{lem:partial-adjoint}
Let $R:V{\rightarrow}W$ be a functor, $V$ be cocomplete, $U$ be a small dense full subcategory of $W$, and $L:U{\rightarrow}V$ be a partial left adjoint to $R$, that is to say, one has isomorphisms $W(S,RX) \iso V(LS,X)$ natural in $S \in U$ and $X \in V$. Defining $\overline{L}:W{\rightarrow}V$ as the left kan extension of $L$ along the inclusion $I:U{\rightarrow}W$, one has $\overline{L} \ladj R$.
\end{lemma}
\begin{proof}
Denoting by $p:I/Y{\rightarrow}U$ the canonical forgetful functor for $Y \in W$ and recalling that $\overline{L}Y = \colim(Lp)$, one obtains the desired natural isomorphism as follows
\[ \begin{array}{rcccl} {V(\overline{L}Y,X)} & {\iso} & {[I/Y,V](Lp,\textnormal{const}(X))} & {\iso} & {\lim_{f{\in}I/Y} V(L(\textnormal{dom}(f)),X)} \\ & {\iso} & {\lim_f W(\textnormal{dom}(f),RX)} & {\iso} &  {\ca B(Y,RX)} \end{array} \]
for all $X \in V$.
\end{proof}
\begin{lemma}\label{lem:lra-dense}
Let $T:V{\rightarrow}W$ be a functor, $V$ be cocomplete and $W$ have a small dense subcategory $U$. Then $T$ is a local right adjoint iff every $f:S{\rightarrow}TX$ with $S \in U$ admits a generic factorisation. If in addition $V$ has a terminal object denoted $1$, then generic factorisations in the case $X=1$ suffice.
\end{lemma}
\begin{proof}
For the first statement ($\implies$) is true by definition so it suffices to prove the converse. The given generic factorisations provide a partial left adjoint $L:I/TX{\rightarrow}V$ to $T_X:V/X{\rightarrow}W/TX$ where $I$ is the inclusion of $U$. Now $I/TX$ is a small dense subcategory of $W/TX$, and so by the previous lemma $L$ extends to a genuine left adjoint to $T_X$. In the case where $V$ has $1$ one requires only generic factorisations in the case $X=1$ by the results of \cite{Weber-Fam2fun} section(2).
\end{proof}
\noindent The analogous result for presheaf categories, with the representables forming the chosen small dense subcategory, was discussed in \cite{Weber-Fam2fun} section(2). With these results in hand we may now exhibit the
\begin{proof}
(\emph{of theorem(\ref{thm:preservation-by-Gamma-E})(\ref{thmcase:Gamma-E-lra})})

The aspects of this result involving the cartesianness of the units, multiplication and substitution are covered already by (\ref{thmcase:Gamma-E-cartesian}). Suppose that $E$ is local right adjoint. Let $\ca D$ be a small dense subcategory of $V$ consisting of $\lambda$-presentable connected objects. By lemma(\ref{lem:lra-dense}) and lemma(\ref{lem:GV-dense}) it suffices to exhibit generic factorisations of maps
\[ f:S \rightarrow \Gamma{E}1 \]
where $S$ is either $0$ or $(D)$ for some $D \in \ca D$. In the case where $S$ is $0$ the first arrow in the composite
\[ \xymatrix{0 \ar[r] & {\Gamma{E}0} \ar[r]^-{\Gamma{E}t} & {\Gamma{E}1}} \]
is generic because $0$ is the initial $V$-graph with one object (and $t$ here is the unique map). In the case where $S=(D)$, to give $f$ is to give a map $f':D{\rightarrow}E_n1$ in $V$ since $D$ is connected. Since $E$ is a local right adjoint, $E_n$ is too and so one can generically factor $f'$ to obtain
\[ \xymatrix{D \ar[r]^-{g'_f} & {\opE\limits_iZ_i} \ar[r]^-{\opE\limits_it} & {E_n1}} \]
from which we obtain the generic factorisation
\[ \xymatrix{{(D)} \ar[r]^-{g_f} & {\Gamma{E}Z} \ar[r]^-{\Gamma{E}t} & {\Gamma{E}1}} \]
where $Z=(Z_1,...,Z_n)$, the object map of $g_f$ is given by $0 \mapsto 0$ and $1 \mapsto n$, and the hom map of $g_f$ is $g'_f$ composed with the coproduct inclusion.

Conversely suppose that $\Gamma E$ is local right adjoint. It suffices by lemma(\ref{lem:lra-dense}) to exhibit a generic factorisation for maps of the form on the left in
\[ \begin{array}{lccr} {f : Y \to E(X_1,...X_n)} &&& {f' : (Y) \to \Gamma E(X_1,...,X_n)} \end{array} \]
where $Y$ is connected. Such an $f$ determines $f'$ as in the previous display unique with object map $(0,1) \mapsto (0,n)$ and hom map between $0$ and $1$ given by $f$, modulo the canonical isomorphism $E(X_1,...,X_n) \iso \Gamma E(X_1,...,X_n)(0,n)$ that we described already in the proof of (\ref{thmcase:Gamma-E-pres-filtered-colims}).

Consider a factorisation
\[ \xygraph{!{0;(2,0):} {(Y)}="l" [r(.75)] {\Gamma EZ}="m" [r(1.25)] {\Gamma E(X_1,...,X_n)}="r" "l":"m"^-{g}:"r"^-{\Gamma Eh}} \]
of $f'$. The object map of $h$ partitions the objects of $Z$ into $n+1$ subsets $Z_{(0)},...,Z_{(n)}$. The strict initiality of $\emptyset$ and the definition of $(X_1,...,X_n) \in \ca GV$ ensures that the only homs of $Z$ that are possibly non-initial, are those between $a$ and $b$ living in consecutive cells of this partition. Thus in addition to this partition $h$ amounts to maps $h_{a,b} : Z(a,b) \to X_i$ for all $a \in Z_{(i-1)}$ and $b \in Z_{(i)}$. The connectedness of $Y$ ensures that the hom map of $g$ between $0$ and $1$ factors through a unique summand of the appropriate hom of $\Gamma EZ$. Thus the data of $g$ comes down to: $1 \leq i \leq j \leq n$, $c_r \in Z_{(r)}$ for $i \leq r \leq j$ and a map $g_{0,1} : Y \to \opE\limits_{i{<}r{\leq}j} Z(c_{r-1},c_r)$. Consider the canonical inclusion
\[ \overline{c} : (Z(c_i,c_{i+1}), ..., Z(c_{j-1},c_j)) \hookrightarrow Z \]
and note that by the above description one may factor $g$ as
\[ \xygraph{!{0;(3,0):} {(Y)}="l" [r] {\Gamma E(Z(c_i,c_{i+1}), ..., Z(c_{j-1},c_j))}="m" [r] {\Gamma E Z.}="r" "l":"m"^-{g'}:"r"^-{\Gamma E \overline{c}}} \]
If $g$ were in fact generic then $\overline{c}$ would have a section and thus be an isomorphism. It follows that any generic factorisation of $f'$ is necessarily of the form
\[ \xygraph{!{0;(2,0):} {(Y)}="l" [r] {\Gamma E(X_1',...,X_n')}="m" [r(2)] {\Gamma E(X_1,...,X_n)}="r" "l":"m"^-{g}:"r"^-{\Gamma E(h_1,...,h_n)}} \]
for $h_i:X_i' \to X_i$ in $V$. Moreover it is easily shown that the hom of this factorisation between $0$ and $1$ gives a generic factorisation for the original map $f$, thereby exhibiting $E$ as local right adjoint.
\end{proof}

\section{Constructing a multitensor from a path-like monad}
\label{sec:Monads-Operads-Multitensors}

The passage $(V,E) \mapsto (\ca GV,\Gamma E)$ that we studied in the previous section is really the object map of a 2-functor 
\[ \Gamma : \DISTMULT \rightarrow \MND(\CAT/\Set). \]
In fact there are two (dual) ways of exhibiting $\Gamma$ as being 2-functorial. It is these 2-functors that are the principal objects of study in this section.  The 2-functoriality is given in section(\ref{sec:2-functoriality-Gamma}). In theorem(\ref{thm:characterisation-of-image-of-Gamma}) we characterise monads of the form $\Gamma E$, and propositions(\ref{prop:Gamma-locally-ff}) and (\ref{prop:Psi-2-ff}) essentially characterise the the images of the one and 2-cell maps of $\Gamma$. Finally in section(\ref{sec:compatibility-of-2-functor-Gamma}) we witness the compatibility of $\Gamma$ with cartesian transformations, which will lead in section(\ref{sec:operads-multitensors-basic}), to an understanding of the relation between multitensors and operads.

\subsection{Constructing a multitensor from a monad}
\label{ssec:monads-to-multitensors}
For a category $V$ a monad $(T,\eta,\mu)$ \emph{over $\Set$ on $\ca GV$} is a monad on
\[ (-)_0 : \ca GV \rightarrow \Set \]
in the 2-category $\CAT/\Set$. Thus as explained in section(\ref{sec:monads-from-multitensors}), the functor $T$ doesn't affect the object sets and similarly for maps, and moreover the components of $\eta$ and $\mu$ are identities on objects. Recall from section(\ref{ssec:enriched-graphs}) that if $V$ has an initial object then one can regard any sequence of objects $(Z_1,...,Z_n)$ of $V$ as a $V$-graph. This is clearly functorial in the $Z_i$. Moreover for $1 \leq a \leq b \leq n$ one has subsequence inclusions
\[ (Z_a,...,Z_b) \hookrightarrow (Z_1,...,Z_n) \]
defined in the obvious way, with the object map preserving successor and $0 \mapsto (a-1)$, and the hom maps being identities. This enables us to construct a multitensor on $V$ from $T$, essentially by applying $T$ to sequences and looking at the homs.

Explicitly one defines this associated multitensor $(\overline{T},\overline{\eta},\overline{\mu})$ as follows. The $n$-ary tensor product is defined by
\[ \overline{T}(Z_1,...,Z_n) := T(Z_1,...,Z_n)(0,n). \]
Recall that for $Z \in V$, $(Z)$ is the $V$-graph with object set $\{0,1\}$, hom between $0$ and $1$ equal to $Z$, and other homs initial. The unit $\overline{\eta}_Z:Z \to \overline{T}_1 Z$ is the hom map of $\eta_{(Z)}$ between $0$ and $1$. In order to define the substitution, note that given objects $Z_{ij}$ of $V$ where $1 \leq i \leq k$ and $1 \leq j \leq n_i$, one has a map
\[ \tilde{\tau}_{Z_{ij}} : (\Tbar\limits_{1{\leq}j{\leq}n_1}Z_{1j},...,\Tbar\limits_{1{\leq}j{\leq}n_k}Z_{kj}) \rightarrow T(Z_{11},......,Z_{kn_k}) \]
given on objects by $0 \mapsto 0$ and $i \mapsto (i,n_i)$ for $1 \leq i \leq k$. The hom map between $(i-1)$ and $i$ is the hom map of
\[ Ts_i : T(Z_{i1},...,Z_{in_i}) \to T(Z_{11},...,Z_{kn_k})  \]
between $0$ and $n_i$, where $s_i$ is the $i$-th subsequence inclusion. The component $\overline{\mu}_{Z_{ij}}$ is defined to be the hom map of $\mu \comp T(\tilde{\tau}_{Z_{ij}})$ between $0$ and $k$.

In order to understand why $(\overline{T},\overline{\eta},\overline{\mu})$ form a multitensor, it is worthwhile to take a more conceptual approach. This begins with the observation that a sequence $(Z_1,...,Z_n)$ of objects of $V$ may be viewed as a cospan
\[ \xygraph{!{0;(2,0):} {0}="l" [r] {(X_1,...,X_n)}="m" [r] {0}="r" "l":"m"^-{b}:@{<-}"r"^-{t}} \]
in $\ca GV$ in which $b$ picks out the ``bottom'' object $0$ and $t$ picks out the ``top'' object $n$. Moreover pushout composition
\[ \xygraph{!{0;(1.5,0):(0,.667)::} {0}="tl" [r(2)] {0}="tm" [r(2)] {0}="tr" [dl] {(Z_1,...,Z_n)}="mr" [l(2)] {(Y_1,...,Y_m)}="ml" [dr] {(Y_1,...,Y_m,Z_1,...,Z_n)}="b" "tl":"ml"^-{b}:@{<-}"tm"^-{t}:"mr"^-{b}:@{<-}"tr"^-{t} "ml":"b":@{<-}"mr" "ml":@{}"mr"|-{po}} \]
of such cospans in $\ca GV$ corresponds, as shown, to concatenation of sequences. These pushouts are special in that they only require an initial object in $V$ for their construction.

Pushout composition in $\ca GV$ of general cospans of the form
\[ \xygraph{{0}="l" [r] {X}="m" [r] {0}="r" "l":"m"^-{}:@{<-}"r"^-{}} \]
require coproducts in $V$ for their construction. Note that such cospans are, as already pointed out in section(\ref{ssec:enriched-graphs}), nothing more than bipointed $V$-graphs. Thus when $V$ has coproducts, pushout composition of cospans endows the category $\ca G(t_V)^{\times 2}_{\bullet}$ of bipointed $V$-graphs with a monoidal structure whose tensor product we denote as ``$*$''. Moreover given a monad $(T,\eta,\mu)$ on $\ca GV$ over $\Set$, one obtains a monoidal monad $(T_{\bullet},\eta_{\bullet},\mu_{\bullet})$ on $\ca G(t_V)^{\times 2}_{\bullet}$. The underlying endofunctor
\[ \begin{array}{lccr} {T_{\bullet} : \ca G(t_V)^{\times 2}_{\bullet} \to \ca G(t_V)^{\times 2}_{\bullet}} &&& {(X,a,b) \mapsto (TX,a,b)} \end{array} \]
as object map as described in the previous display. In terms of cospans, this is just the application of $T$ to cospans plus composition with the unique identity-on-objects $0 \to T0$ in order to get an endocospan of $0$. The monoidal functor coherences for $T_{\bullet}$ are the maps that give the obstruction to $T$ preserving the pushouts involved. The data $(\eta_{\bullet},\mu_{\bullet})$ are defined in the evident way from $(\eta,\mu)$.

The assignation of a cospan/bipointed $V$-graph from a sequence may done in two steps
\[ (Z_1,...,Z_n) \in MV \mapsto (((Z_1),0,1),...,((Z_n),0,1)) \mapsto ((Z_1,...,Z_n),0,n) \]
and so is the object map of the composite
\[ \xygraph{!{0;(2,0):} {MV}="p1" [r] {M\ca G(t_V)^{\times 2}_{\bullet}}="p2" [r] {\ca G(t_V)^{\times 2}_{\bullet}.}="p3" "p1":"p2"^-{ML_V}:"p3"^-{*}} \]
Thus one can view the functor $\overline{T}:MV \to V$ in more conceptual terms as the composite
\[ \xygraph{!{0;(2,0):} {MV}="p1" [r] {M\ca G(t_V)^{\times 2}_{\bullet}}="p2" [r] {\ca G(t_V)^{\times 2}_{\bullet}}="p3" [r] {\ca G(t_V)^{\times 2}_{\bullet}}="p4" [r] {V.}="p5" "p1":"p2"^-{ML_V}:"p3"^-{*}:"p4"^-{T_{\bullet}}:"p5"^-{\varepsilon_V}} \]
Observe that $(T_{\bullet},\eta_{\bullet},\mu_{\bullet})$ is a monoidal monad and $L_V \ladj \varepsilon_V$. Moreover in general one has
\begin{lemma}\label{lem:general-multitensor-facts-for-Tbar}
\begin{enumerate}
\item Let $(E,u,\sigma)$ be a multitensor on $V$ and $(T,\eta,\mu)$ be a monoidal monad on $(V,E)$ with monoidal functor coherences for $T$ written as
\[ \tau_{X_i} : \opE\limits_i TX_i \to T\opE\limits_i X_i. \]
Then $(F,u',\sigma')$ defines another multitensor on $V$ where $\opF\limits_i X_i = T\opE\limits_i X_i$, $u_X' = \eta_{E_1X}u_X$ and $\sigma'$ is the composite
\[ \xygraph{!{0;(2,0):} {T\opE\limits_iT\opE\limits_j}="l" [r] {T^2\opE\limits_i\opE\limits_j}="m" [r] {T\opE\limits_{ij}.}="r" "l":"m"^-{T \tau E}:"r"^-{\mu \sigma}} \]
\label{lemcase:monoidal-monad->multitensor}
\item Let $(E,u,\sigma)$ be a multitensor on $V$ and $L \ladj R : V \to W$ with unit $\eta$ and counit $\varepsilon$. Then $(F,u',\sigma')$ defines multitensor on $W$ where $\opF\limits_i X_i = R \opE\limits_i LX_i$, $u'=(RuM)\eta$ and $\sigma'$ is the composite
\[ \xygraph{!{0;(2.5,0):} {R\opE\limits_iLR\opE\limits_jL}="l" [r] {R\opE\limits_i\opE\limits_jL}="m" [r] {R\opE\limits_{ij}L}="r" "l":"m"^-{R\opE\limits_i\varepsilon\opE\limits_jL}:"r"^-{R \sigma L}} \]
\label{lemcase:multitensor-across-adjunction}
\end{enumerate}
\end{lemma}
\noindent whose proof is an easy exercise in the definitions involved. Starting with the monoidal structure $*$ on $\ca G(t_V)^{\times 2}_{\bullet}$, apply (\ref{lemcase:monoidal-monad->multitensor}) to obtain the multitensor $T_{\bullet}*$ on $\ca G(t_V)^{\times 2}_{\bullet}$, and then apply (\ref{lemcase:multitensor-across-adjunction}) to this using the adjunction $L_V \ladj \varepsilon_V$. It is straight forward to verify directly that the unit and substitution of the resulting multitensor coincides with $(\overline{\eta},\overline{\mu})$ as defined above. Thus we have
\begin{proposition}\label{prop:Tbar}
Let $V$ be a category with coproducts and $(T,\eta,\mu)$ be a monad on $\ca GV$ over $\Set$. Then $(\overline{T},\overline{\eta},\overline{\mu})$ defines a multitensor on $V$.
\end{proposition}
\begin{remark}\label{rem:implicit-Tbar}
Note that the multitensor $(\overline{T},\overline{\eta},\overline{\mu})$ played an implicit role the proofs of the converse implications of theorem(\ref{thm:preservation-by-Gamma-E})(\ref{thmcase:Gamma-E-pres-filtered-colims})-(\ref{thmcase:Gamma-E-lra}). The reason for this is that one has a canonical isomorphism $E \iso \overline{\Gamma E}$ of multitensors. The isomorphism at the level of functors $MV \to V$ was described in the proof of theorem(\ref{thm:preservation-by-Gamma-E})(\ref{thmcase:Gamma-E-pres-filtered-colims}), and the reader will easily verify the compatibility of this isomorphism with the unit and substitution maps.
\end{remark}

\subsection{Characterisation of monads coming from multitensors}
\label{sec:charn-monads-from-multitensors}
In this section we characterise the monads of the form $(\ca GV,\Gamma E)$ as those monads $T$ on $\ca GV$ over $\Set$ which are distributive and path-like in the sense to be defined below.

Let us consider first the basic example in which $T$ is the monad on $\Graph = \ca G\Set$ whose algebras are categories. For any graph $X$ and $a,b \in X_0$, $TX(a,b)$ is by definition the set of paths in $X$ from $a$ to $b$. Each such path determines a sequence $x=(x_0,...,x_n)$ of objects of $X$ such that $x_0=a$ and $x_n=b$, by reading off the objects of $X$ as they are visited by the given path. The set of all paths visiting exactly these objects of $X$ is the product $\prod_{i=1}^n X(x_{i-1},x_i)$ and by definition one has
\[ \begin{array}{lll} {\prod\limits_{i=1}^n X(x_{i-1},x_i)} & {=} & {T(X(x_0,x_1), X(x_1,x_2)...,X(x_{n-1},x_n))(0,n)} \\
& {=} & {\Tbar\limits_i X(x_{i-1},x_i).}  \end{array} \]
Recall,
\[ (X(x_0,x_1), X(x_1,x_2)...,X(x_{n-1},x_n)) \]
is the graph with set of objects $\{0,...,n\}$, whose hom from $(i{-}1)$ to $i$ is $X(x_{i{-}1},x_i)$, and whose other homs are empty. Thus one can express the general hom $TX(a,b)$ in terms of those of the form
\[ T(X(x_0,x_1), X(x_1,x_2)...,X(x_{n-1},x_n))(0,n). \]
More precisely one has a canonical bijection
\[ \coprod\limits_{a=x_0,...,x_n=b} T(X(x_0,x_1), X(x_1,x_2)...,X(x_{n-1},x_n))(0,n) \iso TX(a,b) \]
which we shall now express more generally.

Let $V$ be a category with coproducts. Given a $V$-graph $X$ and sequence $x=(x_0,...,x_n)$ of objects of $X$, one can define the morphism of $V$-graphs
\[ \overline{x} : (X(x_0,x_1),X(x_1,x_2),...,X(x_{n-1},x_n)) \to X \]
whose object map is $i \mapsto x_i$, and whose hom map between $(i-1)$ and $i$ is the identity. For all such sequences $x$ one has
\[ T(\overline{x})_{0,n} : \Tbar\limits_i X(x_{i-1},x_i) \to TX(x_0,x_n) \]
in $V$, and so taking all sequences $x$ starting at $a$ and finishing at $b$ one induces the canonical map
\[ \pi_{X,a,b} : \coprod\limits_{a=x_0,...,x_n=b} \Tbar\limits_i X(x_{i-1},x_i) \rightarrow TX(a,b). \]
\begin{definition}\label{def:path-like}
Let $V$ be a category with coproducts and $(T,\eta,\mu)$ be a monad on $\ca GV$ over $\Set$. Then $T$ is said to be \emph{path-like} when for all $X \in \ca GV$ and $a,b \in X_0$, the maps $\pi_{X,a,b}$ are isomorphisms.
\end{definition}
Clearly by definition, the category monad on $\ca G\Set$ is path-like. Intuitively, the path-likeness of a general $T$ is saying that the homs $TX(a,b)$ are to be thought of as abstract path objects of a certain type.
\begin{proposition}\label{prop:pl-alg<->cat}
Let $V$ have small coproducts and $(T,\eta,\mu)$ be a path-like monad on $\ca GV$ over $\Set$. Then $\ca G(V)^T \iso \Enrich {\overline{T}}$.
\end{proposition}
\begin{proof}
Let $X$ be a $V$-graph. To give an identity on objects map $a:TX{\rightarrow}X$ is to give maps $a_{y,z}:TX(y,z){\rightarrow}X(y,z)$. By path-likeness these amount to giving for each $n \in \N$ and $x=(x_0,...,x_n)$ such that $x_0=y$ and $x_n=z$, a map
\[ a_x : \Tbar\limits_iX(x_{i-1},x_i) \rightarrow X(y,z) \]
since $\Tbar\limits_iX(x_{i-1},x_i)=Tx^*X(0,n)$, that is $a_x=a_{y,z}T\overline{x}_{0,n}$. When $n=1$, for a given $y,z \in X_0$, $x$ can only be the sequence $(y,z)$. The naturality square for $\eta$ at $\overline{x}$ implies that $\{\eta_X\}_{y,z}=T\overline{x}_{0,1}\{\eta_{(X(y,z))}\}_{0,1}$, and the definition of $\overline{(\,\,)}$ says that $\{\eta_{(X(y,z))}\}_{0,1}=\overline{\eta}_{X(y,z)}$. Thus to say that a map $a:TX{\rightarrow}X$ satisfies the unit law of a $T$-algebra is to say that $a$ is the identity on objects and that the $a_x$ described above satisfy the unit axioms of a $\overline{T}$-category.

To say that $a$ satisfies the associative law is to say that for all $y,z \in X_0$,
\begin{equation}\label{eq:assoc1} \xymatrix{{T^2X(y,z)} \ar[r]^-{\{\mu_X\}_{y,z}} \ar[d]_{Tx_{y,z}} & {TX(y,z)} \ar[d]^{a_{y,z}} \\ {TX(y,z)} \ar[r]_-{a_{y,z}} & {X(y,z)}} \end{equation}
commutes. Given $x=(x_0,...,x_n)$ from $X$ with $x_0=y$ and $x_n=z$, and $w=(w_0,...,w_k)$ from $T(X(x_0,x_1),...,X(x_{n-1},x_n))$ with $w_0=0$ and $w_k=n$, one can consider the composite map $T(\overline{x})_{0,n}T(\overline{w})_{0,k}$, and since the coproduct of coproducts is a coproduct, all such maps for $x$ and $w$ such that $x_0=y$ and $x_n=z$ form a coproduct cocone. Precomposing (\ref{eq:assoc1}) with the coproduct inclusions gives the commutativity of
\[ \xymatrix{{\Tbar\limits_i\Tbar\limits_jX(x_{ij-1},x_{ij})} \ar[r]^-{\overline{\mu}} \ar[d]_{\Tbar\limits_ia} & {\Tbar\limits_{ij}X(x_{ij-1},x_{ij})} \ar[d]^{a_x} \\ {\Tbar\limits_iX(x_{w_{i-1}},x_{w_i})} \ar[r]_-{a_w} & {X(y,z)}} \]
and conversely by the previous sentence if these squares commute for all $x$ and $w$, then one recovers the commutativity of (\ref{eq:assoc1}). This completes the description of the object part of $\ca G(V)^T \iso \Enrich {\overline{T}}$.

Let $(X,a)$ and $(X',a')$ be $T$-algebras and $F:X{\rightarrow}X'$ be a $V$-graph morphism. To say that $F$ is a $T$-algebra map is a condition on the maps $F_{y,z}:X(y,z){\rightarrow}X'(Fy,Fz)$ for all $y,z \in X_0$, and one uses path-likeness in the obvious way to see that this is equivalent to saying that the $F_{y,z}$ are the hom maps for a $\overline{T}$-functor.
\end{proof}
Returing to our basic example in which $T$ is the category monad on $\ca G\Set$, note that one can decompose the general hom $TX(a,b)$ even further when the $X(x_{i-1},x_i)$ are themselves coproducts (in $\Set$). For instance for sets $A$, $B$ and $C$ one has
\[ \begin{array}{lll} {T(A + B,C)(0,2)} & {\iso} & {(A + B) \times C} \iso {(A \times C) + (B \times C)} \\
& {\iso} & {T(A,C)(0,2) + T(B,C)(0,2)} \end{array} \]
by the distributivity of coproducts over products in $\Set$. Most succinctly one has this kind of decomposition simply because in this case $\overline{T}$ is the cartesian product of $\Set$ which is distributive as a multitensor.
\begin{definition}\label{def:distributive-monad-over-Set}
Let $V$ be a category with coproducts and $(T,\eta,\mu)$ be a monad on $\ca GV$ over $\Set$. Then $T$ is said to be \emph{distributive} when $\overline{T}$ is a distributive multitensor.
\end{definition}
For a more explicit rendering of definition(\ref{def:distributive-monad-over-Set}) which avoids explicit mention of $\overline{T}$, consider a finite sequence of families of sets
\[ ((Z_{i_j} : i_j \in I_j) \,\, | \,\, 1 \leq j \leq n). \]
Then for any sequence of indices $(i_1,...,i_n) \in I_1 \times ... \times I_n$, the coproduct inclusions give identity-on-objects morphisms of $V$-graphs
\[ \begin{array}{c} {(c_{i_1},...,c_{i_n}) : (Z_{i_1},...,Z_{i_n}) \longrightarrow (\coprod\limits_{i_1} Z_{i_1},...,\coprod\limits_{i_n} Z_{i_n}),} \end{array} \]
and thus morphisms
\[ \begin{array}{c} {T(c_{i_1},...,c_{i_n})_{0,n} : T(Z_{i_1},...,Z_{i_n})(0,n) \longrightarrow T(\coprod\limits_{i_1} Z_{i_1},...,\coprod\limits_{i_n} Z_{i_n})(0,n)} \end{array} \]
in $V$, which together give a morphism
\[ \begin{array}{c} {\delta_{(Z_{i_j})_{j}} : \coprod\limits_{(i_1,...,i_n)} T(Z_{i_1},...,Z_{i_n})(0,n) \longrightarrow T(\coprod\limits_{i_1} Z_{i_1},...,\coprod\limits_{i_n} Z_{i_n})(0,n).} \end{array} \]
in $V$. The distributivity of $T$ then says that for all such finite sequences of families of sets, this induced morphism $\delta_{(Z_{i_j})_{j}}$ is an isomorphism. 
\begin{theorem}\label{thm:characterisation-of-image-of-Gamma}
Let $V$ have coproducts. Then a monad $T$ on $\ca GV$ over $\Set$ is of the form $(\ca GV,\Gamma{E})$ iff it is distributive
and path-like, and in this case $E$ is recovered as $\overline{T}$.
\end{theorem}
\begin{proof}
Suppose that $T$ is distributive and path-like. Since $\overline{T}$ is distributive the morphisms $\pi_{X,a,b}$ are the hom maps of the components of a natural transformation $\pi : \Gamma \overline{T} \to T$, which is easily seen to be compatible with the monad structures. Since $T$ is path-like, this is an isomorphism. The converse follows from remark(\ref{rem:implicit-Tbar}).
\end{proof}

\subsection{2-functoriality of $\Gamma$}
\label{sec:2-functoriality-Gamma}
As the lax-algebras of a 2-monad $M$ lax monoidal categories form a 2-category $\LaxAlg M$. See \cite{Lack-Codescent} for a complete description of the 2-category of lax algebras for an arbitrary 2-monad. Explicitly a lax monoidal functor between lax monoidal categories $(V,E)$ and $(W,F)$ consists of a functor $H:V{\rightarrow}W$, and maps
\[ \psi_{X_i} : \opF\limits_i HX_i \rightarrow H \opE\limits_i X_i \]
natural in the $X_i$ such that
\[ \xygraph{{\xybox{\xygraph{!{0;(.75,0):(0,1.333)::} {HX}="l" [r(2)] {F_1HX}="r" [dl] {HE_1X}="b" "l"(:"r"^-{u_{HX}}:"b"^{\psi_X},:"b"_{Hu_X})}}} [r(5)]
{\xybox{\xygraph{!{0;(2,0):(0,.5)::} {\opF\limits_i\opF\limits_jHX_{ij}}="tl" [r] {\opF\limits_iH\opE\limits_jX_{ij}}="tm" [r] {H\opE\limits_i\opE\limits_jX_{ij}}="tr" [l(.5)d] {H\opE\limits_{ij}X_{ij}}="br" [l] {\opF\limits_{ij}HX_{ij}}="bl" "tl" (:@<1ex>"tm"^-{\opF\limits_i\psi}:@<1ex>"tr"^-{\psi\opE\limits_j}:"br"^{H\sigma},:"bl"_{\sigma{H}}:@<1ex>"br"_-{\psi})}}}} \]
commute for all $X$ and $X_{ij}$ in $V$. A monoidal natural transformation between lax monoidal functors
\[ (H,\psi),(K,\kappa):(V,E){\rightarrow}(W,F) \]
consists of a natural transformation $\phi:H{\rightarrow}K$ such that
\[ \xygraph{!{0;(2,0):(0,.5)::} {\opF\limits_iHX_i}="tl" [r] {H\opE\limits_iX_i}="tr" [d] {K\opE\limits_iX_i}="br" [l] {\opF\limits_iKX_i}="bl" "tl" (:@<1ex>"tr"^-{\psi}:"br"^{\phi\opE\limits_i},:"bl"_{\opF\limits_i\phi}:@<1ex>"br"_-{\kappa})} \]
commutes for all $X_i$. We denote by $\DISTMULT$ the 2-category $\DISTMULT$ of distributive multitensors. It is the full sub-2-category of $\LaxAlg M$ consisting of the $(V,E)$ such that $V$ has coproducts and $E$ is distributive.

For any 2-category $\ca K$ recall the 2-category $\MND(\ca K)$ from \cite{Street-FTM} of monads in $\ca K$. Another way to describe this very canonical object is that it is the 2-category of lax algebras of the identity monad on $\ca K$. Explicitly the 2-category $\MND(\CAT)$ has as objects pairs $(V,T)$ where $V$ is a category and $T$ is a monad on $V$. An arrow $(V,T){\rightarrow}(W,S)$ is a pair consisting of a functor $H:V{\rightarrow}W$ and a natural transformation $\psi:SH{\rightarrow}HT$ satisfying the obvious 2 axioms: these are just the ``unary'' analogues of the axioms for a lax monoidal functor written out above. For example, any lax monoidal functor $(H,\psi)$ as above determines a monad functor $(H,\psi_1):(V,E_1){\rightarrow}(W,F_1)$. A monad transformation between monad functors
\[ (H,\psi),(K,\kappa):(V,T){\rightarrow}(W,S) \]
consists of a natural transformation $\phi:H{\rightarrow}K$ satisfying the obvious axiom. For example a monoidal natural transformation $\phi$ as above is a monad transformation $(H,\psi_1){\rightarrow}(K,\kappa_1)$.

In fact as we are interested in monads over $\Set$, we shall work not with $\MND(\CAT)$ but rather with $\MND(\CAT/\Set)$. An arrow $(V,T) \rightarrow (W,S)$ of this 2-category is a pair $(H,\psi)$ as in the case of $\MND(\CAT)$, with the added condition that $\psi$'s components are the identities on objects, and similarly the 2-cells of $\MND(\CAT/\Set)$ come with an extra identity-on-object condition.

We shall now exhibit the 2-functor
\[ \Gamma : \DISTMULT \rightarrow \MND(\CAT/\Set) \]
which on objects is given by $(V,E) \mapsto (\ca GV,\Gamma{E})$. Let $(H,\psi):(V,E){\rightarrow}(W,F)$ be a lax monoidal functor between distributive lax monoidal categories. Then for $X \in \ca GV$ and $a,b \in X_0$, we define the hom map $\Gamma(\psi)_{X,a,b}$ as
\[ \xygraph{!{0;(5,0):(0,.3)::} {\coprod\limits_{a=x_0,...,x_n=b} \opF\limits_iHX(x_{i-1},x_i)}="l" [r] {\coprod\limits_{a=x_0,...,x_n=b} H\opE\limits_iX(x_{i-1},x_i)}="m" [l(.5)d] {H \coprod\limits_{a=x_0,...,x_n=b} \opE\limits_iX(x_{i-1},x_i)}="r" "l":"m"^-{\coprod \psi}:"r"^(.4){\tn{obst.}}:@{<-}"l"^(.6){\Gamma(\psi)_{X,a,b}} "l" [d(.5)r(.5)] {=}} \]
where ``$\tn{obst.}$'' denotes the obstruction map to $H$ preserving coproducts. It follows easily from the definitions that $(\ca GH,\Gamma(\psi))$ as defined here satisfies the axioms of a monad functor. Moreover given a monoidal natural transformation $\phi:(H,\psi){\rightarrow}(K,\kappa)$, it also follows easily from the definitions that
\[ \ca G\phi:(\ca GH,\Gamma(\psi)){\rightarrow}(\ca GK,\Gamma(\kappa)) \]
is a monad transformation. It is also straight-forward to verify that these assignments are 2-functorial.

Lax algebras of a 2-monad organise naturally into \emph{two} different 2-categories depending on whether one takes lax or oplax algebra morphisms. So in particular one has the 2-category $\OpLaxAlg M$ of lax monoidal categories, \emph{op}lax-monoidal functors between them and monoidal natural transformations between those. The coherence $\psi$ for an oplax $(H,\psi):(V,E){\rightarrow}(W,F)$ goes in the other direction, and so its components look like this:
\[ \psi_{X_i} : H \opE\limits_i X_i \rightarrow \opF\limits_i HX_i. \]
The reader should easily be able to write down explicitly the two coherence axioms that this data must satisfy, as well as the condition that must be satisfied by a monoidal natural transformation between oplax monoidal functors. Similarly there is a dual version $\OpMND(\ca K)$ of the 2-category $\MND(\ca K)$ of monads in a given 2-category $\ca K$ discussed above \cite{Street-FTM}. An arrow $(V,T){\rightarrow}(W,S)$ of $\OpMND(\CAT)$ consists of a functor $H:V{\rightarrow}W$ and a natural transformation $\psi:HT{\rightarrow}SH$ satisfying the two obvious axioms. An arrow of $\OpMND(\CAT)$ is called a monad opfunctor. As before $\OpMND(\CAT/\Set)$ differs from $\MND(\CAT/\Set)$ in that all the categories involved come with a functor into $\Set$, and all the functors and natural transformations involved are compatible with these forgetful functors.

When defining the one-cell map of $\Gamma$ above we were helped by the fact that the coproduct preservation obstruction went the right way: see the definition of the monad functor $(\ca GH,\Gamma{\psi})$ above. This time however we will not be so lucky. For this reason we define the 2-category $\OpDISTMULT$ to be the locally full sub-2-category of $\OpLaxAlg M$ consisting of the distributive lax monoidal categories, and the oplax monoidal functors $(H,\psi)$ such that $H$ preserves coproducts. Thus we can define 
\[ \Gamma' : \OpDISTMULT \rightarrow \OpMND(\CAT/\Set) \]
on objects by $(V,E) \mapsto (\ca GV,\Gamma{E})$. Let $(H,\psi):(V,E){\rightarrow}(W,F)$ be an oplax monoidal functor between distributive lax monoidal categories. Then for $X \in \ca GV$ and $a,b \in X_0$, we define the hom map $\Gamma'(\psi)_{X,a,b}$ as
\[ \xygraph{!{0;(5,0):(0,.3)::} {H \coprod\limits_{a=x_0,...,x_n=b} \opE\limits_iX(x_{i-1},x_i)}="l" [r] {\coprod\limits_{a=x_0,...,x_n=b} H\opE\limits_iX(x_{i-1},x_i)}="m" [l(.5)d] {\coprod\limits_{a=x_0,...,x_n=b} \opF\limits_iHX(x_{i-1},x_i)}="r" "l":"m"^-{\iso}:"r"^(.4){\coprod \psi}:@{<-}"l"^(.6){\Gamma'(\psi)_{X,a,b}} "l" [d(.5)r(.5)] {=}} \]
It follows easily from the definitions that $(\ca GH,\Gamma'(\psi))$ as defined here satisfies the axioms of a monad opfunctor. Moreover given a monoidal natural transformation $\phi:(H,\psi){\rightarrow}(K,\kappa)$, it also follows easily from the definitions that
\[ \ca G\phi:(\ca GH,\Gamma'(\psi)){\rightarrow}(\ca GK,\Gamma'(\kappa)) \]
is a monad transformation. It is also straight-forward to verify that these assignments are 2-functorial.

\subsection{Properties of the 2-functor $\Gamma$}
\label{sec:properties-of-2-functor-Gamma}
\begin{proposition}\label{prop:Gamma-locally-ff}
$\Gamma$ and $\Gamma'$ are locally fully faithful 2-functors.
\end{proposition}
\begin{proof}
We will verify that $\Gamma$ is locally fully faithful; the proof for $\Gamma'$ is similar. Let $(H,\psi),(K,\kappa):(V,E) \to (W,F)$ be lax monoidal functors between distributive lax monoidal categories. Given $\phi:\ca GH \to \ca GK$ so that $\phi:(\ca GH,\Gamma(\psi)) \to (\ca GK,\Gamma(\kappa))$ is a monad 2-cell, we must exhibit a unique monoidal natural transformation $\phi':H \to K$ such that $\ca G\phi'=\phi$. By proposition(\ref{prop:G1-locally-ff}) there is a unique $\phi':H \to K$ such that $\ca G\phi'=\phi$, and from the proof of proposition(\ref{prop:G1-locally-ff}) this is defined by $\phi'_Z=\phi_{(Z),0,1}$. So it suffices to show that $\phi$ satisfies the monad 2-cell axiom iff $\phi'$ satisfies the monoidal naturality axiom. The monad 2-cell axiom says the outside of
\[ \xygraph{!{0;(3,0):(0,.333)::} {\coprod \opF\limits_i HX(x_{i-1},x_i)}="tl" [r] {\coprod H\opE\limits_i X(x_{i-1},x_i)}="tm" [r] {H \coprod \opE\limits_i X(x_{i-1},x_i)}="tr" [d] {K \coprod \opE\limits_i X(x_{i-1},x_i)}="br" [l] {\coprod K\opE\limits_i X(x_{i-1},x_i)}="bm" [l] {\coprod \opF\limits_i HX(x_{i-1},x_i)}="bl"
"tl":"tm"^-{\coprod \psi}:"tr"^-{\tn{obstn}} "bl":"bm"_-{\coprod \kappa}:"br"_-{\tn{obstn}} "tl":"bl"_{\coprod \opF\limits_i \phi} "tm":"bm"^-{\coprod \phi'} "tr":"br"^{\phi_{\Gamma E(X),a,b}}} \]
commutes for all $X \in \ca GV$ and $a,b \in X_0$, and where all the coproducts are taken over all sequences $a=x_0,...,x_n=b$. Since $\phi_{\Gamma E(X),a,b} = \phi'_{\Gamma E(X)(a,b)}$, the right hand square commutes by the naturality of the obstruction maps. Monoidal naturality of $\phi'$ says that for all $(Z_1,...,Z_n)$
\[ \xygraph{!{0;(2,0):(0,.5)::} {\opF\limits_iHZ_i}="tl" [r] {H\opE\limits_iZ_i}="tr" [d] {K\opE\limits_iZ_i}="br" [l] {\opF\limits_iKZ_i}="bl" "tl":"tr"^-{\psi}:"br"^-{\phi'}:@{<-}"bl"^-{\psi'}:@{<-}"tl"^-{\opF\limits_i\phi'}} \]
commutes, which implies that the left hand square above commutes, and so monoidal naturality of $\phi'$ implies the monad 2-cell axiom for $\phi$. For the converse take $X=(Z_1,...,Z_n)$, $a=0$ and $b=n$. In this case the coproduct involved in the monad 2-cell axiom has only one non-trivial summand, that for the sequence $(0,1,...,n)$. Thus the obstruction maps are isomorphisms, and the left hand square is exactly the monoidal naturality axiom for $\phi'$.
\end{proof}
While $\Gamma$ and $\Gamma'$ aren't themselves 2-fully faithful, proposition(\ref{prop:Psi-2-ff}) is a useful related statement which is true. By definition $\Gamma$ and $\Gamma'$ fit into commutative squares
\[ \xygraph{{\xybox{\xygraph{!{0;(2.3,0):(0,.4347)::} {\DISTMULT}="tl" [r] {\MND(\CAT/\Set)}="tr" [d] {\CAT/\Set}="br" [l] {\CAT}="bl" "tl":"tr"^-{\Gamma}:"br"^-{}:@{<-}"bl"^-{\ca G_1}:@{<-}"tl"^-{}}}} [r(4.9)]
{\xybox{\xygraph{!{0;(2.8,0):(0,.3571)::} {\OpDISTMULT}="tl" [r] {\OpMND(\CAT/\Set)}="tr" [d] {\CAT/\Set}="br" [l] {\CAT}="bl" "tl":"tr"^-{\Gamma'}:"br"^-{}:@{<-}"bl"^-{\ca G_1}:@{<-}"tl"^-{}}}}} \]
in which the vertical arrows are the obvious forgetful 2-functors. Let us write $\GMND$ (resp. $\GOpMND$) for the 2-categories obtained by pulling back $\ca G_1$ along the appropriate forgetful 2-functor, and by
\[ \begin{array}{lccr} {\Psi : \DISTMULT \to \GMND} &&& {\Psi' : \OpDISTMULT \to \GOpMND} \end{array} \]
the induced 2-functors.

In more concrete terms an object of $\GMND$ (or of $\GOpMND$) is a pair $(V,T)$ where $V$ is a category with coproducts and $T$ is a monad on $\ca GV$ over $\Set$. By definition and by theorem(\ref{thm:characterisation-of-image-of-Gamma}), we know that $(V,T)$ is in the image of $\Psi$ (or of $\Psi'$) iff $T$ is distributive and path-like. A morphism $(V,T) \to (W,S)$ in $\GMND$ is a pair $(H,\psi)$ where $H:V \to W$ is a functor, and $\psi$ is 2-cell data making $(\ca GH,\psi) : (\ca GV,T) \to (\ca GW,S)$ a monad functor. Similarly, a morphism $(V,T) \to (W,S)$ in $\GOpMND$ is a pair $(H,\psi)$ where $H:V \to W$ is a coproduct preserving functor, and $\psi$ is 2-cell data making $(\ca GH,\psi) : (\ca GV,T) \to (\ca GW,S)$ a monad opfunctor. A two cell $\phi : (H,\psi) \to (K,\kappa)$ of $\GMND$ is a natural transformation $\phi:H \to K$, making $\ca G\phi : (\ca GH,\psi) \to (\ca GK,\kappa)$ a monad 2-cell, and the 2-cells of $\GOpMND$ are described similarly.
\begin{proposition}\label{prop:Psi-2-ff}
$\Psi$ and $\Psi'$ are 2-fully faithful.
\end{proposition}
\begin{proof}
We shall prove that $\Psi$ is 2-fully faithful; the proof for $\Psi'$ is similar. By definition and proposition(\ref{prop:Gamma-locally-ff}) $\Psi$ is locally fully faithful. Thus it suffices to show that $\Psi$ is fully faithful as a mere functor. This in turn amounts to showing that for any functor $H:V \to W$ between categories with coproducts, and any natural transformation $\psi:\Gamma(F)\ca G(H) \to \ca G(H)\Gamma(E)$ such that $(\ca GH,\psi):(\ca GV,\Gamma E) \to (\ca GW,\Gamma F)$ is a monad functor, that there exists a unique $\psi':FM(H) \to HE$ making $(H,\psi'):(V,E) \to (W,F)$ a lax monoidal functor such that $\Gamma \psi' = \psi$.

The homs of the components of $\psi$ are maps in $V$ of the form
\[ \psi_{X,a,b} : \coprod\limits_{a=x_0,...,x_n=b} \opF\limits_iHX(x_{i-1},x_i) \to H \coprod\limits_{a=x_0,...,x_n=b} \opE\limits_iX(x_{i-1},x_i) \]
and in the case where $X=(Z_1,...,Z_n)$, $a=0$ and $b=n$, the coproducts here have only one non-trivial summand, that for the sequence $(0,1,...,n)$, and so we define
\[ \psi'_{Z_i} := \psi_{(Z_1,...,Z_n),0,n} : \opF\limits_iHZ_i \to H\opE\limits_iZ_i. \]
The lax monoidal functor coherence axioms for $\psi'$ follow easily from the monad functor coherence axioms for $\psi$. To say that $\Gamma \psi' = \psi$ is to say that
\[ \xygraph{!{0;(3,0):(0,.4)::} {\coprod\limits_{a=x_0,...,x_n=b} \opF\limits_iHX(x_{i-1},x_i)}="l" [dr] {\coprod\limits_{a=x_0,...,x_n=b} H\opE\limits_iX(x_{i-1},x_i)}="m" [ur] {H \coprod\limits_{a=x_0,...,x_n=b} \opE\limits_iX(x_{i-1},x_i)}="r" "l":"m"_(.4){\coprod \psi'}:"r"_(.6){\tn{obstn}}:@{<-}"l"_-{\psi_{X,a,b}}} \]
commutes for all $X \in \ca GV$ and $a,b \in X_0$, which is to say that
\[ \xygraph{!{0;(4,0):(0,.25)::} {\opF\limits_iHX(x_{i-1},x_i)}="tl" [r] {\coprod\limits_{a=x_0,...,x_n=b} \opF\limits_iHX(x_{i-1},x_i)}="tr" [d] {H \coprod\limits_{a=x_0,...,x_n=b} \opE\limits_iX(x_{i-1},x_i)}="br" [l] {H \opE\limits_iX(x_{i-1},x_i)}="bl" "tl":"tr"^-{c_{x_i}}:"br"^-{\psi_{X,a,b}}:@{<-}"bl"^-{Hc_{x_i}}:@{<-}"tl"^-{\psi'}} \]
commutes for all $X \in \ca GV$ and $a=x_0,...,x_n=b \in X_0$. But this last square is just the hom between $a$ and $b$ for the naturality square for $\psi$ with respect to the canonical inclusion $(X(x_0,x_1),...,X(x_{n-1},x_n)) \hookrightarrow X$. Finally we note that given $\phi$ making $(H,\phi):(V,E) \to (W,F)$ a lax monoidal functor one has for $Z_1,...,Z_n \in V$
\[ (\Gamma \phi)'_{Z_1,...,Z_n} = (\Gamma \phi)_{(Z_1,...,Z_n),0,n} = \phi_{Z_i} \]
and so $\psi \mapsto \psi'$ is the inverse of the arrow map of $\Psi$.
\end{proof}

\subsection{Cartesian transformations}
\label{sec:compatibility-of-2-functor-Gamma}
We now note that the above constructions are compatible with cartesian transformations.
\begin{lemma}\label{lem:cartesian-obstructions}
Suppose that $H:V \to W$ is a pullback preserving functor between extensive categories. Then the obstruction maps
\[ \coprod_{i \in I} HX_i \to H \coprod_{i \in I} X_i \]
are cartesian-natural in the $X_i$.
\end{lemma}
\begin{proof}
The naturality squares in question appear as the right hand square in
\[ \xygraph{!{0;(2,0):(0,.5)::} {HX_i}="tl" [r] {\coprod HX_i}="tm" [r] {H \coprod X_i}="tr" [d] {H \coprod Y_i}="br" [l] {\coprod HY_i}="bm" [l] {HY_i}="bl" "tl"(:"tm"^-{c_{HX_i}}:"tr"^-{\tn{obstn}},:@/^{2pc}/"tr"^-{Hc_{X_i}}) "bl"(:"bm"_-{c_{HY_i}}:"br"_-{\tn{obstn}},:@/_{2pc}/"br"_-{Hc_{Y_i}}) "tl":"bl"_{Hf_i} "tm":"bm"^{\coprod Hf_i} "tr":"br"^{H \coprod f_i}} \]
Since $V$ is extensive and $H$ preserves pullbacks it follows that outside square is a pullback for all $i \in I$. Since $W$ is extensive it follows then that the right hand square is a pullback as required.
\end{proof}
\begin{proposition}\label{prop:Gamma-respects-cartesian-transformations}
\begin{enumerate}
\item Let $(H,\psi) : (V,E) \to (W,F)$ be a lax monoidal functor between distributive lax monoidal categories, and suppose that $V$ and $W$ are extensive and $H$ preserves pullbacks. Then the following statements are equivalent:
\begin{enumerate}
\item $\psi$ is a cartesian transformation.
\item $\Gamma \psi$ is a cartesian transformation.
\item $\Psi \psi$ is a cartesian transformation.
\end{enumerate}
\label{propcase:Gamma-cart}
\item Let $(H,\psi) : (V,E) \to (W,F)$ be a coproduct preserving oplax monoidal functor between distributive lax monoidal categories, and suppose that $W$ is extensive. Then the following statements are equivalent:
\begin{enumerate}
\item $\psi$ is a cartesian transformation.
\item $\Gamma' \psi$ is a cartesian transformation.
\item $\Psi' \psi$ is a cartesian transformation.
\end{enumerate}
\label{propcase:Gamma-prime-cart}
\end{enumerate}
\end{proposition}
\begin{proof}
The statements that $\Gamma \psi$ cartesian iff $\Psi \psi$ is cartesian, and similarly for $\Gamma' \psi$ and $\Psi' \psi$, follows by definition. For $(H,\psi)$ as in (\ref{propcase:Gamma-cart}) note that by the extensivity of $W$ the maps
\[ \coprod\limits_{a=x_0...,x_n=b} \psi_{X(x_{i-1},x_i)} : \coprod\limits_{a=x_0...,x_n=b} \opF\limits_i HX(x_{i-1},x_i) \to \coprod\limits_{a=x_0...,x_n=b} H\opE\limits_i X(x_{i-1},x_i) \]
are cartesian natural iff $\psi$ is, and so (\ref{propcase:Gamma-cart}) follows by lemma(\ref{lem:cartesian-obstructions}) and the definition of $\Gamma$. For $(H,\psi)$ as in (\ref{propcase:Gamma-prime-cart}) one has by the extensivity of $W$ that the cartesian naturality of
\[ \coprod\limits_{a=x_0...,x_n=b} \psi_{X(x_{i-1},x_i)} : \coprod\limits_{a=x_0...,x_n=b} H\opE\limits_i X(x_{i-1},x_i) \to \coprod\limits_{a=x_0...,x_n=b} \opF\limits_i HX(x_{i-1},x_i) \]
is equivalent to that of $\psi$, so that (\ref{propcase:Gamma-prime-cart}) follows from the definition of $\Gamma'$.
\end{proof}
Recall that for a cartesian monad $(T,\eta,\mu)$ on a category $V$ with pullbacks, a \emph{$T$-operad} consists of another monad $A$ on $V$ together with a cartesian monad morphism $\alpha:A \to T$, that is to say, one has a natural transformation $\alpha:A \to T$ whose naturality squares are pullbacks, and is a morphism of monoids in the monoidal category $\End(V)$. Given a cartesian monad $T$ on $\ca GV$ over $\Set$, a \emph{$T$-operad over $\Set$} is defined in the same way except that the natural transformation $\alpha$ lives over $\Set$, which means that in addition $\alpha$'s components are identities on objects. For instance for $T = \ca T_{\leq n}$ the strict $n$-category monad on $\ca G^n\Set$, $T$-operads over $\Set$ were called \emph{normalised} $n$-operads of \cite{Batanin-MonGlobCats} and the terminology ``normalised'' was also used in \cite{BataninWeber-EnHop}. Finally we note that $\Gamma$'s image is closed under cartesian monad maps.
\begin{lemma}\label{lem:transfer-dpl}
Let $V$ be a lextensive category and $T$ be a cartesian monad on $\ca GV$ over $\Set$. Let $\alpha:A \rightarrow T$ be a $T$-operad over $\Set$. 
\begin{enumerate}
\item  If $T$ is distributive then so is $A$.\label{tdpl1}
\item  If $T$ is path-like then so is $A$.\label{tdpl2}
\end{enumerate}
\end{lemma}
\begin{proof}
(\ref{tdpl1}): given an $n$-tuple $(X_1,...,X_n)$ of objects of $V$ and a coproduct cocone
\[ (c_j : X_{ij} \rightarrow X_i \,\, : \,\, j \in J) \]
where $1{\leq}i{\leq}n$, we must show that the hom-maps
\[ A(X_1,...,c_j,...,X_n)_{0,n} : A(X_1,...,X_{ij},...,X_n)(0,n) \rightarrow A(X_1,...,X_i,...,X_n)(0,n) \]
form a coproduct cocone. For $j \in J$ we have a pullback square
\[ \xygraph{!{0;(6,0):(0,.167)::}
{A(X_1,...,X_{ij},...,X_n)(0,n)}="tl" [r] {A(X_1,...,X_i,...,X_n)(0,n)}="tr" [d] {T(X_1,...,X_i,...,X_n)(0,n)}="br" [l] {T(X_1,...,X_{ij},...,X_n)(0,n)}="bl" "tl"(:"tr"^-{A(X_1,...,c_j,...,X_n)_{0,n}}:"br"^{\alpha},:"bl"_{\alpha}:"br"_-{T(X_1,...,c_j,...,X_n)_{0,n}})
"tl":@{}"br"|-{pb}} \]
and by the distributivity of $T$ the $T(X_1,...,c_j,...,X_n)_{0,n}$ form a coproduct cocone, and thus so do the $A(X_1,...,c_j,...,X_n)_{0,n}$ by the extensivity of $V$.

(\ref{tdpl2}): given $X \in \ca GV$, $a,b \in X_0$ and a sequence $(x_0,...,x_n)$ of objects of $X$ such that $x_0=a$ and $x_n=b$, we have the map
\[ A\overline{x}_{0,n} : A(X(x_0,x_1),...,X(x_{n-1},x_n))(0,n) \rightarrow AX(a,b) \]
and we must show that these maps, where the $x_i$ range over all sequences from $a$ to $b$, form a coproduct cocone. By the path-likeness of $T$ we know that the maps
\[ T\overline{x}_{0,n} : T(X(x_0,x_1),...,X(x_{n-1},x_n))(0,n) \rightarrow TX(a,b) \]
form a coproduct cocone, so we can use the cartesianness of $\alpha$ and the extensivity of $V$ to conclude as in (\ref{tdpl1}).
\end{proof}

\section{Distributive laws between monads and multitensors}
\label{sec:dist-laws-from-monoidal-monads}

In section(1) of the seminal paper \cite{Beck-DLaws} of Jon Beck on monad distributive laws, it is shown that there are three equivalent ways of regarding a distributive law of a monad $S$ over a monad $T$ on the same category:
\begin{enumerate}
\item As a natural transformation $TS \to ST$ satisfying some axioms,
\item as a natural transformation $STST \to ST$ satisfying some axioms, one of which is that it is the multiplication of a monad, and
\item as a lifting of the monad $S$ to the category of algebras of $T$.
\end{enumerate}
In the previous section we saw that $\Gamma$ can be seen as a 2-functor landing in certain 2-categories of monads, and from \cite{Street-FTM} we know that monads in these 2-categories of monads are distributive laws. Thus $\Gamma$ sends monoidal and opmonoidal monads to distributive laws.

In this section we exhibit an analogue of Beck's basic result in our situation \emph{before} applying $\Gamma$, and then relate this to the monad distributive laws one has upon $\Gamma$'s application. Given a monad $T$ and a multitensor $E$ on a category $V$, the analogue of the data $TS \to ST$ of a monad distributive law is that of the coherences making $T$ into a monoidal or opmonoidal monad with respect to $E$. Theorem(\ref{thm:monoidal-monad-dist-law}) is the analogue of Beck's result for monoidal monads, and theorem(\ref{thm:opmonoidal-monad-dist-law}) is the analogue for opmonoidal monads.

Section(\ref{sec:dist-multitensor-over-monad}) completes our development of the theory of monads and multitensors, and in section(\ref{sec:strict-n-cat-monad}) we give an illustration of our theory to efficiently construct the monads for strict $n$-categories and deduce all their important properties.

\subsection{Distributing a monad over a multitensor}
\label{sec:dist-monad-over-multitensor}
Let $(V,E,u,\sigma)$ be a lax monoidal category. Recall that a \emph{monoidal monad} on $(V,E)$ is a monad on $(V,E)$ in the 2-category of lax monoidal categories, lax monoidal functors and monoidal natural transformations. In more explicit terms this amounts to a monad $(T,\eta,\mu)$ on $V$ together with coherence maps
\[ \tau_{X_i} : \opE\limits_i TX_i \to T\opE\limits_i X_i \]
such that
\[ \xygraph{{\xybox{\xygraph{!{0;(.75,0):(0,1.333)::} {T}="l" [r(2)] {E_1T}="r" [dl] {TE_1}="b" "l"(:"r"^-{u T}:"b"^{\tau},:"b"_{Tu})}}} [r(4.5)]
{\xybox{\xygraph{!{0;(2,0):(0,.5)::} {\opE\limits_i\opE\limits_jT}="tl" [r] {\opE\limits_iT\opE\limits_j}="tm" [r] {T\opE\limits_i\opE\limits_j}="tr" [l(.5)d] {T\opE\limits_{ij}}="br" [l] {\opE\limits_{ij}T}="bl" "tl" (:@<1ex>"tm"^-{\opE\limits_i\tau}:@<1ex>"tr"^-{\tau\opE\limits_j}:"br"^{T\sigma},:"bl"_{\sigma T}:@<1ex>"br"_-{\tau})}}}} \]
and
\[ \xygraph{{\xybox{\xygraph{!{0;(.75,0):(0,1.333)::} {E_1}="l" [r(2)] {E_1T}="r" [dl] {TE_1}="b" "l"(:"r"^-{E_1 \eta}:"b"^{\tau},:"b"_{\eta E_1})}}} [r(4.5)]
{\xybox{\xygraph{!{0;(2,0):(0,.5)::} {\opE\limits_iT^2}="tl" [r] {T\opE\limits_iT}="tm" [r] {T^2\opE\limits_i}="tr" [l(.5)d] {T\opE\limits_i}="br" [l] {\opE\limits_iT}="bl" "tl" (:@<1ex>"tm"^-{\tau T}:@<1ex>"tr"^-{T\tau}:"br"^{\mu \opE\limits_i},:"bl"_{\opE\limits_i\mu}:@<1ex>"br"_-{\tau})}}}} \]
commute. Ignoring the subscripts in the above data and axioms one can see immediately the formal resemblence with monad distributive laws. Restricting attention to singleton sequences of objects from $V$ one has a monad distributive law of $T$ over $E_1$.

Given a multitensor $E$ on a category $V$ and a monad $(S,\eta,\mu)$ on $\ca GV$, a \emph{lifting of $S$ to $\Enrich{E}$} is a monad $(S',\eta',\mu')$ on $\Enrich{E}$ such that $SU^E = U^ES'$, $\eta U^E = U^E\eta'$ and $\mu U^E = U^E \mu'$, where we recall that $U^E:\Enrich{E} \to \ca GV$ is the forgetful functor. We arrive now at our monoidal monad analogue of Beck's basic monad distributive law result.
\begin{theorem}\label{thm:monoidal-monad-dist-law}
Let $V$ be a category, $(E,u,\sigma)$ be a multitensor on $V$ and $(T,\eta,\mu)$ be a monad on $V$. Then there is a bijection between the following types of data:
\begin{enumerate}
\item  Morphisms $\tau_{X_i} : \opE\limits_i TX_i \to T \opE\limits_i X_i$
of $V$ providing the coherences making $T$ into a monoidal monad.\label{data:monoidal-monad-coherence}
\item  Morphisms $\sigma_{X_{ij}}' : T\opE\limits_iT\opE\limits_j X_{ij} \to T\opE\limits_{ij} X_{ij}$ of $V$ providing the substitutions for a multitensor $(TE,u',\sigma')$ where $u_X'=\eta_{E_1X}u_X$, $\eta E:E \to TE$ is a multitensor map, $Tu:T \to TE_1$ is a monad map, and the composite
\[ \xygraph{!{0;(2,0):} {T\opE\limits_i}="l" [r] {TE_1T\opE\limits_i}="m" [r] {T\opE\limits_i}="r" "l":"m"^-{Tu\eta\opE\limits_i}:"r"^-{\sigma'}} \]
is the identity.\label{data:composite-multitensor}
\newcounter{continued-numbering}
\setcounter{continued-numbering}{\value{enumi}}
\end{enumerate}
if in addition $V$ has coproducts and $E$ is distributive, then the data (\ref{data:monoidal-monad-coherence})-(\ref{data:composite-multitensor}) are also in bijection with
\begin{enumerate}
\setcounter{enumi}{\value{continued-numbering}}
\item  Identity on object morphisms $\lambda_X:\Gamma(E)\ca G(T)(X) \to \ca G(T)\Gamma(E)(X)$ of $\ca GV$ providing a monad distributive law of $\ca G(T)$ over $\Gamma(E)$.\label{data:monad-distributive-law}
\item  Liftings of the monad $\ca G(T)$ to a monad $T'$ on $\Enrich{E}$.\label{data:monad-lifting-ECat}
\end{enumerate}
and for any given instance of such data one has an isomorphism
\[ \Enrich{(TE)} \iso \Enrich{E}^{T'} \]
of categories commuting with the forgetful functors into $\ca GV$.
\end{theorem}
\begin{proof}
(\ref{data:monoidal-monad-coherence}){$\iff$}(\ref{data:composite-multitensor}): The basic idea of this proof is to adapt the discussion of \cite{Beck-DLaws} section 1 replacing one of the monads by a multitensor. Suppose maps $\tau_{X_i}$ are given which make $T$ into a monoidal monad. Define $u_X'=\eta_{E_1X}u_X$ and $\sigma_{X_{ij}}'$ to be given by the composite
\[ \xygraph{!{0;(2.5,0):} {T\opE\limits_iT\opE\limits_j X_{ij}}="l" [r] {T^2\opE\limits_i\opE\limits_j X_{ij}}="m" [r] {T\opE\limits_{ij} X_{ij}.}="r" "l":"m"^-{T \tau E}:"r"^-{\mu \sigma}} \]
The axioms exhibiting $(TE,u',\sigma')$ as a multitensor, $\eta E:E \to TE$ as a multitensor map, $Tu:T \to TE_1$ as a monad map, and $\sigma'(Tu\eta\opE\limits_i) = \id$ follow easily from the multitensor axioms on $E$, the monad axioms on $T$ and the monoidal monad coherence axioms. Conversely given the data $\sigma'$ as in (\ref{data:composite-multitensor}) one defines the monoidal monad coherence $\tau$ as the composite
\[ \xygraph{!{0;(2,0):} {\opE\limits_iT}="l" [r] {T\opE\limits_iTE_1}="m" [r] {T\opE\limits_i.}="r" "l":@<1ex>"m"^-{\eta\opE\limits_iTu}:@<1ex>"r"^-{\sigma'}} \]
The axioms involving $\tau$, $\eta$ and $u$ are verified easily.

In order to verify the other two axioms one must first observe that
\begin{equation}\label{eq:unit-TE-T-E}
\xygraph{{\xybox{\xygraph{!{0;(2,0):(0,.5)::} {T\opE\limits_i\opE\limits_j}="l" [r] {T\opE\limits_iT\opE\limits_j}="r" [d] {T\opE\limits_{ij}}="b" "l":"r"^-{T\opE\limits_i\eta\opE\limits_j}:"b"^-{\sigma'}:@{<-}"l"^-{T\sigma}}}} [r(4)]
{\xybox{\xygraph{!{0;(2,0):(0,.5)::} {T^2\opE\limits_i}="l" [r] {TE_1T\opE\limits_i}="r" [d] {T\opE\limits_{i}}="b" "l":"r"^-{TuT\opE\limits_i}:"b"^-{\sigma'}:@{<-}"l"^-{\mu\opE\limits_i}}}}}
\end{equation}
commutes. One witnesses the commutativity of the triangle on the left from
\[ \xygraph{!{0;(1.5,0):(0,.667)::} {T\opE\limits_i\opE\limits_j}="t1" [r(2)] {TE_1T\opE\limits_i\opE\limits_j}="t2" [r(2)] {T\opE\limits_i\opE\limits_j}="t3" [dr] {T\opE\limits_iT\opE\limits_j}="m4" [l(2)] {TE_1T\opE\limits_iT\opE\limits_j}="m3" [l(2)] {TE_1\opE\limits_i\opE\limits_j}="m2" [l(2)] {T\opE\limits_{ij}}="m1" [dr] {TE_1\opE\limits_{ij}}="b1" [r(2)] {TE_1T\opE\limits_{ij}}="b2" [r(2)] {T\opE\limits_{ij}}="b3"
"t1"(:"t2"^-{Tu\eta\opE\limits_i\opE\limits_j}:"t3"^-{\sigma'\opE\limits_j}:"m4"^-{T\opE\limits_i\eta\opE\limits_j}:"b3"^(.4){\sigma'},:"m1"_-{T\sigma}:"b1"_-{Tu\opE\limits_{ij}}:"b2"_-{TE_1\eta\opE\limits_{ij}}:"b3"_-{\sigma'})
"m2":"m3"^-{TE_1\eta\opE\limits_i\eta\opE\limits_j}:"m4"^-{\sigma'T\opE\limits_j}
"t1":"m2"^(.6){Tu\opE\limits_i\opE\limits_j}:"b1"^(.4){TE_1\sigma}
"t2":"m3"^(.6){TE_1T\opE\limits_i\eta\opE\limits_j}:"b2"^(.4){TE_1\sigma'}} \]
using also $\sigma'(Tu\eta\opE\limits_i) = \id$, and one witnesses the commutativity of the triangle on the right of (\ref{eq:unit-TE-T-E}) from
\[ \xygraph{!{0;(1.5,0):(0,.667)::} {T^2\opE\limits_i}="t1" [r(2)] {T^2E_1T\opE\limits_i}="t2" [r(2)] {T^2\opE\limits_i}="t3" [dr] {TE_1T\opE\limits_i}="m4" [l(2)] {TE_1TE_1T\opE\limits_i}="m3" [l(2)] {T^3\opE\limits_i}="m2" [l(2)] {T\opE\limits_i}="m1" [dr] {T^2\opE\limits_i}="b1" [r(2)] {TE_1T\opE\limits_i}="b2" [r(2)] {T\opE\limits_i}="b3"
"t1"(:"t2"^-{T^2u\eta\opE\limits_i}:"t3"^-{T\sigma'}:"m4"^-{TuT\opE\limits_i}:"b3"^(.4){\sigma'},:"m1"_-{\mu\opE\limits_i}:"b1"_-{T\eta\opE\limits_i}:"b2"_-{TuT\opE\limits_i}:"b3"_-{\sigma'})
"m2":"m3"^-{TuTuT\opE\limits_i}:"m4"^-{TE_1\sigma'}
"t1":"m2"^(.6){T^2\eta\opE\limits_i}:"b1"^(.4){\mu T\opE\limits_iE_1}
"t2":"m3"^(.6){TuTE_1T\opE\limits_i}:"b2"^(.4){\sigma'T\opE\limits_i}} \]
and $\sigma'(Tu\eta\opE\limits_i) = \id$. Second, one observes that 
\begin{equation}\label{eq:assoc-TE-T-E}
\xygraph{{\xybox{\xygraph{!{0;(2,0):(0,.5)::} {T^2\opE\limits_iT\opE\limits_j}="tl" [r] {T^2\opE\limits_{ij}}="tr" [d] {T\opE\limits_{ij}}="br" [l] {T\opE\limits_iT\opE\limits_j}="bl" "tl":"tr"^-{T\sigma'}:"br"^-{\mu\opE\limits_{ij}}:@{<-}"bl"^-{\sigma'}:@{<-}"tl"^-{\mu\opE\limits_iT\opE\limits_j}}}} [r(4.5)u(.1)]
{\xybox{\xygraph{!{0;(2,0):(0,.5)::} {T\opE\limits_iT\opE\limits_j\opE\limits_k}="tl" [r] {T\opE\limits_{ij}\opE\limits_k}="tr" [d] {T\opE\limits_{ijk}}="br" [l] {T\opE\limits_iT\opE\limits_{jk}}="bl" "tl":"tr"^-{\sigma'\opE\limits_k}:"br"^-{T\sigma}:@{<-}"bl"^-{\sigma'}:@{<-}"tl"^-{T\opE\limits_iT\sigma}}}}}
\end{equation}
commute, but these identities are easily witnessed from
\[ \xygraph{{\xybox{\xygraph{!{0;(2,0):(0,.5)::} 
{T^2\opE\limits_iT\opE\limits_j}="tl" [r] {T^2\opE\limits_{ij}}="tr" [d] {TE_1T\opE\limits_{ij}}="mr" [d] {T\opE\limits_{ij}}="br" [l] {T\opE\limits_iT\opE\limits_j}="bl" [u] {TE_1T\opE\limits_iT\opE\limits_j}="ml"
"tl":"tr"^-{T\sigma'}:"mr"^-{TuT\opE\limits_{ij}}:"br"^-{\sigma'}:@{<-}"bl"^-{\sigma'}:@{<-}"ml"^-{\sigma'T\opE\limits_j}:@{<-}"tl"^-{TuT\opE\limits_iT\opE\limits_j} "ml":"mr"^-{TE_1\sigma'}
}}} [r(4.5)u(.1)]
{\xybox{\xygraph{!{0;(2,0):(0,.5)::}
{T\opE\limits_iT\opE\limits_j\opE\limits_k}="tl" [r] {T\opE\limits_{ij}\opE\limits_k}="tr" [d] {T\opE\limits_{ij}T\opE\limits_k}="mr" [d] {T\opE\limits_{ijk}}="br" [l] {T\opE\limits_iT\opE\limits_{jk}}="bl" [u] {T\opE\limits_iT\opE\limits_jT\opE\limits_k}="ml"
"tl":"tr"^-{\sigma'\opE\limits_k}:"mr"^-{T\opE\limits_{ij}u\opE\limits_k}:"br"^-{\sigma'}:@{<-}"bl"^-{\sigma'}:@{<-}"ml"^-{T\opE\limits_i\sigma'}:@{<-}"tl"^-{T\opE\limits_iT\opE\limits_ju\opE\limits_k} "ml":"mr"^-{\sigma'T\opE\limits_k}}}}} \]
and (\ref{eq:unit-TE-T-E}).

With (\ref{eq:assoc-TE-T-E}) now verified we now proceed to the verification of the other axioms for $\tau$. The axiom expressing the compatibility between $\tau$ and $\sigma$ is verified in
\[ \xygraph{!{0;(1,0):(0,.6)::} {\opE\limits_i\opE\limits_jT}="p11" [r(2)] {\opE\limits_iT\opE\limits_jTE_1}="p12" [r(2)] {\opE\limits_iT\opE\limits_j}="p13" [r(2)] {T\opE\limits_iTE_1\opE\limits_j}="p14" [d(3)] {T\opE\limits_i\opE\limits_j}="p33" [d(3)] {T\opE\limits_{ij}}="p53" [l(3)] {T\opE\limits_{ij}TE_1}="p52" [l(3)] {\opE\limits_{ij}T}="p51" [u(3)r(2)] {T\opE\limits_iT\opE\limits_jTE_1}="p31" [r(2)] {T\opE\limits_iT\opE\limits_j}="p32"
"p11"(:"p12"^-{\opE\limits_i\eta\opE\limits_jTu}:"p13"^-{\opE\limits_i\sigma'}:"p14"^-{\eta\opE\limits_iTu\opE\limits_j}:"p33"^-{\sigma'\opE\limits_j}:"p53"^-{T\sigma},:"p51"_-{\sigma T}:"p52"_-{\eta\opE\limits_{ij}Tu}:"p53"_-{\sigma'})
"p11":"p31"|(.4){\eta\opE\limits_i\eta\opE\limits_jTu}:"p32"^-{T\opE\limits_i\sigma'}:@{<-}"p14"|-{T\opE\limits_iT\sigma} "p12":"p31"^(.6){\eta\opE\limits_iT\opE\limits_jTE_1}:"p52"^-{\sigma'TE_1} "p13":"p32"_(.4){\eta\opE\limits_iT\opE\limits_j}:"p53"^-{\sigma'}} \]
and the axiom expressing the compatibility between $\tau$ and $\mu$ is verified in
\[ \xygraph{!{0;(1,0):(0,.6)::} {\opE\limits_iT^2}="p11" [r(2)] {T\opE\limits_iTE_1T}="p12" [r(2)] {T\opE\limits_iT}="p13" [r(2)] {T^2\opE\limits_iTE_1}="p14" [d(3)] {T^2\opE\limits_i}="p33" [d(3)] {T\opE\limits_i}="p53" [l(3)] {T\opE\limits_iTE_1}="p52" [l(3)] {\opE\limits_iT}="p51" [u(3)r(2)] {T\opE\limits_iTE_1TE_1}="p31" [r(2)] {T\opE\limits_iTE_1}="p32"
"p11"(:"p12"^-{\eta\opE\limits_iTuT}:"p13"^-{\sigma'T}:"p14"^-{T\eta\opE\limits_iTu}:"p33"^-{T\sigma'}:"p53"^-{\mu\opE\limits_i},:"p51"_-{\opE\limits_i\mu}:"p52"_-{\eta E_1Tu}:"p53"_-{\sigma'})
"p11":"p31"|(.4){\eta\opE\limits_iTuTu}:"p32"^-{\sigma'TE_1}:@{<-}"p14"|-{\mu\opE\limits_iTE_1} "p12":"p31"^(.6){T\opE\limits_iTE_1Tu}:"p52"^-{T\opE\limits_i\sigma'} "p13":"p32"_(.4){T\opE\limits_iTu}:"p53"^-{\sigma'}} \]
It follows immediately from the unit laws of $T$ and $E$, that the composite function (\ref{data:monoidal-monad-coherence})$\to$(\ref{data:composite-multitensor})$\to$(\ref{data:monoidal-monad-coherence}) via the above constructions is the identity. It is the commutativity of the outside of
\[ \xygraph{!{0;(2.5,0):(0,.5)::} {T\opE\limits_iT\opE\limits_j}="l" [ur] {T^2\opE\limits_iTE_1\opE\limits_j}="tl" [r] {T^2\opE\limits_i\opE\limits_j}="tr" [d] {T\opE\limits_i\opE\limits_j}="mr" [d] {T\opE\limits_{ij}}="br" [l] {T\opE\limits_iT\opE\limits_j}="bl" [u] {T\opE\limits_iTE_1\opE\limits_j}="ml"
"l":"tl"^-{T\eta\opE\limits_iTu\opE\limits_j}:"tr"^-{T\sigma'\opE\limits_j}:"mr"_-{\mu\opE\limits_i\opE\limits_j}:"br"_-{T\sigma}:@{<-}"bl"^-{\sigma'}:@{<-}"l"^-{1} "tl":"ml"^-{\mu\opE\limits_iTE_1\opE\limits_j}:"bl"^-{T\opE\limits_iT\sigma} "ml":"mr"^-{\sigma'\opE\limits_j} "tr":@/^{4pc}/"br"^-{\mu\sigma}} \]
that says that (\ref{data:composite-multitensor})$\to$(\ref{data:monoidal-monad-coherence})$\to$(\ref{data:composite-multitensor}) is the identity.

(\ref{data:monoidal-monad-coherence}){$\iff$}(\ref{data:monad-distributive-law}): A monoidal monad on $(V,E)$ is a monad on $(V,E)$ in the 2-category $\DISTMULT$. By \cite{Street-FTM} to give a monad $S$ on $\ca GV$ over $\Set$ and a distributive law of $S$ over $\Gamma E$, is to give a monad on $(\ca GV,\Gamma E)$ in $\MND(\CAT/\Set)$. Moreover such distributive laws, for the case where $S = \ca GT$ for some monad $T$ on $V$, are exactly monads in $\GMND$. Thus applying $\Psi$ to monads gives the desired bijection.

(\ref{data:monad-distributive-law}){$\iff$}(\ref{data:monad-lifting-ECat}): By the usual theory of monad distributive laws and since $\ca GV^{\Gamma E} \iso \Enrich{E}$ over $\ca GV$ by proposition(\ref{prop:Gamma-E-algebras-as-E-categories}).

One can readily unpack the lifted monad $T'$ in terms of the monoidal coherence data using the details of the proof of proposition(\ref{prop:Gamma-E-algebras-as-E-categories}) which explain how to regard an $E$-category as a $\Gamma E$-algebra. Let $X \in \Enrich{E}$ and as in section(\ref{sec:recalling-multitensors}) write $\kappa_{x_i}:\opE\limits_iX(x_{i-1},x_i) \to X(x_0,x_n)$ for the composition maps. Then since $T'$ is a lifting of $\ca GT$, $T'X$ must have underlying $V$-graph $\ca GT(X)$, which has the same objects as $X$ and homs given by $(T'X)(a,b) = T(X(a,b))$. The composition map $\kappa_{x_i}':\opE\limits_i T'X(x_{i-1},x_i) \to T'X(x_0,x_n)$ is given by the composite
\[ \xygraph{!{0;(3,0):} {\opE\limits_iTX(x_{i-1},x_i)}="l" [r] {T\opE\limits_i X(x_{i-1},x_i)}="m" [r] {TX(x_0,x_n)}="r" "l":"m"^-{\tau}:"r"^-{T \kappa_{x_i}}} \]
Consider $Z \in \ca GV$. To endow $Z$ with the structure of a $TE$-category is to give maps
\[ \kappa_{z_i} : T\opE\limits_iZ(z_{i-1},z_i) \to Z(z_0,z_n) \]
satisfying the usual axioms. But by precomposing these with the unit for $T$ gives the compositions for an $E$-category structure on $Z$, and by the above explicit description of $T'$, one may readily verify that the remaining structure is exactly that of a $T'$-algebra structure. Similarly given $V$-graph map $f:Z \to Z'$ between $TE$-categories, one may readily verify that $f$ is a $TE$-functor iff it is an $E$-functor and a $T'$-algebra map. Thus we have the object and arrow maps of the required isomorphism $\Enrich{(TE)} \iso \Enrich{E}^{T'}$.
\end{proof}
All aspects of the above result apply to the familiar examples of monoidal monads on $\Set$ regarded as monoidal via its cartesian product, since $\times$ preserves all colimits in each variable and so is certainly distributive. These familiar examples include: the pointed set monad, the covariant power set monad, the monad obtained from a commutative ring $R$ by taking $R$-linear combinations. In \cite{Weber-Funny} a tensor product on $\ca GV$ is provided, under very slight conditions on $V$, with respect to which any monad on $\ca GV$ over $\Set$ is (symmetric) monoidal, giving many examples relevant to higher category theory.

\subsection{Distributing a multitensor over a monad}
\label{sec:dist-multitensor-over-monad}
In a completely analogous fashion one may also regard opmonoidal monads as distributive laws. Once again let $(V,E,u,\sigma)$ be a lax monoidal category. An \emph{opmonoidal monad} on $(V,E)$ is a monad on $(V,E)$ in the 2-category of lax monoidal categories, oplax monoidal functors and monoidal natural transformations, which amounts to a monad $(T,\eta,\mu)$ on $V$ together with coherence maps
\[ \tau_{X_i} : T\opE\limits_i X_i \to \opE\limits_i TX_i \]
such that
\[ \xygraph{{\xybox{\xygraph{!{0;(.75,0):(0,1.333)::} {T}="l" [r(2)] {TE_1}="r" [dl] {E_1T}="b" "l"(:"r"^-{Tu }:"b"^{\tau},:"b"_{uT})}}} [r(4.5)]
{\xybox{\xygraph{!{0;(2,0):(0,.5)::} {T\opE\limits_i\opE\limits_j}="tl" [r] {\opE\limits_iT\opE\limits_j}="tm" [r] {\opE\limits_i\opE\limits_jT}="tr" [l(.5)d] {\opE\limits_{ij}T}="br" [l] {T\opE\limits_{ij}}="bl" "tl" (:@<1ex>"tm"^-{\tau\opE\limits_j}:@<1ex>"tr"^-{\opE\limits_i\tau}:"br"^{\sigma T},:"bl"_{T\sigma}:@<1ex>"br"_-{\tau})}}}} \]
and
\[ \xygraph{{\xybox{\xygraph{!{0;(.75,0):(0,1.333)::} {E_1}="l" [r(2)] {TE_1}="r" [dl] {E_1T}="b" "l"(:"r"^-{\eta E_1 }:"b"^{\tau},:"b"_{E_1\eta})}}} [r(4.5)]
{\xybox{\xygraph{!{0;(2,0):(0,.5)::} {T^2\opE\limits_i}="tl" [r] {T\opE\limits_iT}="tm" [r] {\opE\limits_iT^2}="tr" [l(.5)d] {\opE\limits_iT}="br" [l] {T\opE\limits_i}="bl" "tl" (:@<1ex>"tm"^-{T \tau}:@<1ex>"tr"^-{\tau T}:"br"^{\opE\limits_i \mu},:"bl"_{\mu\opE\limits_i}:@<1ex>"br"_-{\tau})}}}} \]
commute. Recalling that $M$ is our notation for the monoid monad on $\CAT$, we shall use the alternative notations
\[ EM(T)(X_1,...,X_n) = \opE\limits_i TX_i \]
interchangeably as convenience dictates.

Given a multitensor $(E,u,\sigma)$ on a category $V$ and a monad $(T,\eta,\mu)$ on $V$, a \emph{lifting of $E$ to $V^T$} is defined to be a multitensor $(E',u',\sigma')$ on $V^T$ such that $U^TE'=EM(U^T)$, $U^Tu'=uU^T$ and $U^T\sigma'=\sigma M^2(U^T)$, where we recall that $U^T:V^T \to V$ is the forgetful functor. In more explicit terms to give such a lifting is to give maps
\[ a_{x_i} : T\opE\limits_i X_i \to \opE\limits_i X_i \]
in $V$ for all sequences $((X_1,x_1),...,(X_n,x_n))$ of $T$-algebras, such that these maps satisfy the axioms making $(\opE\limits_i X_i,a_{x_i})$ $T$-algebras, and with respect to these structures, $\opE\limits_i f_i$ is a morphism of $T$-algebras for any sequence $f_i:(X_i,x_i) \to (Y_i,y_i)$ of $T$-algebra maps, and moreover, $u$ and $\sigma$ are $T$-algebra morphisms.
\begin{theorem}\label{thm:opmonoidal-monad-dist-law}
Let $V$ be a category, $(E,u,\sigma)$ be a multitensor on $V$ and $(T,\eta,\mu)$ be a monad on $V$. Then there is a bijection between the following types of data:
\begin{enumerate}
\item  Morphisms $\tau_{X_i} : T\opE\limits_i X_i \to \opE\limits_i TX_i$
of $V$ providing the coherences making $T$ into a opmonoidal monad.\label{data:opmonoidal-monad-coherence}
\item  Morphisms $\sigma_{X_{ij}}' : \opE\limits_iT\opE\limits_j TX_{ij} \to \opE\limits_{ij} TX_{ij}$ of $V$ providing the substitutions for a multitensor $(EM(T),u',\sigma')$ where $u_X'=u_{TX}\eta_X$, $E \eta:E \to EM(T)$ is a multitensor map, $uT:T \to E_1T$ is a monad map, and the composite
\[ \xygraph{!{0;(2,0):} {\opE\limits_iT}="l" [r] {\opE\limits_iTE_1T}="m" [r] {\opE\limits_iT}="r" "l":"m"^-{\opE\limits_i\eta uT}:"r"^-{\sigma'}} \]
is the identity.\label{data:composite-multitensor-2}
\item  Liftings of the multitensor $E$ to a multitensor $E'$ on $V^T$.\label{data:multitensor-lifting}
\setcounter{continued-numbering}{\value{enumi}}
\end{enumerate}
and for any given instance of such data one has an isomorphism
\[ \Enrich{EM(T)} \iso \Enrich{E'} \]
of categories commuting with the forgetful functors into $\ca GV$.
If in addition $V$ has coproducts, $T$ preserves them and $E$ is distributive, then the data (\ref{data:opmonoidal-monad-coherence})-(\ref{data:composite-multitensor-2}) are also in bijection with
\begin{enumerate}
\setcounter{enumi}{\value{continued-numbering}}
\item  Identity on object morphisms $\lambda_X:\ca G(T)\Gamma(E)(X) \to \Gamma(E)\ca G(T)(X)$ of $\ca GV$ providing a monad distributive law of $\Gamma(E)$ over $\ca G(T)$.\label{data:monad-distributive-law-2}
\end{enumerate}
\end{theorem}
\begin{proof}
(\ref{data:opmonoidal-monad-coherence}){$\iff$}(\ref{data:composite-multitensor-2}): This is completely analogous to the bijection between (\ref{data:monoidal-monad-coherence}) and (\ref{data:composite-multitensor}) in theorem(\ref{thm:monoidal-monad-dist-law}).

(\ref{data:opmonoidal-monad-coherence}){$\iff$}(\ref{data:multitensor-lifting}): Suppose that opmonoidal monad coherence data $\tau_{X_i}$ is given. Then for a given sequence of $T$-algebras $((X_1,x_1),...,(X_n,x_n))$ we define $a_{x_i}$ to be given by the composite
\[ \xygraph{!{0;(2,0):} {T\opE\limits_i X_i}="l" [r] {\opE\limits_iTX_i}="m" [r] {\opE\limits_i X_i.}="r" "l":"m"^-{\tau}:"r"^-{\opE\limits_i X_i}} \]
The verifications that these maps satisfy the $T$-algebra axioms, and that with respect to these structures $u$ and $\sigma$ are $T$-algebra morphisms are straight forward. Conversely given a lifting of $E$ to $V^T$ we construct, for a given sequence $(X_1,...,X_n)$ of objects of $V$, the coherence map $\tau_{X_i}$ as the composite
\[ \xygraph{!{0;(2,0):} {T\opE\limits_i X_i}="l" [r] {T\opE\limits_iTX_i}="m" [r] {\opE\limits_iX_i.}="r" "l":"m"^-{T\opE\limits_i\eta}:"r"^-{a_{\mu_{X_i}}}} \]
The axioms expressing the compatibilities between $\tau$ with $u$, $\eta$ and $\mu$ are all routinely verified. The axiom for the compatibility of $\tau$ with $\sigma$ is witnessed in
\[ \xygraph{!{0;(1.5,0):(0,.6667)::} {T\opE\limits_i\opE\limits_jX_{ij}}="tl" [r(2)] {T\opE\limits_iT\opE\limits_jX_{ij}}="tm" [r(2)] {\opE\limits_iT\opE\limits_jX_{ij}}="tr" [d(2)] {\opE\limits_iT\opE\limits_jTX_{ij}}="mr" [d(2)] {\opE\limits_i\opE\limits_jTX_{ij}}="br" [l(2)d] {\opE\limits_{ij}TX_{ij}}="b" [l(2)u] {T\opE\limits_{ij}TX_{ij}}="bl" [u(2)] {T\opE\limits_{ij}X_{ij}}="ml" [r(1.25)d] {T\opE\limits_i\opE\limits_jTX_{ij}}="il" [u(1.5)r(.75)] {T\opE\limits_iT\opE\limits_jTX_{ij}}="im" [d(1.5)r(.75)] {T\opE\limits_i\opE\limits_jTX_{ij}}="ir"
"tl":"tm"^-{T\opE\limits_i\eta\opE\limits_j}:"tr"^-{a_{\mu_{\opE\limits_jX_{ij}}}}:"mr"^-{\opE\limits_iT\opE\limits_j\eta}:"br"^-{\opE\limits_ia_{\mu_{X_{ij}}}}:"b"^-{\sigma T}:@{<-}"bl"^-{a_{\mu_{X_{ij}}}}:@{<-}"ml"^-{T\opE\limits_{ij}\eta}:@{<-}"tl"^-{T\sigma} "il":"im"^-{T\opE\limits_i\eta\opE\limits_jT}:"ir"^-{T\opE\limits_ia_{\mu_{X_{ij}}}}:@{<-}"il"^-{1} "tl":"il"^-{T\opE\limits_i\opE\limits_j\eta}:"bl"^-{T\sigma T} "tm":"im"^-{T\opE\limits_iT\opE\limits_j\eta}:"mr"^-{a_{\mu_{\opE\limits_jTX_{ij}}}} "ir":"br"^-{a_{a_{\mu_{X_{ij}}}}}
"im":@{}"tr"|-{(I)} "ir":@{}"mr"|-{(II)} "bl":@{}"br"|-{(III)}} \]
in which the unlabelled regions commute for obvious reasons. Region (III) commutes since $\sigma$ is a $T$-algebra map. Since the $a_{\mu_{X_{ij}}}$ are $T$-algebra maps so is $\opE\limits_i a_{\mu_{X_{ij}}}$ and so region (II) commutes. The morphisms $T\opE\limits_j\eta$ are $T$-algebra maps and so $\opE\limits_iT\opE\limits_j\eta$ are $T$-algebra maps whence region (I) commutes also. Thus we have established functions that turn opmonoidal coherence data into liftings and vice versa, and the verification that these are inverse to each other is straight forward.

The isomorphism $\Enrich{E'} \iso \Enrich{(EM(T))}$ over $\ca GV$: To give $X \in \ca GV$ the structure of an $EM(T)$-category is to give morphisms
\[ a_{x_i} : \opE\limits_i T(X(x_{i-1},x_i)) \to X(x_0,x_n) \]
in $V$ for each sequence $(x_0,...,x_n)$ of objects of $X$, such that all diagrams of the form
\[ \xygraph{{\xybox{\xygraph{!{0;(2,0):(0,.5)::} {X(x_0,x_1)}="tl" [r] {TX(x_0,x_1)}="tr" [d] {E_1TX(x_0,x_1)}="br" [l] {X(x_0,x_1)}="bl" "tl":"tr"^-{\eta}:"br"^-{uT}:"bl"^-{a_{x_0,x_1}}:@{<-}"tl"^-{1}}}} [r(5)]
{\xybox{\xygraph{!{0;(3,0):(0,.3333)::} {\opE\limits_iT\opE\limits_jTX(x_{ij-1},x_{ij})}="tl" [r] {\opE\limits_i\opE\limits_jT^2X(x_{ij-1},x_{ij})}="tr" [d] {\opE\limits_{ij}TX(x_{ij-1},x_{ij})}="br" [l] {\opE\limits_iTX(x_{i-1},x_i)}="bl" [r(.5)d] {X(x_0,x_n)}="b" "tl":"tr"^-{\opE\limits_i\tau T}:"br"^-{\sigma \mu}:"b"^-{a_{x_{ij}}}:@{<-}"bl"^-{a_{a_i}}:@{<-}"tl"^-{\opE\limits_iTa_{x_{ij}}}}}}} \]
commute. On the hand to give $X$ the structure of an $E'$-category is to give maps
\[ \begin{array}{lccr} {b_{x_0,x_1}:TX(x_0,x_1) \to X(x_0,x_1)} &&& {c_{x_i}:\opE\limits_iX(x_{i-1},x_i) \to X(x_0,x_n)} \end{array} \]
such that the $b_{x_0,x_1}$ satisfy the $T$-algebra axioms, the $c_{x_i}$ satisfy the $E$-category axioms, and moreover the $c_{x_i}$ are $T$-algebra morphisms with respect to the $T$-algebra structures given by the $b_{x_0,x_1}$. Supposing that maps $a_{x_i}$ are given as above one defines maps $b_{x_0,x_1}$ and $c_{x_i}$ as composites
\[ \xygraph{{\xybox{\xygraph{!{0;(3,0):} {TX(x_0,x_1)}="l" [r] {E_1TX(x_0,x_1)}="m" [r] {X(x_0,x_1)}="r" "l":"m"^-{uT}:"r"^-{a_{x_0,x_1}}}}} [d(.5)]
{\xybox{\xygraph{!{0;(3,0):} {\opE\limits_i X(x_{i-1},x_i)}="l" [r] {\opE\limits_i TX(x_{i-1},x_i)}="m" [r] {X(x_0,x_n).}="r" "l":"m"^-{\opE\limits_i \eta}:"r"^-{a_{x_i}}}}}} \]
The $T$-algebra axioms for $b_{x_0,x_1}$ and the $E$-category axioms for $c_{x_i}$ are easily verified. That $c_{x_i}$ is a $T$-algebra morphism is expressed by the commutativity of the outside of the diagram
\[ \xygraph{!{0;(1.25,0):(0,1.25)::} {T\opE\limits_iX(x_{i-1},x_i)}="tll" [r(1.5)u] {\opE\limits_iTX(x_{i-1},x_i)}="tl" [r(3)] {\opE\limits_iE_1TX(x_{i-1},x_i)}="tr" [r(1.5)d] {\opE\limits_iX(x_{i-1},x_i)}="trr" [d(2)] {\opE\limits_iTX(x_{i-1},x_i)}="mr" [d(2)] {X(x_0,x_n)}="br" [l(3)d] {E_1TX(x_0,x_n)}="bm" [l(3)u] {TX(x_0,x_n)}="bl" [u(2)] {T\opE\limits_iTX(x_{i-1},x_i)}="ml" [u(1.5)r] {\opE\limits_iT^2X(x_{i-1},x_i)}="itl" [r(2)u(.5)] {\opE\limits_iT^2X(x_{i-1},x_i)}="itm" [r(2)d] {\opE\limits_iTE_1TX(x_{i-1},x_i)}="itr" [d(1.25)l(.75)] {\opE\limits_iE_1T^2X(x_{i-1},x_i)}="imr" [d] {\opE\limits_iTX(x_{i-1},x_i)}="ibr" [l(2.25)u(1.5)] {E_1\opE\limits_iT^2X(x_{i-1},x_i)}="ibm" [d(1.75)l(.5)] {E_1T\opE\limits_iTX(x_{i-1},x_i)}="ibl"
"tll":"tl"^-{\tau}:"tr"^-{\opE\limits_iuT}:"trr"^-{\opE\limits_ia_{x_{i-1},x_i}}:"mr"^-{\opE\limits_i\eta}:"br"^-{a_{x_i}}:@{<-}"bm"^-{a_{x_0,x_n}}:@{<-}"bl"^-{uT}:@{<-}"ml"^-{Ta_{x_i}}:@{<-}"tll"^-{T\opE\limits_i\eta}
"ml":"itl"|-{\tau T}:@{<-}"tl"|-{\opE\limits_i T\eta}:"itm"^-{\opE\limits_i \eta T}:"itr"^-{\opE\limits_i TuT}:"imr"|-{\opE\limits_i\tau T}:"ibr"|-{\sigma\mu}:@{<-}"ibm"|-{\sigma\mu}:@{<-}"ibl"|-{E_1\tau T}:@{<-}"ml"|-{uT\opE\limits_iT}
"itl":"ibm"^-{u\opE\limits_iT^2} "tl":@{.>}@/_{1pc}/"ibr"|-{\id} "itm":"imr"_-{\opE\limits_iuT^2} "tr":"itr"_-{\opE\limits_i\eta E_1T}:"mr"|-{\opE\limits_iTa_{x_{i-1},x_i}} "ibr":"br"_-{a_{x_i}} "ibl":"bm"^-{E_1Ta_{x_i}}} \]
whose internal regions all clearly commute. Conversely given structure maps $b_{x_0,x_1}$ and $c_{x_i}$ as above, one constructs the $a_{x_i}$ as the composites
\[ \xygraph{!{0;(3,0):} {\opE\limits_i TX(x_{i-1},x_i)}="l" [r] {\opE\limits_iX(x_{i-1},x_i)}="m" [r] {X(x_0,x_n).}="r" "l":"m"^-{\opE\limits_ib_{x_{i-1},x_i}}:"r"^-{c_{x_i}}} \]
The verification that these maps satisfy the axioms on the $a_{x_i}$ described above is straight forward. It is also straight forward to check that these constructions give a bijection between $EM(T)$-category structures and $E'$-category structures on $X$, and moreover that this can be extended to maps giving the required isomorphism of categories over $\ca GV$.

(\ref{data:opmonoidal-monad-coherence}){$\iff$}(\ref{data:monad-distributive-law-2}): This bijection is given in basically the same way as that for (\ref{data:monoidal-monad-coherence}){$\iff$}(\ref{data:monad-distributive-law}) in theorem(\ref{thm:monoidal-monad-dist-law}), using the 2-functor $\Psi'$ instead of $\Psi$.
\end{proof}
\begin{remark}\label{rem:distributive-opmonoidal-monad-elaboration}
Note in particular that, in the context of theorem(\ref{thm:opmonoidal-monad-dist-law}), when $V$ has coproducts, $E$ is distributive and $T$ preserves coproducts, then the composite multitensor $EM(T)$ is clearly distributive. Moreover since $U^TE'=EM(U^T)$ and $U^T$ creates coproducts, the lifted multitensor $E'$ is also distributive. Thus by theorem(\ref{thm:opmonoidal-monad-dist-law}) and proposition(\ref{prop:Gamma-E-algebras-as-E-categories}) one in fact has isomorphisms
\[ \ca G(V)^{\Gamma(EM(T))} \iso \Enrich{EM(T)} \iso \ca G(V)^{\Gamma(E)\ca G(T)} \iso \Enrich{E'} \iso \ca G(V^T)^{\Gamma E'} \]
over $\ca GV$. Either by a direct verification, or by applying structure-semantics{\footnotemark{\footnotetext{This is the well-known fact due to Lawvere that for any category $\ca E$, the canonical functor
\[ \begin{array}{lccccc} {\tn{Mnd}(\ca E)^{\op} \to \CAT/\ca E} &&& {T} & {\mapsto} & {U:\ca E^T \to \ca E} \end{array} \]
with object map indicated is fully faithful (see \cite{Street-FTM} for a proof). In particular this implies that for monads $S$ and $T$ on $\ca E$, an isomorphism $\ca E^T \iso \ca E^S$ over $\ca E$ is the same thing as a monad isomorphism $S \iso T$.}}} to $\ca G(V)^{\Gamma(EM(T))} \iso \ca G(V)^{\Gamma(E)\ca G(T)}$, one has also an isomorphism
\[ \Gamma(EM(T)) \iso \Gamma(E)\ca G(T) \]
of monads. Moreover the monad $T$ may also be regarded as a multitensor, whose unary part is $T$ and whose $n$-ary parts for $n \neq 1$ are constant at $\emptyset$, and then $\Gamma T \iso \ca GT$ as monads. Thus if $E$ and $T$ are $\lambda$-accessible, then so is $EM(T)$ by theorem(\ref{thm:preservation-by-Gamma-E})(\ref{thmcase:Gamma-E-pres-filtered-colims}).
\end{remark}
\begin{remark}\label{rem:EMT-lra}
If moreover $E$ and $T$ are local right adjoint and the coherences $\tau_{X_i}$ are cartesian natural in the $X_i$, then by the explicit description of the composite multitensor $EM(T)$ and theorem(\ref{thm:preservation-by-Gamma-E})(\ref{thmcase:Gamma-E-lra}), $EM(T)$ is also local right adjoint.
\end{remark}
\begin{example}\label{ex:Tcross}
When $V$ is a category with finite products and $E$ is given by them, any monad $T$ on $V$ is canonically opmonoidal, with the coherences provided by the product preservation obstruction maps. The composite monad, whose tensor product is
\[ EM(T)(X_1,...,X_n) = \prod\limits_{1 \leq i \leq n} TX_i \]
was called $T^{\times}$ in \cite{BataninWeber-EnHop}. Proposition(2.8) of \cite{BataninWeber-EnHop}, which says that
\[ \Enrich{T^{\times}} \iso \Enrich{(V^T,\times)} \]
over $\ca GV$, follows by applying theorem(\ref{thm:opmonoidal-monad-dist-law}) to this example. When $T$ is local right adjoint and $V$ is distributive, the product obstruction maps for $T$ are cartesian natural by lemma(2.15) of \cite{Weber-Fam2fun} and so by remark(\ref{rem:EMT-lra}) $T^{\times}$ is local right adjoint.
\end{example}

\subsection{Wreath products}
\label{ssec:wreath}
Continuing with this last example, one can give an account of the wreath products with $\Delta$ which are central to \cite{Berger-IteratedWreathProduct}, by combining the present discussion with the theory of monads with arities as described in \cite{BergMellWeber-MonadsArities}. Let us first recall and in some ways update some of this theory.

A \emph{monad with arities} is a monad in the 2-category $\CatAr$ of \emph{categories with arities} that we now describe. An object is a fully faithful dense functor $i:\ca A \to \ca E$ such that $\ca A$ is small and $\ca E$ is locally small. A morphism from $i:\ca A \to \ca E$ to $j:\ca B \to \ca F$ is a functor $F: \ca E \to \ca F$ such that the composite functor
\[ \xygraph{{\ca E}="p0" [r] {\ca F}="p1" [r(1.2)] {\PSh {\ca B}}="p2" "p0":"p1"^-{F}:"p2"^-{\ca F(j,1)}} \]
is the left kan extension along $i$ of $\ca F(j,1)Fi$. See lemma(2.2) and section(2.4) of \cite{BergMellWeber-MonadsArities} to see why these compose, and for an elementary characterisation of this last condition. A 2-cell $F \to G$ is just a natural transformation.

In \cite{BergMellWeber-MonadsArities} $i:\ca A \to \ca E$ was referred to as a pair $(\ca A,\ca E)$ and assumed to be the inclusion of a full subcategory. Thus strictly speaking, our $\CatAr$ has more objects than that of \cite{BergMellWeber-MonadsArities}. However any $i:\ca A \to \ca E$ in $\CatAr$ can easily be seen to be isomorphic to one for which the functor is a subcategory inclusion. The objects in the image of $i$ are small projective iff $\ca E(i,1)$ is an equivalence $\ca E \catequiv \PSh {\ca A}$ (see the discussion preceeding lemma(\ref{lem:GV-dense})). Working with such an $i$ is a more flexible alternative than working with categories which are equal or isomorphic to $\PSh {\C}$ for some small category $\C$, for instance, when dealing with $\ca G(\PSh {\C})$.
\begin{remark}\label{rem:arities-enriched-graphs}
By lemma(\ref{lem:GV-dense}), if $V$ has a strict initial object and $i:\ca D \to V$ is a category with arities, then so is $i^+:\ca D_+ \to \ca GV$. Moreover if the objects in the image of $i$ are small projective, then so are those in the image of $i^+$.
\end{remark}
\begin{remark}\label{rem:arities-comma}
If $i:\ca A \to \ca E$ is a category with arities and $X \in \ca E$, then it is straight forward to verify that
\[ \begin{array}{lccr} {i_X:i/X \to \ca E/X} &&& {(A,h:iA \to X) \, \mapsto \, (iA,h)} \end{array} \]
is also a category with arities, and moreover if the objects in the image of $i$ are small projective, then so are those in the image of $i_X$.
\end{remark}
Let us now reformulate Berger's definition of the wreath product with $\Delta$ (\cite{Berger-IteratedWreathProduct} definition(3.1)) in terms of the language of this paper. Given a distributive category $V$, for any sequence of objects $(Z_1,...,Z_n)$ regarded as a $V$-graph, the $V$-graph $\Gamma (\prod)(Z_1,...,Z_n)$, underlying the free $V$-category on $(Z_1,...,Z_n)$, has the following explicit description by the definition of $\Gamma(\prod)$. Its set of objects is $\{0,...,n\}$, and its homs are given by
\[ \Gamma(\textstyle{\prod})(Z_1,...,Z_n)(i,j) =
\left\{\begin{array}{cll} {\prod\limits_{i{<}k{\leq}j} Z_k} && {\tn{if} \,\, i \leq j} \\ {\emptyset} && {\tn{if} \,\, i > j.} \end{array}\right.  \]
\begin{definition}\label{def:wreath}
\cite{Berger-IteratedWreathProduct}
Let $i:\ca A \to \ca E$ be in $\CatAr$ such that $\ca E$ is a distributive category. Then $\Delta \wr \ca A$ is the following category
\begin{itemize}
\item objects are finite sequences of objects of $\ca A$.
\item an arrow $f:(A_1,...,A_m) \to (B_1,...,B_n)$ is an $\ca E$-functor
\[ f:\Gamma(\textstyle{\prod})(iA_1,...,iA_m) \to \Gamma(\textstyle{\prod})(iB_1,...,iB_n).  \]
\end{itemize}
\end{definition}
The enrichment over $\ca E$ implicit in the above is with respect to cartesian product in $\ca E$. Suppose the objects of $\ca A$ are not initial and $\ca E$'s initial object is strict. Then the object map $\phi:\{0,...,m\} \to \{0,...,n\}$ of $f$ above is forced to be an order preserving function. As $f$ amounts to an $\ca E$-graph morphism as on the left in
\[ \begin{array}{lccr} {(iA_1,...,iA_m) \to \Gamma(\textstyle{\prod})(iB_1,...,iB_n)} &&& {f_j:iA_j \longrightarrow \prod\limits_{\phi(j-1)<k{\leq}\phi(j)} iB_k} \end{array} \]
it is completely determined by $\phi$ and the morphisms $f_j$ in $\ca E$ for $1 \leq j \leq m$ as indicated in the previous display. In the case where $\ca E=\PSh {\ca A}$ and $i$ is the yoneda embedding, $\Delta \wr \ca A$ is thus exactly as defined in \cite{Berger-IteratedWreathProduct} definition(3.1). In fact our definition is no more general than that of Berger's.
\begin{remark}\label{rem:reconcile-Berger-wreath}
$\Delta \wr \ca A$ as defined in definition(\ref{def:wreath}) does not depend on $i$ or $\ca E$. For $\ca E(i,1):\ca E \to \PSh {\ca A}$ is fully faithful and product preserving, and so $\ca G(\ca E(i,1))$ is also fully faithful and underlies a monad morphism $(\ca G(\ca E),\Gamma(\prod)) \to (\ca G(\PSh {\ca A}),\Gamma(\prod))$ whose 2-cell datum is invertible. Thus the corresponding commutative square
\[ \xygraph{!{0;(2,0):(0,.5)::} {\ca G(\ca E)^{\Gamma(\prod)}}="p0" [r] {\ca G(\PSh {\ca A})^{\Gamma(\prod)}}="p1" [d] {\ca G(\PSh {\ca A})}="p2" [l] {\ca G(\ca E)}="p3" "p0":"p1"^-{\overline{i}}:"p2"^-{U^{\Gamma(\prod)}}:@{<-}"p3"^-{\ca G(\ca E(i,1))}:@{<-}"p0"^-{U^{\Gamma(\prod)}}} \]
is a pullback by \cite{BergMellWeber-MonadsArities} proposition(1.3)(c), and so $\overline{i}$ is also fully faithful. Moreover for any sequence $(A_1,...,A_n)$ of objects of $\ca A$ one has
\[ \overline{i} \, \Gamma(\textstyle{\prod})(iA_1,...,iA_n) \iso \Gamma(\textstyle{\prod})(yA_1,...,yA_n) \]
where $y$ is the yoneda embedding $\ca A \to \PSh {\ca A}$. Thus the application of $\overline{i}$ gives a bijection between $\ca E$-functors $\Gamma(\prod)(iA_1,...,iA_n) \to \Gamma(\prod)(iB_1,...,iB_n)$ and $\PSh {\ca A}$-functors $\Gamma(\prod)(yA_1,...,yA_n) \to \Gamma(\prod)(yB_1,...,yB_n)$.
\end{remark}
As defined in definition(\ref{def:wreath}), the category $\Delta \wr \ca A$ is the image of the identity on objects-fully faithful factorisation of the composite
\[ \xygraph{!{0;(1.5,0):(0,1)::} {M\ca A}="p0" [r] {M\ca E}="p1" [r] {\ca G\ca E}="p2" [r(1.5)] {\Enrich {\ca E}=\ca G(\ca E)^{\Gamma(\prod)}}="p3" "p0":"p1"^-{Mi}:"p2"^-{\tn{seq}_{\ca E}}:"p3"^-{F^{\Gamma(\prod)}}} \]
where $\tn{seq}_{\ca E}$ is the process of viewing sequences as $\ca E$-graphs.

Let $i:\ca A \to \ca E$ be a category with arities and suppose that $\ca E$ has a terminal object $1$. Given a local right adjoint monad{\footnotemark{\footnotetext{In \cite{BergMellWeber-MonadsArities} local right adjoint monads were called strongly cartesian monads.}}} $T$ on $\ca E$, we shall now explain how one can extend the arities $\ca A$ of $\ca E$ in such a way as to make $T$ a monad with arities. In \cite{BergMellWeber-MonadsArities} this was explained in section(2.6) in the case where $i$ was a yoneda embedding.

Since $T$ is local right adjoint its effect $T_1:\ca E \to \ca E/T1$ on arrows $X \to 1$ has a left adjoint $L_T$. Define the full and faithful functor $i_0:\ca A_T \to \ca E$ as the right part of the identity on objects-fully faithful factorisation of the composite
\[ \xygraph{!{0;(1.5,0):(0,1)::} {i/T1}="p0" [r] {\ca E/T1}="p1" [r] {\ca E.}="p2" "p0":"p1"^-{i_{T1}}:"p2"^-{L_T}} \]
Given $f:iA \to T1$ and writing $g_{(A,f)}$ for the component of the unit of $L_T \ladj T_1$ at $(A,f)$
\[ \xygraph{!{0;(2,0):(0,1)::} {iA}="p0" [r] {TL_Ti_{T1}(A,f)}="p1" [r] {T1}="p2" "p0":"p1"^-{g_{(A,f)}}:"p2"^-{T!}} \]
is a $T$-generic factorisation of $f$ in the sense of \cite{Weber-Fam2fun,Weber-Generic}. Thus the data of $L_Ti_{T1}$ comes down to a choice of such generic factorisation for each $(A,f)$. Since generics for the identity are exactly isomorphisms, the cartesianness of the unit $\eta:1 \to T$ ensures that its components are $T$-generic by \cite{Weber-Generic} proposition(5.10)(2). As in the discussion of section(6.1) of \cite{BataninWeber-EnHop}, one can thus take $g_{(A,f)}$ for $f$ of the form
\[ \xygraph{{iA}="p0" [r] {1}="p1" [r] {T1}="p2" "p0":"p1"^-{!}:"p2"^-{\eta_1}} \]
to be $\eta_{iA}$, whence $L_Ti_{TA}(A,f)=iA$ in this case. This ensures that $i$ factors through $i_0$, and so by theorem(5.13) of \cite{Kelly-EnrichedCatsBook}, $i_0$ is dense. By the same argument as theorem(2.9) of \cite{BergMellWeber-MonadsArities}, $\ca A_T$ are arities for $T$, that is to say $T$ is a monad on $i_0:A_T \to \ca E$ in $\CatAr$.

Define $i_T:\Theta_T \to \ca E^T$ to be the right part of the identity on objects-fully faithful factorisation of the composite
\[ \xygraph{{\ca A_T}="p0" [r] {\ca E}="p1" [r] {\ca E^T}="p2" "p0":"p1"^-{i_0}:"p2"^-{F^T}} \]
By the nerve theorem (\cite{BergMellWeber-MonadsArities} theorem(1.10)), $i_T$ is also dense and one has a characterisation of the image of the nerve functor $\ca E^T(i_T,1):\ca E^T \to \PSh {\Theta_T}$.

The basic example worth recalling is where $\ca E = \Graph$, $i$ is the yoneda embedding and $T$ is the monad for categories. Then $\Theta_T$ is equivalent to $\Delta$, though not isomorphic -- unwinding the definitions in this case reveals that $\Theta_T$ differs from $\Delta$ in that there are two copies of the object $[0]$. Thus to recover $\Delta$ up to isomorphism from the above considerations, one must take a skeleton of $\Theta_T$. Similar remarks apply to the other examples considered in section(4) of \cite{Weber-Fam2fun}.

Let $i:\ca A \to \ca E$ be a category with arities, $\ca E$ be locally c-presentable, and $T$ be a coproduct preserving local right adjoint monad on $\ca E$. Then by remark(\ref{rem:arities-enriched-graphs}) and theorem(\ref{thm:preservation-by-Gamma-E})(\ref{thmcase:Gamma-E-lra}), $i^+:\ca A_+ \to \ca G\ca E$ is a category with arities, $\ca G\ca E$ is locally c-presentable, and $\Gamma(T^{\times})$ is a coproduct preserving local right adjoint monad on $\ca G\ca E$. Thus as above, one can define $\Theta_{\Gamma(T^{\times})}$. Recall from section(\ref{sec:dist-multitensor-over-monad}) that $\Gamma(T^{\times})$ is the composite monad $\Gamma(\prod)\ca G(T)$ defined via a distributive law.
\begin{proposition}\label{prop:gen-wreath-formula}
For $i$, $\ca E$ and $T$ as above there exists a fully faithful functor
\[ w_T : \Theta_{\Gamma(T^{\times})} \longrightarrow \Delta \wr \Theta_T.  \]
\end{proposition}
\begin{proof}
The category $\Delta \wr \Theta_T$ is the image of the identity on objects-fully faithful factorisation of the composite
\begin{equation}\label{eq:comp-for-Delta-wr-iT}
\xygraph{!{0;(1.5,0):(0,1)::} {M\Theta_{i,T}}="p0" [r] {M\ca E^T}="p1" [r] {\ca G(\ca E^T)}="p2" [r(2)] {\ca G(\ca E^T)^{\Gamma(\prod)}=\ca G(E)^{\Gamma(T^{\times})}}="p3" "p0":"p1"^-{Mi_T}:"p2"^-{\tn{seq}_{\ca E^T}}:"p3"^-{F^{\Gamma(\prod)}}} \end{equation}
and the category $\Theta_{\Gamma(T^{\times})}$ is the image of the identity on objects-fully faithful factorisation of the composite
\begin{equation}
\label{eq:comp-for-Theta-Gamma-Tcross}
\xygraph{!{0;(2,0):(0,1)::} {i^+/\Gamma(T^{\times})(1)}="p0" [r(1.5)] {\ca G(\ca E)/\Gamma(T^{\times})(1)}="p1" [r] {\ca G(\ca E)}="p2" [r] {\ca G(\ca E)^{\Gamma(T^{\times})}.}="p3" "p0":"p1"^-{i^+_{\Gamma(T^{\times})(1)}}:"p2"^-{L_{\Gamma(T^{\times})}}:"p3"^-{F^{\Gamma(T^{\times})}}}
\end{equation}
Thus one has fully faithful functors
\[ \begin{array}{lccr} {i_1:\Delta \wr \Theta_T \longrightarrow \ca G(\ca E)^{\Gamma(T^{\times})}} &&& {i_2:\Theta_{\Gamma(T^{\times})} \longrightarrow \ca G(\ca E)^{\Gamma(T^{\times})}.} \end{array} \]
To demonstrate the existence of $w_T$ such that $i_2 = i_1w_T$, it suffices to show that if $X \in \ca G(\ca E)^{\Gamma(T^{\times})}$ is isomorphic to an object in the image of $i_2$, then $X$ is isomorphic to an object in the image of $i_1$.

To say that $X \iso X' \in \tn{im}(i_2)$ is to say that $X \iso \Gamma(T^{\times})(Y)$ and there exists $B \in \ca A_+$ together with a generic morphism $g:B \to \Gamma(T^{\times})(Y)$, because then $Y$ will have been obtained by applying the first two functors of (\ref{eq:comp-for-Theta-Gamma-Tcross}) to $f=\Gamma(T^{\times})(!)g$. In other words, $X \iso \Gamma(T^{\times})(Y)$ where $Y$ was obtained by generically factoring some morphism $f:B \to \Gamma(T^{\times})(1)$, where $B \in \ca A_+$. Such generic factorisations were understood in the proof of theorem(\ref{thm:preservation-by-Gamma-E})(\ref{thmcase:Gamma-E-lra}). When $B=0$ one may take $g$ to be $\eta_0$. On the other hand when $B=(A)$, one may take $g:(A) \to \Gamma(T^{\times})(Z_1,...,Z_n)$ such that $g0=0$, $g1=n$ and $g_{0,1}:A \to \prod\limits_{1 \leq i \leq n} TZ_i$ such that $\tn{pr}_ig_{0,1}$ is $T$-generic for all $i$. Thus $X \iso \Gamma(T^{\times})(Y)$ where $Y=(p_1,...,p_n)$, and there exists $A \in \ca A$ and $T$-generic maps $g_i:A \to Tp_i$ for $1 \leq i \leq n$ (the case $B=0$ captured by the case $n=0$).

To say that $X \iso X' \in \tn{im}(i_1)$ is to say that $X \iso \Gamma(\prod)(q_1,...,q_n)$ where the $q_i$ are isomorphic to objects in the image of $i_T$. But this says that for all $i$, $q_i \iso Tp_i$ and there exists $A_i \in \ca A$ and a $T$-generic morphism $g_i:A_i \to Tp_i$. Thus $X \iso \Gamma(Tp_1,...,Tp_n)$ and for all $i$, there exists $A_i \in \ca A$ and a $T$-generic morphism $g_i:A_i \to Tp_i$. Now $(Tp_1,...,Tp_n)$ is a sequence of free $T$-algebras viewed as an $\ca E^T$-graph, and since $T$ preserves coproducts this can be rewritten as $\ca G(T)(p_1,...,p_n)$. Thus $X \iso \Gamma(T^{\times})(Y)$ where $Y=(p_1,...,p_n)$, and for $1 \leq i \leq n$ there exists $A_i \in \ca A$ and $T$-generic maps $g_i:A_i \to Tp_i$ for $1 \leq i \leq n$. 
\end{proof}
\begin{remark}\label{rem:condition-for-wreath-equivalence}
From the previous proof, it is clear that $w_T$ is essentially surjective on objects when $T$ and $i$ satisfy the following condition -- given $A_1, A_2 \in \ca A$, and generics $g_1:A_1 \to TX$ and $g_2:A_2 \to TY$, there exists $A \in \ca A$ and generics $g_1':A \to TX$ and $g_2':A \to TY$.
\end{remark}
\begin{example}\label{ex:wreath-semicat}
Consider the case when $\ca E=\Graph$, $i$ is the yoneda embedding and $T$ is the monad for semicategories. Recall that a semicategory is a graph with an associative binary composition, but not necessarily identities for this composition. Thus $TX$ is the graph whose vertices are those of $X$, and whose edges are \emph{non-empty} paths. In this case $\Theta_T$ is the full subcategory of $\Delta$ consisting of the strictly monotone functions (ie the injections). Thus an object of $\Delta \wr \Theta_T$ is a sequence of finite non-empty ordinals $([n_1],...,[n_k])$. As a subcategory of $\ca G^2(\Set)^{\Gamma(T^{\times})}$ the objects of $\Delta \wr \Theta_T$ are free $\Gamma(T^{\times})$ algebras on 2-dimensional globular pasting diagrams. For instance  $([3],[0],[2])$ and $([3],[1],[2])$ are identified with the globular pasting diagrams
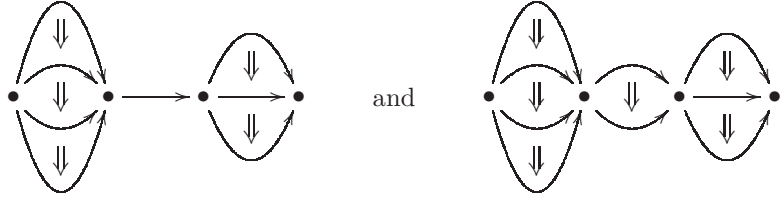
\[ \xygraph{{\xybox{\xygraph{{\bullet}="p0" [r] {\bullet}="p1" [r] {\bullet}="p2" [r] {\bullet}="p3"
"p0":@/^{3pc}/"p1"|-{}="q10" "p0":@/^{1pc}/"p1"|-{}="q11" "p0":@/_{1pc}/"p1"|-{}="q12" "p0":@/_{3pc}/"p1"|-{}="q13" "p1":"p2" "p2":@/^{2pc}/"p3"|-{}="q30" "p2":"p3"|-{}="q31" "p2":@/_{2pc}/"p3"|-{}="q32"
"q10":@{}"q11"|(.5){}="m11" "m11" [u(.15)] :@{=>}[d(.3)]
"q11":@{}"q12"|(.5){}="m12" "m12" [u(.15)] :@{=>}[d(.3)]
"q12":@{}"q13"|(.5){}="m13" "m13" [u(.15)] :@{=>}[d(.3)]
"q30":@{}"q31"|(.5){}="m31" "m31" [u(.15)] :@{=>}[d(.3)]
"q31":@{}"q32"|(.5){}="m32" "m32" [u(.15)] :@{=>}[d(.3)]}}}
[r(2.5)] {\tn{and}} [r(2.5)]
{\xybox{\xygraph{{\bullet}="p0" [r] {\bullet}="p1" [r] {\bullet}="p2" [r] {\bullet}="p3"
"p0":@/^{3pc}/"p1"|-{}="q10" "p0":@/^{1pc}/"p1"|-{}="q11" "p0":@/_{1pc}/"p1"|-{}="q12" "p0":@/_{3pc}/"p1"|-{}="q13" "p1":@/^{1pc}/"p2"|-{}="q20" "p1":@/_{1pc}/"p2"|-{}="q21"
"p2":@/^{2pc}/"p3"|-{}="q30" "p2":"p3"|-{}="q31" "p2":@/_{2pc}/"p3"|-{}="q32"
"q10":@{}"q11"|(.5){}="m11" "m11" [u(.15)] :@{=>}[d(.3)]
"q11":@{}"q12"|(.5){}="m12" "m12" [u(.15)] :@{=>}[d(.3)]
"q12":@{}"q13"|(.5){}="m13" "m13" [u(.15)] :@{=>}[d(.3)]
"q20":@{}"q21"|(.5){}="m21" "m21" [u(.15)] :@{=>}[d(.3)]
"q30":@{}"q31"|(.5){}="m31" "m31" [u(.15)] :@{=>}[d(.3)]
"q31":@{}"q32"|(.5){}="m32" "m32" [u(.15)] :@{=>}[d(.3)]}}}} \]
respectively. The algebras of $\Gamma(T^{\times})$ in this case are categories enriched in semicategories using the cartesian product. In other words $\Gamma(T^{\times})$ algebras are just like strict $2$-categories except that they needn't have identity 2-cells. In particular the lack of identity 2-cells means that there is no meaningful operation of whiskering in a $\Gamma(T^{\times})$ algebra as there is in a 2-category. As a subcategory of $\ca G^2(\Set)^{\Gamma(T^{\times})}$ the objects of $\Theta_{\Gamma(T^{\times})}$ are free on those globular pasting diagrams which one doesn't require whiskering to build. For instance the left pasting diagram above does not live in $\Theta_{\Gamma(T^{\times})}$ whereas
the right one does. Thus this is an example where $w_T$ is \emph{not} an equivalence.
\end{example}

\subsection{Strict $n$-category monads}
\label{sec:strict-n-cat-monad}
One can consider the following inductively-defined sequence of monads
\begin{itemize}
\item  Put $\ca T_{{\leq}0}$ equal to the identity monad on $\Set$.
\item  Given a monad $\ca T_{{\leq}n}$ on $\ca G^n\Set$, define the monad
$\ca T_{{\leq}n+1} = \Gamma \ca T^{\times}_{{\leq}n}$
on $\ca G^{n+1}\Set$.
\end{itemize}
recalling that $\ca G^n\Set$ is the category of $n$-globular sets.
By example(\ref{ex:Tcross}) and proposition(\ref{prop:Gamma-E-algebras-as-E-categories}) it follows that $\ca G^n(\Set)^{\ca T_{\leq n}}$ is the category of strict $n$-categories and strict $n$-functors between them. By remarks(\ref{rem:distributive-opmonoidal-monad-elaboration}) and (\ref{rem:EMT-lra}), and example(\ref{ex:Tcross}), we recover the fundamental properties of these monads, that is that they are coproduct preserving, finitary and local right adjoint. Moreover from this inductive description of $\ca T_{\leq{n}}$ and theorem(\ref{thm:opmonoidal-monad-dist-law}) one recovers the distibutive law
\[ \begin{array}{c} {\ca G(\ca T_{\leq{n}})\Gamma(\prod) \rightarrow \Gamma(\prod)\ca G(\ca T_{\leq{n}})} \end{array} \]
for all $n$, between monads on $\ca G^n\Set$, with composite monad $\Gamma(\prod)\ca G(\ca T_{\leq{n}}) = \ca T_{\leq{(n{+}1)}}$, as witnessed in \cite{Cheng-IteratedDist}.

We denote by $k$ the free living $k$-cell viewed as a representable $n$-globular set (for $0 \leq k \leq n$). Recall \cite{Weber-Generic} that a globular set $A$ is a globular pasting diagram of dimension $\leq k$ iff there exists a generic morphism $k \to \ca T_{\leq n}A$. A globular pasting diagram of dimension $\leq k$ is obviously a globular pasting diagram of dimension $\leq n$, and so for such an $A$ there exists a generic morphism $n \to \ca T_{\leq n}A$. Thus $\ca T_{\leq n}$ satisfies the condition of remark(\ref{rem:condition-for-wreath-equivalence}), and so
\[ \Theta_{\ca T_{\leq n{+}1}} \catequiv \Delta \wr \Theta_{\ca T_{\leq n}} \]
by proposition(\ref{prop:gen-wreath-formula}) and remark(\ref{rem:condition-for-wreath-equivalence}). Up to isomorphism the category $\Theta_n$ of \cite{Berger-IteratedWreathProduct} may be defined as a skeleton of $\Theta_{\ca T_{\leq n}}$, and so one has $\Theta_{n{+}1} \iso \Delta \wr \Theta_n$.

As far as defining the monads $\ca T_{\leq n}$ is concerned, one could just as well start with any locally finitely c-presentable $V$ in place of $\Set$, giving monads whose algebras are $V$-enriched strict $n$-categories. This enrichment gives objects in $V$ of $n$-cells between parallel pairs of $(n-1)$-cells. By the same arguments these monads are also coproduct preserving, finitary and local right adjoint, and moreover give rise to distributive laws in the same way.

\section{Higher operads and contractibility}
\label{sec:contractible-operads-and-multitensors}

Since higher operads are monad morphisms of a certain kind, the functorial correspondence between monads and multitensors gives a multitensor viewpoint on higher operads. This is theorem(\ref{thm:equivalence-operads<->multitensors}) which generalises the main results of \cite{BataninWeber-EnHop}. Sections(\ref{ssec:trivial-fibrations}) and (\ref{ssec:contractible-multitensors}) are then concerned with extending this to give a sensible notion of ``contractible multitensor'' and its relation to contractible operads. Finally in section(\ref{sec:Trimble}) we use our theory to recover Trimble's definition, as described in \cite{Cheng-ComparingOperadic}.

\subsection{The basic correspondence}
\label{sec:operads-multitensors-basic}
Recall \cite{BataninWeber-EnHop} that just as one can define $T$-operads for a cartesian monad $T$, one also has a notion of $E$-multitensor for any cartesian multitensor $E$. For $(V,E)$ a cartesian multitensor, one defines an \emph{$E$-multitensor} to consist of another multitensor $A$ on $V$ together with a natural transformation $\alpha:A \to E$ which is cartesian natural and compatible with the multitensor structures.

For $V$ a category with pullbacks and $T$ a cartesian monad on $\ca GV$ over $\Set$, a $T$-operad $\alpha:A \to T$ over $\Set$ may be regarded as either a monad functor $(1_{\ca GV},\alpha) : (\ca GV,T) \to (\ca GV,A)$ over $\Set$ whose 1-cell datum is an identity and 2-cell datum is a cartesian transformation, or equally well as a monad opfunctor $(1_{\ca GV},\alpha) : (\ca GV,A) \to (\ca GV,T)$ whose 1-cell datum is an identity and 2-cell datum is cartesian. Similarly, an $E$-multitensor may be regarded as either a lax monoidal functor $(1_V,\alpha) : (V,E) \to (V,A)$ over $\Set$ whose 1-cell datum is an identity and 2-cell datum is a cartesian transformation, or equally well as an oplax monoidal functor $(1_V,\alpha) : (V,A) \to (V,E)$ whose 1-cell datum is an identity and 2-cell datum is cartesian.

We denote by $\NOp{T}$ the category of $T$-operads over $\Set$ and their morphisms. A morphism from $\alpha:A \to T$ to $\beta:B \to T$ is just a monad morphism $\gamma:A \to B$ such that $\alpha=\beta\gamma$. It follows that $\gamma$ is itself cartesian and over $\Set$. Thus a $T$-operad morphism may be regarded either as a monad functor $(1_V,\gamma) : (V,T) \to (V,A)$ or a monad opfunctor $(1_V,\gamma) : (V,A) \to (V,T)$. Similarly one has the category $\Mult{E}$ of $E$-multitensors and their morphisms, with a morphism from $\alpha:A \to E$ to $\beta:B \to E$ being multitensor map over $E$. As with operad morphisms, morphisms of $E$-multitensors are reexpressable either as lax monoidal functors under $(V,E)$ or as oplax monoidal functors over $(V,E)$. Thus by applying either $\Gamma$ or $\Gamma'$ to $E$-multitensors and their morphisms, one obtains a functor
\[ \Gamma_E : \Mult{E} \to \NOp{\Gamma E}. \]
By lemma(\ref{lem:transfer-dpl}), theorem(\ref{thm:characterisation-of-image-of-Gamma}) and proposition(\ref{prop:Gamma-respects-cartesian-transformations}), $\Gamma_E$ is essentially surjective on objects. By proposition(\ref{prop:Psi-2-ff}) it is fully faithful, and so we have obtained
\begin{theorem}\label{thm:equivalence-operads<->multitensors}
Let $V$ be lextensive. Then $\Gamma_E$ gives an equivalence of categories
\[ \Mult{E} \catequiv \NOp{\Gamma E}.\]
\end{theorem}
In the case where $E=\ca T_{\leq n}^{\times}$ we recover the first equivalence of corollary(7.10) and of corollary(8.3) of \cite{BataninWeber-EnHop}.

\subsection{Trivial Fibrations}
\label{ssec:trivial-fibrations}
Let $V$ be a category and $\ca I$ a class of maps in $V$. Denote by $\ca I^{\uparrow}$ the class of maps in $V$ that have the right lifting property with respect to all the maps in $\ca I$. That is to say, $f:X{\rightarrow}Y$ is in $\ca I^{\uparrow}$ iff for every $i:S{\rightarrow}B$ in $\ca I$, $\alpha$ and $\beta$ such that the outside of
\[ \xygraph{{S}="tl" [r] {X}="tr" [d] {Y}="br" [l] {B}="bl" "tl" (:"tr"^-{\alpha}:"br"^{f},:"bl"_{i}(:"br"_-{\beta},:@{.>}"tr"|{\gamma}))} \]
commutes, then there is a $\gamma$ as indicated such that $f\gamma{=}\beta$ and $\gamma{i}=\alpha$. An $f \in \ca I^{\uparrow}$ is called a \emph{trivial $\ca I$-fibration}. The basic facts about $\ca I^{\uparrow}$ that we shall use are summarised in
\begin{lemma}\label{lem:basic-tf}
Let $V$ be a category, $\ca I$ a class of maps in $V$, $J$ a set and \[ (f_j:X_j{\rightarrow}Y_j \,\,\, | \,\,\, j \in J) \]
a family of maps in $V$.
\begin{enumerate}
\item  $\ca I^{\uparrow}$ is closed under composition and retracts.\label{tfib1}
\item  If $V$ has products and each of the $f_j$ is a trivial $\ca I$-fibration, then
\[ \begin{array}{c} {\prod\limits_{j} f_j : \prod\limits_j X_j \rightarrow \prod\limits_j Y_j} \end{array} \]
is also a trivial $\ca I$-fibration.\label{tfib2}
\item  The pullback of a trivial $\ca I$-fibration along any map is a trivial $\ca I$-fibration.\label{tfib3}
\item  If $V$ is extensive and $\coprod_jf_j$ is a trivial $\ca I$-fibration, then each of the $f_j$ is a trivial fibration.\label{tfib4}
\item  If $V$ is extensive, the codomains of maps in $\ca I$ are connected and each of the $f_j$ is a trivial $\ca I$-fibration, then $\coprod_jf_j$ is a trivial $\ca I$-fibration.\label{tfib5}
\end{enumerate}
\end{lemma}
\begin{proof}
(\ref{tfib1})-(\ref{tfib3}) is standard. If $V$ is extensive then the squares
\[ \xygraph{!{0;(1.5,0):(0,.666)::}
{X_j}="tl" [r] {\coprod_jX_j}="tr" [d] {\coprod_jY_j}="br" [l] {Y_j}="bl" "tl" (:"tr":"br"^{\coprod_jf_j},:"bl"_{f_j}:"br")} \]
whose horizontal arrows are the coproduct injections are pullbacks, and so (\ref{tfib4}) follows by the pullback stability of trivial $\ca I$-fibrations. As for (\ref{tfib5}) note that for $i:S{\rightarrow}B$ in $\ca I$, the connectedness of $B$ ensures that any square as indicated on the left
\[ \xygraph{{\xybox{\xygraph{!{0;(1.5,0):(0,.666)::} {S}="tl" [r] {\coprod_jX_j}="tr" [d] {\coprod_jY_j}="br" [l] {B}="bl" "tl" (:"tr":"br"^{\coprod_jf_j},:"bl"_{i}:"br")}}}
[r(4)]
{\xybox{\xygraph{!{0;(1.5,0):(0,.666)::} {S}="tl" [r] {X_j}="tr" [d] {Y_j}="br" [l] {B}="bl" "tl" (:"tr":"br"^{f_j},:"bl"_{i}:"br")}}}} \]
factors through a unique component as indicated on the right, enabling one to induce the desired filler.
\end{proof}
\begin{definition}
Let $F,G:W{\rightarrow}V$ be functors and $\ca I$ be a class of maps in $V$. A natural transformation $\phi:F{\implies}G$ is a \emph{trivial $\ca I$-fibration} when its components are trivial $\ca I$-fibrations.
\end{definition}
\noindent Note that since trivial $\ca I$-fibrations in $V$ are pullback stable, this reduces, in the case where $W$ has a terminal object $1$ and $\phi$ is cartesian, to the map $\phi_1:F1{\rightarrow}G1$ being a trivial $\ca I$-fibration.

Given a category $V$ with an initial object, and a class of maps $\ca I$ in $V$, we denote by $\ca I^+$ the class of maps in $\ca GV$ containing the maps{\footnotemark{\footnotetext{Recall that $0$ is the $V$-graph with one object whose only hom is initial, or in other words the representing object of the functor $\ca GV{\rightarrow}\Set$ which sends a $V$-graph to its set of objects.}}}
\[ \begin{array}{lccr} {\emptyset \rightarrow 0} &&& {(i) : (S) \rightarrow (B)} \end{array} \]
where $i \in \ca I$. The proof of the following lemma is trivial.
\begin{lemma}\label{lem:ind-tf}
Let $V$ be a category with an initial object and $\ca I$ a class of maps in $V$. Then $f:X{\rightarrow}Y$ is a trivial $\ca I^+$-fibration iff it is surjective on objects and all its hom maps are trivial $\ca I$-fibrations.
\end{lemma}
\noindent In particular starting with $V=\PSh {\G}$ the category of globular sets and $\ca I_{-1}$ the empty class of maps, one generates a sequence of classes of maps $\ca I_n$ of globular sets by induction on $n$ by the formula $\ca I_{n+1}=(\ca I_{n})^+$ since $\ca G(\PSh {\G})$ may be identified with $\PSh {\G}$, and moreover one has inclusions $\ca I_n \subset \ca I_{n+1}$. More explicitly, the set $\ca I_n$ consists of $(n+1)$ maps: for $0{\leq}k{\leq}n$ one has the inclusion $\partial{k} \hookrightarrow k$, where $k$ here denotes the representable globular set, that is the ``$k$-globe'',
and $\partial{k}$ is the $k$-globe with its unique $k$-cell removed. One defines $\ca I_{{\leq}\infty}$ to be the union of the $\ca I_n$'s.
Note that by definition $\ca I_{{\leq}\infty}=\ca I_{{\leq}\infty}^{+}$.

There is another version of the induction just described to produce, for each $n \in \N$, a class $\ca I_{{\leq}n}$ of maps of $\ca G^n(\Set)$. The set $\ca I_{{\leq}0}$ consists of the functions
\[ \begin{array}{lccr} {\emptyset \rightarrow 0} &&& {0+0 \rightarrow 0,} \end{array} \]
so $\ca I_{\leq 0}^{\uparrow}$ is the class of bijective functions. For $n \in \N$, $\ca I_{{\leq}n{+}1}=\ca I_{{\leq}n}^{+}$. As maps of globular sets, the class $\ca I_{\leq n}$ consists of all the maps of $\ca I_{n}$ together with the unique map $\partial{(n{+}1)}{\rightarrow}n$. A map of $n$-globular sets is a trivial $\ca I_{\leq n}$ fibration iff it has the right lifting property with respect to all the morphisms of $\ca I_n$ and moreover the unique right lifting property in the top dimension (ie with respect to $\partial n \hookrightarrow n$).
\begin{definition}
Let $0{\leq}n{\leq}\infty$. An $n$-operad{\footnotemark{\footnotetext{The monad $\ca T_{{\leq}\infty}$ on globular sets is usually just denoted as $\ca T$: it is the monad whose algebras are strict $\omega$-categories.}}} $\alpha:A{\rightarrow}\ca T_{{\leq}n}$ is \emph{contractible} when it is a trivial $\ca I_{{\leq}n}$-fibration. An $n$-multitensor $\varepsilon:E{\rightarrow}\ca T_{{\leq}n}^{\times}$ is \emph{contractible} when it is a trivial $\ca I_{{\leq}n}$-fibration.
\end{definition}
\noindent By the preceeding two lemmas, an $(n+1)$-operad $\alpha:A{\rightarrow}\ca T_{{\leq}n{+}1}$ over $\Set$ is contractible iff the hom maps of $\alpha_1$ are trivial $\ca I_{\leq n}$-fibrations.

\subsection{Contractible operads versus contractible multitensors}
\label{ssec:contractible-multitensors}
As one would expect a $\ca T_{\leq n{+}1}$-operad over $\Set$ is contractible iff its associated $\ca T^{\times}_{{\leq}n}$-multitensor is contractible. This fact has quite a general explanation.
\begin{proposition}\label{prop:contractible}
Let $(H,\psi):(V,E){\rightarrow}(W,F)$ be a lax monoidal functor between distributive lax monoidal categories, and $\ca I$ a class of maps in $W$. Suppose that $W$ is extensive, $H$ preserves coproducts and the codomains of maps in $\ca I$ are connected. Then the following statements are equivalent
\begin{enumerate}
\item $\psi$ is a trivial $\ca I$-fibration.\label{propcase:contractible-monoidal-coherence}
\item $\Gamma\psi$ is a trivial $\ca I^+$-fibration.\label{propcase:contractible-monad-map}
\end{enumerate}
\end{proposition}
\begin{proof}
For each $X \in \ca GV$ the component $\{\Gamma\psi\}_X$ is the identity on objects and for $a,b \in X_0$, the corresponding hom map is obtained as the composite of
\[ \begin{array}{c} {\coprod\limits_{a{=}x_0,...,x_n{=}b} \psi : \coprod\limits_{x_0,...,x_n} \opF\limits_iHX(x_{i-1}x_i) \rightarrow \coprod\limits_{x_0,...,x_n} H\opE\limits_iX(x_{i-1}x_i)} \end{array} \]
and the canonical isomorphism that witnesses the fact that $H$ preserves coproducts. In particular note that for any sequence $(Z_1,...,Z_n)$ of objects of $V$, regarded as $V$-graph in the usual way, one has \[ \{\Gamma\psi\}_{(Z_1,...,Z_n)} = \psi_{Z_1,...,Z_n}.\]
Thus $(\ref{propcase:contractible-monoidal-coherence}){\iff}(\ref{propcase:contractible-monad-map})$ follows from lemmas(\ref{lem:basic-tf}) and (\ref{lem:ind-tf}).
\end{proof}
\begin{corollary}\label{cor:contractible}
Let $0{\leq}n{\leq}\infty$, $\alpha:A{\rightarrow}\ca T_{{\leq}n{+}1}$ be a $\ca T_{\leq n{+}1}$-operad over $\Set$ and $\varepsilon:E{\rightarrow}\ca T^{\times}_{{\leq}n}$ be the corresponding $\ca T^{\times}_{{\leq}n}$-multitensor. TFSAE:
\begin{enumerate}
\item  $\alpha:A{\rightarrow}\ca T_{{\leq}n{+}1}$ is contractible.
\item  $\varepsilon:E{\rightarrow}\ca T^{\times}_{{\leq}n}$ is contractible.
\end{enumerate}
\end{corollary}
\begin{proof}
By induction one may easily establish that the codomains of the maps in any of the classes: $\ca I_n$, $\ca I_{{\leq}n}$, $\ca I_{{\leq}\infty}$ are connected so that proposition(\ref{prop:contractible}) may be applied.
\end{proof}

\subsection{Trimble's construction}
\label{sec:Trimble}
In this section we exhibit Cheng's analysis of Trimble's definition \cite{Cheng-ComparingOperadic} as fitting within our framework.

\emph{Topological preliminaries}. Given a topological space $X$ and points $a$ and $b$ therein, one may define the topological space $X(a,b)$ of paths in $X$ from $a$ to $b$ at a high degree of generality. In recalling the details let us denote by $\Top$ a category of ``spaces'' which is complete, cocomplete and cartesian closed. We shall write $1$ for the terminal object. We shall furthermore assume that $\Top$ comes equipped with a bipointed object $I$ playing the role of the interval. A conventional choice for $\Top$ is the category of compactly generated Hausdorff spaces with its usual interval, although there are many other alternatives which would do just as well from the point of view of homotopy theory.

Let us denote by $\sigma{X}$ the suspension of $X$, which can be defined as the pushout
\[ \xygraph{!{0;(1.5,0):(0,.667)::} {X{+}X}="tl" [r] {I{\times}X}="tr" [d] {\sigma{X}.}="br" [l] {1{+}1}="bl" "tl"(:"tr":"br",:"bl":"br") "br" [u(.3)l(.3)] (:@{-}[r(.15)],:@{-}[d(.15)])} \]
Writing $\Top_{\bullet}$ for the category of bipointed spaces, that is to say the coslice $1{+}1/\Top$, the above definition exhibits the suspension construction as a functor
\[ \sigma : \Top \rightarrow \Top_{\bullet}. \]
Applying $\sigma$ successively to the inclusion of the empty space into the point, one obtains the inclusions of the $(n{-}1)$-sphere into the $n$-disk for all $n \in \N$, and its right adjoint
\[ h : \Top_{\bullet} \rightarrow \Top \]
is the functor which sends the bipointed space $(a,X,b)$, to the space $X(a,b)$ of paths in $X$ from $a$ to $b$. This adjunction $\sigma \ladj h$ is easy to verify directly using the above elementary definition of $\sigma(X)$ as a pushout, and the pullback square
\[ \xygraph{!{0;(1.5,0):(0,.667)::} {X(a,b)}="tl" [r] {X^I}="tr" [d] {X^{1{+}1}}="br" [l] {1}="bl" "tl"(:"tr":"br"^{X^i},:"bl":"br"_-{(a,b)}) "tl" [d(.3)r(.3)] (:@{-}[l(.15)],:@{-}[u(.15)])} \]
where $i$ is the inclusion of the boundary of $I$.

Thus to each space $X$ one can associate a canonical topologically enriched graph whose homs are the path spaces of $X$. In section(\ref{sec:categories-of-enriched-graphs}) we described explicitly the adjunction $(-)^{\times 2}_{\bullet} \ladj \ca G_1$ and in particular its unit $\eta$, from which one may readily verify that the assignment $X \mapsto PX$ is the object map of the composite
\[ \xygraph{!{0;(2,0):} {\Top}="l" [r] {\ca G(\Top_{\bullet})}="m" [r] {\ca G\Top}="r" "l":"m"^-{\eta}:"r"^-{\ca Gh}} \]
and by proposition(\ref{prop:unit-G1}), the component $\eta_f$ of the unit of this adjunction at $f:A \to \Set$ has a left adjoint when $A$ is cocomplete. Since $\Top$ is cocomplete, $h$ has left adjoint $\sigma$, and $\ca G$ is a 2-functor, whence $P$ is a right adjoint.

Recall from section(\ref{sec:recalling-multitensors}) that non-symmetric operads within braided monoidal categories may be regarded as multitensors, and that these are distributive when the tensor product is distributive. To say that a non-symmetric topological operad $A$ acts on $P$ is to say that $P$ factors as
\[ \xygraph{!{0;(2,0):} {\Top}="l" [r] {\Enrich A}="m" [r] {\ca G(\Top)}="r" "l":"m"^-{P_A}:"r"^-{U^A}} \]
The main example to keep in mind is the version of the little intervals operad recalled in \cite{Cheng-ComparingOperadic} definition(1.1). As this $A$ is a contractible non-symmetric operad, $A$-categories may be regarded as a model of $A$-infinity spaces. Since $P$ is a right adjoint, $P_A$ is also a right adjoint by the Dubuc adjoint triangle theorem.

\emph{Inductive construction}. Let $A$ be a non-symmetric topological operad which acts on $P$. Applying a product preserving functor
\[ Q : \Top \rightarrow V \]
into a distributive category to the operad $A$ in $\Top$, produces an operad $QA$ in $V$. Moreover $Q$ may be regarded as the underlying functor of a strong monoidal functor $(\Top,A) \rightarrow (V,QA)$ between lax monoidal categories. Applying $\Gamma$ to this gives us a monad functor
\[ (\ca G(\Top),\Gamma(A)) \rightarrow (\ca GV,\Gamma(QA)) \]
with underlying functor $\ca GQ$, which amounts to giving a lifting $\overline{Q}$ as indicated in the commutative diagram
\[ \xygraph{!{0;(2,0):(0,.5)::} {\Top}="tl" [r] {\Enrich A}="tm" [r] {\Enrich {QA}}="tr" [d] {\ca GV}="br" [l] {\ca G(\Top)}="bl" "tl"(:"tm"^-{P_A}(:"tr"^-{\overline{Q}}:"br"^{U^{QA}},:"bl":"br"_-{\ca G(Q)}),:"bl"_{P})} \]
and so we have produced another product preserving functor
\[ Q^{(+)} : \Top \rightarrow V^{(+)} \]
where $Q^{(+)}=\overline{Q}P_A$ and $V^{(+)}=\Enrich {QA}$. The functor $\overline{Q}$ is product preserving since $\ca G(Q)$ is and $U^{QA}$ creates products. The assignment
\[ (Q,V) \mapsto (Q^{(+)},V^{(+)}) \]
in the case where $A$ is as described in \cite{Cheng-ComparingOperadic} definition(1.1), is the inductive process lying at the heart of the Trimble definition. In this definition one begins with the path components functor $\pi_0 : \Top \rightarrow \Set$ and defines the category $\Trimble 0$ of ``Trimble 0-categories'' to be $\Set$. The induction is given by
\[ (\Trimble {n{+}1},\pi_{n{+}1}) := (\textnormal{Trm}_{n}^{(+)},\pi_{n}^{(+)}) \]
and so this definition constructs not only a notion of weak $n$-category but the product preserving $\pi_n$'s to be regarded as assigning the fundamental $n$-groupoid to a space.

\emph{Operads for Trimble $n$-categories}. In the context of a product preserving functor $Q:\Top \to V$ as above, suppose that $W$ is a lextensive category, $T$ is a coproduct preserving cartesian monad on $W$, and $\phi:S \to T$ is a $T$-operad. Suppose moreover that $V = W^S$. Then the operad/multitensor $QA$ is a lifting of the operad/multitensor $U^SQA$, and so $\Enrich{QA}$ may be identified with categories enriched in the multitensor on $W$ whose tensor product is given by $(U^SQA)_n \times SX_1 \times ... \times SX_n$ by theorem(\ref{thm:opmonoidal-monad-dist-law}). But the composites
\begin{equation}\label{eq:Trimble-operad}
\xygraph{!{0;(2.5,0):} {U^SQ(A)_n \times \prod\limits_i SX_i}="l" [r] {\prod\limits_i SX_i}="m" [r] {\prod\limits_i TX_i}="r" "l":@<1ex>"m"^-{\tn{proj}}:@<1ex>"r"^-{\prod\limits_i \phi_{X_i}}}
\end{equation}
are the components of a cartesian multitensor map into $T^{\times}$. Thus by theorem(\ref{thm:equivalence-operads<->multitensors}) $\Enrich{(QA)}$ is the category of algebras of a $\Gamma(T^{\times})$-operad over $\Set$. Thus by the inductive definition of $\Trimble n$ and of the monads $\ca T_{\leq n}$, $\Trimble n$ is the category of algebras of a $\ca T_{\leq n}$-operad.

\emph{Contractibility of the Trimble operads}. Let us denote by $\ca J$ the set of inclusions $S^{n{-}1}{\rightarrow}D^n$ of the $n$-sphere into the $n$-disk for $n \in \N$. Recall that these may all be obtained by successively applying the suspension functor $\sigma$ to the inclusion of the empty space into the point. By definition the given topological operad $A$ is contractible when for each $n$ the unique map $A_n \rightarrow 1$ is in $\ca J^{\uparrow}$, and this is equivalent to saying that the cartesian multitensor map $A \rightarrow \prod$ is a trivial $\ca J$-fibration. We shall write $U_n:\Trimble n \rightarrow \ca G^n\Set$ for the forgetful functor for each $n$.
\begin{lemma}\label{lem:triv-fibrations-for-Eugenia}
If $f:X{\rightarrow}Y$ is a trivial $\ca J$-fibration then
\begin{enumerate}
\item $f_{a,b}:X(a,b){\rightarrow}Y(fa,fb)$ is a trivial $\ca J$-fibration for all $a,b \in X$.\label{lemcase:pathspaces}
\item $U_n\pi_n{f}$ is a trivial $\ca I_{\leq{n}}$-fibration.\label{lemcase:U-npi-nf}
\end{enumerate}
\end{lemma}
\begin{proof}
(\ref{lemcase:pathspaces}): To give a commutative square as on the left in
\[ \begin{array}{lccr} {\xybox{\xygraph{!{0;(2,0):(0,.5)::} {S^{n{-}1}}="tl" [r] {X(a,b)}="tr" [d] {Y(fa,fb)}="br" [l] {D^n}="bl" "tl"(:"tr":"br",:"bl":"br")}}} &&&
{\xybox{\xygraph{!{0;(2,0):(0,.5)::} {S^n}="tl" [r] {(a,X,b)}="tr" [d] {(fa,Y,fb)}="br" [l] {D^{n{+}1}}="bl" "tl"(:"tr":"br",:"bl":"br")}}} \end{array} \]
is the same as giving a commutative square in $\Top_{\bullet}$ as on the right in the previous display, by $\sigma \ladj h$. The square on the right admits a diagonal filler $D^{n{+}1} \rightarrow X$ since $f$ is a trivial $\ca J$-fibration, and thus so does the square on the left.

(\ref{lemcase:U-npi-nf}): We proceed by induction on $n$. Having the right lifting property with respect to the inclusions
\[ \begin{array}{lccr} {\emptyset \hookrightarrow 1} &&& {1{+}1=\partial{I} \hookrightarrow I} \end{array} \]
ensures that $f$ surjective and injective on path components, and thus is inverted by $\pi_0$. For the inductive step we assume that $U_n\pi_n$ sends trivial $\ca J$-fibrations to trivial $\ca I_{{\leq}n}$-fibrations and suppose that $f$ is a trivial $\ca J$-fibration. Then so are all the maps it induces between path spaces by (\ref{lemcase:pathspaces}). But from the inductive definition of $\Trimble {n{+}1}$ we have $U_{n{+}1}\pi_{n{+}1}=\ca G(u_n\pi_n)P$ and so $U_{n{+}1}\pi_{n{+}1}(f)$ is a morphism of $(n{+}1)$-globular sets which is surjective on objects (as argued already in the $n=0$ case) and whose hom maps are trivial $\ca I_{{\leq}n}$-fibrations by induction. Thus the result follows by lemma(\ref{lem:ind-tf}).
\end{proof}
\begin{corollary}\label{cor:Eugenia}
\emph{(\cite{Cheng-ComparingOperadic} Theorem(4.8))} $\Trimble n$ is the category of algebras for a contractible $\ca T_{\leq n}$-operad.
\end{corollary}
\begin{proof}
We proceed by induction on $n$, and the case $n=0$ holds trivially. For the inductive step we must show by corollary(\ref{cor:contractible}) that the components (\ref{eq:Trimble-operad}) are trivial $\ca I_{\leq n}$-fibrations, where $Q=\pi_n$, $T=\ca T_{\leq n}$ and $\phi:S \to T$ the contractible operad for Trimble $n$-categories. But by lemma(\ref{lem:triv-fibrations-for-Eugenia}) the unique map $U^SQ(A) \to 1$ is a trivial $\ca I_{\leq n}$-fibration since $A$ is contractible, and so the result follows from lemma(\ref{lem:basic-tf}) since trivial fibrations are closed under products and composition.
\end{proof}

\appendix

\section{Locally connected and locally presentable categories}
\label{sec:lcp}
Higher categorical structures are supposed to model the homotopy-theoretic aspects of spaces. Thus the categories that arise in this work behave in some respects as categories of space-like objects, even before one considers any Quillen model category structures. The formal expression of this is that all the categories at arise in this work are are locally c-presentable in the sense to be discussed in this section. This includes a well-behaved notion of connected component of an object, and that the ability to decompose objects into connected components works as one would want. From a technical standpoint, local c-presentability also plays an important role in the dictionary between monads and multitensors. In particular, the correspondence between local right adjoint multitensors and local right adjoint monads given in theorem(\ref{thm:preservation-by-Gamma-E}) requires that the underlying categories are locally c-presentable.

The material of this section is somewhat of a review, being essentially an instance of the theory given in \cite{ABLR-ClassificationAccessible}. However we do cover the particular case of the theory of locally c-presentable categories in considerably more detail than in \cite{ABLR-ClassificationAccessible}. There are two principal results in this section. The first of these, theorem(\ref{thm:conn-GabUlm}), characterises locally c-presentable categories in various ways. From this result it is clear that locally connected Grothendieck toposes are examples. The second result, theorem(\ref{thm:acc-monad}), exhibits algebras of coproduct preserving accessible monads on locally c-presentable categories as locally c-presentable. By this result the categories of algebras of higher operads are exhibited as locally c-presentable.

\subsection{Connected objects and locally connected categories}
\label{ssec:connected-objects}
We now collect together the basic, mostly well-known, abstract categorical theory of connected objects and coproduct decompositions. Recall that an object $C$ in a category $V$ with coproducts is \emph{connected} when the representable $V(C,-)$ preserves coproducts.

The natural environment within which to study coproduct decompositions is a lextensive category. Recall that a category $V$ is \emph{extensive} when it has coproducts and for all families $(X_i:i \in I)$ of objects of $V$, the functor
\[ \begin{array}{lccr} {\coprod : \prod\limits_i (V/X_i) \rightarrow V/(\coprod\limits_i X_i)}
&&& {(f_i:Y_i{\rightarrow}X_i) \mapsto \coprod\limits_if_i:\coprod\limits_iY_i{\rightarrow}\coprod\limits_iX_i} \end{array} \]
is an equivalence of categories. Note that this terminology is not quite standard: extensivity is usually defined using only finite coproducts.

Recall that coproducts in a category are said to be \emph{disjoint} when coproduct coprojections are mono and the pullback of different coprojections is initial. Recall also that an initial object is said to be \emph{strict} when any map into it is an isomorphism. The fundamental result on extensive categories is
\begin{theorem}\label{thm:extensivity-characterisation}
(\cite{CarboniLackWalters-Extensivity},\cite{Cockett-DistributiveCategories}) A category $V$ is extensive iff it has coproducts, pullbacks along coproduct coprojections and given a family of commutative squares
\[ \xymatrix{{X_i} \ar[r]^-{c_i} \ar[d]_{f_i} & X \ar[d]^{f} \\ {Y_i} \ar[r]_-{d_i} & Y} \]
where $i \in I$ such that the $d_i$ form a coproduct cocone, the $c_i$ form a coproduct cocone iff these squares are all pullbacks. In an extensive category coproducts are disjoint and the initial object of $V$ is strict.
\end{theorem}
We consider now conditions on a category which turn out to be sufficient to ensure extensivity.
\begin{definition}\label{def:locally-connected-category}
A category $V$ is \emph{locally connected} when
\begin{enumerate}
\item $V$ has coproducts.
\item $V$ has pullbacks along coproduct inclusions.
\item every $X \in V$ is a coproduct of connected objects.
\end{enumerate}
\end{definition}
\begin{lemma}\label{lem:easy-ext}
If a category $V$ is locally connected then it is extensive.
\end{lemma}
\begin{proof}
Suppose $f:X \to \emptyset$ is a morphism into the initial object, and $X = \coprod_{i{\in}I} X_i$ is a decomposition of $X$ as a coproduct of connected objects. Then for any $i \in I$, one has by composing with $c_i$ the i-th coproduct coprojection, a map $X_i \to \emptyset$. But since $X_i$ is connected there can be no such map since the hom $V(X_i,\emptyset)$ is empty, and so $I$ must be empty, and so $X$ is initial, and so $f$ is invertible. Thus $V$ has a strict initial object.

Given $A$ and $B$ in $V$, denote by $c_A:A \to A + B$ the coprojection. Given a pair of maps $f,g:X \to A$ such that $c_Af=c_Ag$, using $X$'s coproduct decomposition again one has $c_Afc_i=c_Agc_i$, but since $X_i$ is connected $fc_i=gc_i$, and since this is true for all $i$, $f=g$, and so $c_A$ is mono. On the other hand suppose that a commutative square
\[ \xygraph{ !{0;(1.5,0):(0,.5)::} {X}="tl" [r] {B}="tr" [d] {A + B}="br" [l] {A}="bl" "tl":"tr"^-{}:"br"^-{c_B}:@{<-}"bl"^-{c_A}:@{<-}"tl"^-{}} \]
is given. Then by composing with $c_i$, one obtains another with $X_i$ in place of $X$, but this cannot be since $X_i$ is connected. Thus $I$ is empty, and so $X$ is initial. Thus the coproducts in $V$ are disjoint.

By theorem(\ref{thm:extensivity-characterisation}) it suffices to show that given a family of commutative squares as on the left in
\[ \xygraph{{\xybox{\xygraph{{X_i}="tl" [r] {X}="tr" [d] {Y}="br" [l] {Y_i}="bl" "tl":"tr"^-{c_i}:"br"^-{f}:@{<-}"bl"^-{d_i}:@{<-}"tl"^-{f_i}}}} [r(4)]
{\xybox{\xygraph{!{0;(2,0):(0,.5)::} {V(Z,X_i)}="tl" [r] {V(Z,X)}="tr" [d] {V(Z,Y)}="br" [l] {V(Z,Y_i)}="bl" "tl":"tr"^-{V(Z,c_i)}:"br"^-{V(Z,f)}:@{<-}"bl"^-{V(Z,d_i)}:@{<-}"tl"^-{V(Z,f_i)}}}}} \]
where $i \in I$ such that the $d_i$ form a coproduct cocone, the $c_i$ form a coproduct cocone iff these squares are all pullbacks. By the yoneda lemma, this is equivalent to the same statement for the family of squares on the right (in $\Set$) for all $Z \in V$. Since every object of $V$ is a coproduct of connected ones, it suffices to consider just those $Z \in V$ that are connected. By the definition of connectedness, the functions $V(Z,c_i)$ (resp. $V(Z,d_i)$) form a coproduct cocone for all connected $Z$ iff the maps $c_i$ (resp $d_i$) do so in $V$, and so the result follows by the extensivity of $\Set$.
\end{proof}
\begin{remark}\label{rem:Sh(X)-not-locally-connected}
Let $X$ be a topological space which is not locally connected. Then the topos $\tn{Sh}(X)$ of sheaves on $X$ is an example of a category which is extensive but not locally connected.
\end{remark}
Recall that a category is \emph{lextensive} when it is extensive and has finite limits. We shall now study the decomposability of objects in such categories. As we shall see, the categorical datum which tells us whether all objects in a lextensive category $V$ admit a coproduct decomposition, is the left adjoint $(-) \cdot 1$ to the representable $V(1,-)$, where $1$ as usual denotes the terminal object.
\begin{lemma}\label{lem:lex-discrete}
If $V$ is lextensive, then the representable $V(1,-):V \to \Set$ has a left exact left adjoint $(-) \cdot 1$ given by taking copowers with $1$.
\end{lemma}
\begin{proof}
It is a standard fact, coming from nothing more than the universal property of coproducts, that the left adjoint take this form, and clearly $(-) \cdot 1$ preserves the terminal object. Given a pullback in $\Set$ as on the left in
\[ \xygraph{
{\xybox{\xygraph{{P}="tl" [r] {B}="tr" [d] {C}="br" [l] {A}="bl" "tl":"tr"^-{f}:"br"^-{k}:@{<-}"bl"^-{g}:@{<-}"tl"^-{h} "tl":@{}"br"|*{\scriptstyle{pb}}}}} [r(4)]
{\xybox{\xygraph{!{0;(1.5,0):(0,.667)::} {h^{-1}(a) \cdot 1}="tll" [r] {P \cdot 1}="tl" [r] {B \cdot 1}="tr" [d] {C \cdot 1}="br" [l] {A \cdot 1}="bl" [l] {1}="bll" "tl":"tr"^-{f \cdot 1}:"br"^-{k \cdot 1}:@{<-}"bl"^-{g \cdot 1}:@{<-}"tl"_-{h \cdot 1}
"tll":"tl" "bl":@{<-}"bll"^-{a}:@{<-}"tll" "tll":@{}"bl"|*{\scriptstyle{pb}}}}}} \]
one has for each $a \in A$ a diagram as on the right. Since the original square is a pullback one has canonical bijections $h^{-1}(a) \iso k^{-1}(ga)$ enabling one to identify the top horizontal composite as $(-) \cdot 1$ applied to the inclusion of the fibre $k^{-1}(ga)$, and so for all $a$ these maps exhibit $B \cdot 1$ as a coproduct. By theorem(\ref{thm:extensivity-characterisation}) it follows that the composite square on the right is a pullback. Since this is true for all $a \in A$ the right-most square is a pullback, again by theorem(\ref{thm:extensivity-characterisation}), as required.
\end{proof}
This is very familiar in the case where $V$ is a Grothendieck topos. Then the adjoint pair $(-){\cdot}1 \ladj V(1,-)$ is the global sections geometric morphism. Recall also that in this case the existence of a further left adjoint to $(-){\cdot}1$ is a fundamental property, which in the case of $\tn{Sh}(X)$ for $X$ a topological space, is equivalent to the local connectedness of $X$ (see remark(\ref{rem:Sh(X)-not-locally-connected}) above). Inspired by this case, we make
\begin{definition}\label{def:admits-pi0}
Let $V$ be a category with coproducts and a terminal object $1$. A left adjoint to $(-){\cdot}1$ is denoted as
\[ \pi_0 : V \longrightarrow \Set \]
and when it exists, we say that $V$ \emph{admits a $\pi_0$ functor}.
\end{definition}
We now note that the connectedness of an object in a lextensive category can be characterised in various ways.
\begin{lemma}\label{lem:char-connectedness-inLextCat}
Let $V$ be a lextensive category and $C$ be an object therein. Then the following statements are equivalent:
\begin{enumerate}
\item $C$ is connected.
\item $V(C,-)$ preserves copowers with $1$.
\end{enumerate}
and if in addition $V$ admits a $\pi_0$ functor, then these are moreover equivalent to
\begin{enumerate}[resume]
\item $\pi_0(C) \iso 1$.
\end{enumerate}
\end{lemma}
\begin{proof}
Suppose that $V(C,-)$ preserves copowers with $1$. Coproduct coprojections defining $X=\coprod_{i{\in}I} X_i$ assemble, by theorem(\ref{thm:extensivity-characterisation}), into pullback squares
\[ \xygraph{{X_i}="tl" [r] {X}="tr" [d] {I \cdot 1}="br" [l] {1}="bl" "tl":"tr"^-{c_i}:"br"^-{}:@{<-}"bl"^-{i \cdot 1}:@{<-}"tl"^-{} "tl":@{}"br"|*{\scriptstyle{pb}}} \]
to which we apply $V(C,-)$. By theorem(\ref{thm:extensivity-characterisation}) in the case $V=\Set$, the functions $V(C,c_i)$ form a coproduct cocone since $V(C,i)$ do by hypothesis. Thus $V(C,-)$ does indeed preserve all coproducts. In the case where one has $\pi_0$, by the canonical isomorphisms $V(C,I \cdot 1) \iso \Set(\pi_0C,I)$, the connectedness of $C$ is equivalent to $\Set(\pi_0C,-)$ being isomorphic to the identity, which by the yoneda lemma is equivalent to $\pi_0C \iso 1$.
\end{proof}
We now characterise those lextensive categories in which every object admits a decomposition as a coproduct of connected objects.
\begin{proposition}\label{prop:lext-decompose-charn}
Let $V$ be a lextensive category. Then $V$ is locally connected iff $V$ admits a $\pi_0$ functor.
\end{proposition}
\begin{proof}
Suppose that every object of $V$ can be expressed as a coproduct of connected objects. For each $X \in V$ choose such a decomposition, write $\pi_0(X)$ for the indexing set, and for $i \in \pi_0(X)$ denote by $c_i:X_i \to X$ the corresponding coprojection. One induces the map $\eta_X$ as in
\begin{equation}\label{diag:lext-decompose-charn}
\xygraph{!{0;(2,0):(0,.5)::} {X_i}="tl" [r] {X}="tr" [d] {\pi_0(X) \cdot 1}="br" [l] {1}="bl" "tl":"tr"^-{c_i}:"br"^-{\eta_X}:@{<-}"bl"^-{i \cdot 1}:@{<-}"tl"^-{} "tl":@{}"br"|*{\scriptstyle{pb}} "br" [r] {I \cdot 1}="brr" "tr":"brr"^-{f} "bl":@/_{2pc}/"brr"_-{g(i) \cdot 1} "br":@{.>}"brr"_-{g \cdot 1}} \end{equation}
so that the square commutes, and as indicated this square is a pullback by theorem(\ref{thm:extensivity-characterisation}). Given a set $I$ and a morphism $f$ as in (\ref{diag:lext-decompose-charn}), the connectedness of $X_i$ ensures that there is a unique $g(i) \in I$ making the outside of (\ref{diag:lext-decompose-charn}) commute. In this way we have exhibited a unique $g:\pi_0X \to I$ making the triangle in (\ref{diag:lext-decompose-charn}) commutative, and this exhibits $\eta_X$ as the component at $X$ of a unit of $\pi_0 \ladj (-) \cdot 1$. Moreover the uniqueness of coproduct decompositions is now evident, since any choice of all them gives rise in this way to an explicit left adjoint $\pi_0$ of the same functor $(-) \cdot 1$, and so different choices give rise to canonical isomorphisms between the corresponding $\pi_0$'s, which are compatible with the corresponding units.

For the converse let us suppose that we have $\pi_0 \ladj (-) \cdot 1$. Then one has $\eta_X:X \to \pi_0(X) \cdot 1$, and one then takes pullbacks as in (\ref{diag:lext-decompose-charn}) to obtain the $c_i:X_i \to X$ which form a coproduct cocone by theorem(\ref{thm:extensivity-characterisation}). To finish the proof we must show that all these $X_i$'s are connected, and by lemma(\ref{lem:char-connectedness-inLextCat}) it suffices to show that the cardinality $|\pi_0(X_i)|$ is $1$. If it was $0$ then one would have $\eta_{X_i}:X_i \to \emptyset$ making $X_i$ initial too, since initial objects are strict. But then by defining $I = \pi_0(X) \setminus \{i\}$, writing $j:I \to \pi_0(X)$ for the proper inclusion, one has also $\eta':X \to I \cdot 1$ such that $(j \cdot 1)\eta'=\eta_X$. But by the universal property of $\eta_X$ one also has a section $s:\pi_0(X) \to I$ of $j$, contradicting the properness of $j$. Thus $|\pi_0(X_i)| > 0$. Note that we have a diagram
\[ \xygraph{!{0;(2,0):(0,.5)::} {X_i}="tl" [r] {X}="tr" [d] {\pi_0(X) \cdot 1}="br" [l] {1}="bl" [d(.5)l] {\pi_0(X_i) \cdot 1}="bbll" "tl":"tr"^-{c_i}:"br"^-{\eta_X}:@{<-}"bl"_-{i}:@{<-}"tl"^-{}:"bbll"_-{\eta_{X_i}}(:"bl",:@/_{1.5pc}/"br"_-{\pi_0(c_i) \cdot 1}|{}="cd") "bl":@{}"cd"|-{(I)}} \]
in which the outside and all regions except region (I) are clearly commutative. By the universal property of $\eta_{X_i}$ the function $\pi_0(c_i)$ factors as
\[ \xygraph{{\pi_0(X_i)}="l" [r] {1}="m" [r] {\pi_0(X)}="r" "l":"m"^-{}:"r"^-{i}} \]
so that $|\tn{im}(\pi_0(c_i))| \leq 1$, and since $\pi_0(c_i)$ as a coprojection in $\Set$ is injective, we have $|\pi_0(X_i)| \leq 1$.
\end{proof}

\subsection{Locally c-presentable categories}
\label{ssec:lcc}
We recall first some of the basic notions from the theory of locally presentable categories \cite{AdamekRosicky-LFP,GabrielUlmer-LFP,MakkaiPare-AccessibleCategories}. Let $\lambda$ be a regular cardinal. A \emph{$\lambda$-small category} is one whose class of arrows forms a set of cardinality $< \lambda$, and a category $\ca A$ is \emph{$\lambda$-filtered} when every functor $J \to \ca A$, where $J$ is $\lambda$-small, admits a cocone. Colimits of functors out of $\lambda$-filtered categories are called \emph{$\lambda$-filtered colimits}. An object $X$ of a category $V$ is \emph{$\lambda$-presentable} when the representable $V(X,-)$ preserves all $\lambda$-filtered colimits that exist in $V$. A locally small category $V$ is \emph{locally $\lambda$-presentable} when it is cocomplete and there is a set $\ca S$ of $\lambda$-presentable objects such that every object of $V$ is a $\lambda$-filtered colimit of objects from $\ca S$.

There are many alternative characterisations of locally $\lambda$-presentable categories, the most minimalistic being the following. Recall that a set $\ca D$ of objects of $V$ is a \emph{strong generator} when for all maps $f:X{\rightarrow}Y$, if
\[ V(D,f) : V(D,X) \rightarrow V(D,Y) \]
is bijective for all $D \in \ca D$, then $f$ is an isomorphism. Then a locally small category $V$ is locally $\lambda$-presentable iff it is cocomplete and has a strong generator consisting of $\lambda$-presentable objects. Other characterisations include: as categories of $\Set$-valued models of limit sketches whose distinguished cones are $\lambda$-small; as full reflective subcategories of presheaf categories for which the inclusion is $\lambda$-accessible; to name just two. See for instance \cite{AdamekRosicky-LFP,GabrielUlmer-LFP,MakkaiPare-AccessibleCategories} for a complete discussion of this fundamental notion.

The appropriate functors between such categories are the $\lambda$-accessible ones -- a functor being \emph{$\lambda$-accessible} when it preserves $\lambda$-filtered colimits. Here we describe a mild variant of these notions in which the role of $\lambda$-presentable objects is played by objects which are both $\lambda$-presentable and connected, and accessible functors are replaced by functors which preserve both $\lambda$-filtered colimits and coproducts. A category $\ca A$ is \emph{$\lambda$-c-filtered} when every functor $J \to \ca A$, where $J$ is $\lambda$-small and connected, admits a cocone. Clearly a category is $\lambda$-c-filtered iff its connected components are $\lambda$-filtered, and thus a $\lambda$-c-filtered colimit is the same thing as a coproduct of $\lambda$-filtered colimits{\footnotemark{\footnotetext{When $\lambda$ is the first infinite cardinal, such categories are often said to be ``pseudo filtered''.}}}.

The categories which are $\lambda$-small and connected form a \emph{doctrine} $\D$ in the sense of \cite{ABLR-ClassificationAccessible}, a $\lambda$-c-filtered category is one which is $\D$-filtered in the sense of \cite{ABLR-ClassificationAccessible} definition(1.
2), and this doctrine is easily exhibited as \emph{sound} in the sense of \cite{ABLR-ClassificationAccessible} definition(2.2).
\begin{definition}\label{def:lcc}
\cite{ABLR-ClassificationAccessible} A locally small category $V$ is \emph{locally $\lambda$-c-presentable} when it is cocomplete and has a set $\ca S$ of objects which are connected and $\lambda$-presentable, such that every object of $V$ is a $\lambda$-c-filtered colimit of objects of $\ca S$. A \emph{locally c-presentable} category is one which is locally $\lambda$-c-presentable for some regular cardinal $\lambda$. When $\lambda$ is the first infinite ordinal, we also use the terminology locally \emph{finitely} c-presentable.
\end{definition}
We have attributed definition(\ref{def:lcc}) to \cite{ABLR-ClassificationAccessible} since a locally $\lambda$-c-presentable category is exactly a locally-$\D$-presentable category in the sense of \cite{ABLR-ClassificationAccessible} definition(5.1), for the doctrine $\D$ of $\lambda$-small connected categories. The soundness of this doctrine ensures, by theorem(5.5) of \cite{ABLR-ClassificationAccessible}, that one has various reformulations of the notion of $\lambda$-c-presentable category, analogous to those in the usual theory of locally presentable categories. We record these reformulations in
\begin{theorem}\label{thm:ABLR}
\cite{ABLR-ClassificationAccessible}
For a locally small category $V$ and regular cardinal $\lambda$, the following statements are equivalent.
\begin{enumerate}
\item  $V$ is locally $\lambda$-c-presentable.
\item  $V$ is equivalent to the category of models for a limit sketch whose distingished cones are $\lambda$-small and connected.\label{lc5}
\item  $V$ is equivalent to the full subcategory of $[\ca A,\Set]$ consisting of $\lambda$-small connected limit preserving functors, for some small category $\ca A$ with $\lambda$-small connected limits.\label{lc5.5}
\item  $V$ is a free completion of a small category with $\lambda$-small connected limits under $\lambda$-c-filtered colimits.\label{lc5.75}
\end{enumerate}
\end{theorem}
In theorem(\ref{thm:conn-GabUlm}) below we give further reformulations of this notion.
\begin{lemma}\label{lem:summands-of-lambda-presentables}
Let $\lambda$ be a regular cardinal and $V$ be locally $\lambda$-presentable and extensive. Then a summand of a $\lambda$-presentable object in $V$ is $\lambda$-presentable.
\end{lemma}
\begin{proof}
Suppose that $A$, $B \in V$ and that their coproduct $A + B$ is a $\lambda$-presentable object. Since $V$ is locally $\lambda$-presentable one has a $\lambda$-filtered category $I$, and a colimit cocone $k_i:A_i \to A$ for $i \in I$, where the $A_i$ are $\lambda$-presentable objects. Thus the maps $k_i + 1_B:A_i+B \to A+B$ exhibit $A+B$ as a $\lambda$-filtered colimit. Since $A + B$ is $\lambda$-presentable, there is $j \in I$ and $s:A + B \to A_j + B$ such that $(k_j + 1_B)s = 1_{A+B}$. Extensivity ensures that the right-most square in
\[ \xygraph{!{0;(2,0):(0,.5)::} {A}="tl" [r] {A_j}="tm" [r] {A}="tr" [d] {A+B}="br" [l] {A_j+B}="bm" [l] {A+B}="bl" "tl":@{.>}"tm"_-{t}:"tr"_-{k_j}:"br"^-{c_A}:@{<-}"bm"_-{k_j+1_B}:@{<-}"bl"_-{s}:@{<-}"tl"^-{c_A} "tm":"bm"_-{c_{A_j}} "tl":@/^{1pc}/"tr"^-{1} "bl":@/_{1pc}/"br"_-{1}} \]
is a pullback, enabling us to induce $t$ as shown which exhibits $k_j$ as a retraction, and thus $A$ as $\lambda$-presentable.
\end{proof}
\begin{theorem}\label{thm:conn-GabUlm}
For a locally small category $V$ and regular cardinal $\lambda$, the following statements are equivalent.
\begin{enumerate}
\item  $V$ is cocomplete and has a strong generator consisting of objects which are connected and $\lambda$-presentable.\label{lc0}
\item  $V$ is locally $\lambda$-c-presentable.\label{lc1}
\item  $V$ is cocomplete and has a small dense subcategory consisting of objects which are connected and $\lambda$-presentable.\label{lc2}
\item  $V$ is a full subcategory of a presheaf category for which the inclusion is $\lambda$-accessible, coproduct preserving and has a left adjoint.\label{lc4}
\item  $V$ is locally $\lambda$-presentable and every object of $V$ is a coproduct of connected objects.\label{lc6}
\item  $V$ is locally $\lambda$-presentable, extensive and the functor $(-){\cdot}1:\Set{\rightarrow}V$ has a left adjoint.\label{lc7}
\end{enumerate}
\end{theorem}
\begin{proof}
The implication (\ref{lc2})$\implies$(\ref{lc0}) is trivial, and the equivalence of (\ref{lc6}) and (\ref{lc7}) is an immediate consequence of lemma(\ref{lem:easy-ext}) and proposition(\ref{prop:lext-decompose-charn}). Given (\ref{lc1}) $V$ is clearly locally $\lambda$-presentable and any $X \in V$ is a coproduct of $\lambda$-filtered colimits of $\lambda$-presentable connected objects. But a $\lambda$-filtered colimit of connected objects is connected, and so (\ref{lc1})$\implies$(\ref{lc6}).

(\ref{lc0})$\implies$(\ref{lc1}): Let $\ca D$ be a strong generator of $\lambda$-small connected objects, and denote also by $\ca D$ the full subcategory of $V$ it determines. Take the closure $\ca S$ of $\ca D$ in $V$ under $\lambda$-small connected colimits, and note that $\ca S$ is also essentially small (see \cite{Kelly-EnrichedCatsBook} section(3.5)). Thus $\ca S$ is also a strong generator of $V$ consisting of $\lambda$-small connected objects and moreover, the full subcategory it determines has $\lambda$-small connected colimits. Thus for $X \in V$, the comma category $\ca S/X$ is $\lambda$-c-filtered, and so it suffices to show that the comma object defining $\ca S/X$ exhibits $X$ as a colimit. Denote by $K$ the actual colimit, for $f:A \to X$ in $\ca S/X$ by $\kappa_f:A \to K$ the component of the colimit cocone, and by $k:K \to X$ the induced map. Since $\ca S$ is a strong generator it suffices to show that for all $A \in \ca S$ the function $V(A,k):V(A,C) \to V(A,X)$ is bijective. It is surjective by the 
definition of $k$, which is defined as the unique map such that $k \kappa_f = f$ for all $f \in \ca S/X$. To see that $V(A,k)$ is injective, suppose that one has $b$ and $c:A \to K$ such that $kb=kc$. Then since the colimit defining $K$ is $\lambda$-c-filtered and $A$ is $\lambda$-presentable and connected, one has $b_2:B \to K$ and $b_3:B \to X$ such that $\kappa_{b_3}b_2=b$, and similarly $c_2:C \to K$ and $c_3:C \to X$ such that $\kappa_{c_3}c_2=c$. Take the pushout
\[ \xygraph{{A}="tl" [r] {B}="tr" [d] {D}="br" [l] {C}="bl" "tl":"tr"^-{b_2}:"br"^-{p}:@{<-}"bl"^-{q}:@{<-}"tl"^-{c_2} "tl":@{}"br"|-*{\scriptstyle{po}}} \]
in $\ca S$, and induce $d:D \to X$ as the unique map such that $dp=b_3$ and $dq=c_3$. The result follows by
\[ b = \kappa_{b_3}b_2 = \kappa_d pb_2 = \kappa_d qc_2 = \kappa_{c_3}c_2 = c. \]

(\ref{lc1})$\implies$(\ref{lc2}): Let $\ca S$ be the set of $\lambda$-presentable connected objects required by definition(\ref{def:lcc}), and denote by $i:\ca S \to V$ the inclusion of the corresponding full subcategory of $V$. Let $X,Y \in V$ and suppose that $\phi:V(i,X) \to V(i,Y)$ in $\PSh{\ca S}$ is given. One has $k:J \rightarrow \ca S$ with $J$ small and $\lambda$-c-filtered, such that $\col(ik)=X$, and we denote by $\kappa_j:kj \rightarrow X$ a typical component of the colimiting cocone. Induce $\phi':X \rightarrow Y$ as the unique map such that $\phi'\kappa_j=\phi(\kappa_j)$ for all $j \in J$. But then for all $f:S \rightarrow X$ with $S \in \ca S$, one has $\phi'f=\phi(f)$: since $S$ is $\lambda$-presentable and connected one can find $j \in J$ and $g:S \rightarrow kj$ such that $f=\kappa_jg$ and so
\[ \phi(f) = \phi(\kappa_j)g = \phi'\kappa_jg=\phi'f. \]
Thus $V(i,1):V \rightarrow \PSh S$ is fully-faithful, in other words, $i$ is dense as claimed.

(\ref{lc2})$\implies$(\ref{lc4}): Let $i:\ca S \to V$ be the inclusion of the given dense subcategory. Then $V(i,1):V \to \PSh{\ca S}$ preserves coproducts and is $\lambda$-accessible since the objects of $\ca S$ are connected and $\lambda$-presentable, is fully faithful since $i$ is dense, and has a left adjoint given by left kan extension along $i$ since $V$ is cocomplete and $\ca S$ is small.

(\ref{lc4})$\implies$(\ref{lc2}): Let $I:V \to \PSh{\C}$ be the given inclusion and $L$ be its left adjoint. Let $T$ be the monad induced by $L \ladj I$, and note that since $I$ is fully-faithful, it is monadic. Denote by $i:\ca S \to \PSh{\C}$ the inclusion of the closure of the representables in $\PSh{\C}$ under $\lambda$-small connected colimits. Since the objects of $\ca S$ are $\lambda$-presentable and connected, and $T$ is $\lambda$-accessible and coproduct preserving, it follows that $(T,\ca S)$ is a monad with arities in the sense of \cite{Weber-Fam2fun,Mellies-Segal,BergMellWeber-MonadsArities}. Taking
\[ \xygraph{{\ca S}="l" [r] {\Theta_T}="m" [r] {V}="r" "l":"m"^-{k}:"r"^-{j}} \]
to be the identity on objects fully faithful factorisation of $Li$, it follows from the Nerve theorem \cite{BergMellWeber-MonadsArities} that $j:\Theta_T \to V$ is dense. Since for all $A \in \Theta_T$, $V(jA,-) \iso V(LiA,-) \iso \PSh{\C}(iA,IA)$ and $I$ preserves coproducts and $\lambda$-filtered colimits, it follows that the image of $j$ consists of connected $\lambda$-presentable objects.

(\ref{lc7})$\implies$(\ref{lc0}): $V$ is cocomplete by definition. Let $\ca D$ be a strong generator of $\lambda$-presentable objects of $V$. Decompose each object of $\ca D$ into connected components using proposition(\ref{prop:lext-decompose-charn}), and write $\ca D'$ for the set of summands of objects of $\ca D$ that so arise. Clearly $\ca D'$ is also a strong generator of $V$, its objects are connected by definition and $\lambda$-presentable by lemma(\ref{lem:summands-of-lambda-presentables}).
\end{proof}
\begin{examples}\label{ex:lc-groth-toposes}
By theorem(\ref{thm:conn-GabUlm})(\ref{lc4}) any presheaf topos is locally finitely c-presentable. By theorem(\ref{thm:conn-GabUlm})(\ref{lc7}) a Grothendieck topos is locally connected iff it is locally c-presentable.
\end{examples}
Just as with locally presentable categories, locally c-presentable categories are closed under many basic categorical constructions. For instance from theorem(\ref{thm:conn-GabUlm})(\ref{lc6}), one sees immediately that the slices of a locally $\lambda$-c-presentable category are locally $\lambda$-c-presentable from the corresponding result for locally presentable categories. Another instance of this principle is the following result.
\begin{theorem}\label{thm:acc-monad}
If $V$ is locally $\lambda$-c-presentable and $T$ is a $\lambda$-accessible coproduct preserving monad on $V$, then $V^T$ is locally $\lambda$-c-presentable.
\end{theorem}
\begin{proof}
By the analogous result for locally presentable categories $V^T$ is locally $\lambda$-presentable and thus cocomplete. Defining $\Theta_0$ to be the full subcategory of $V$ consisting of the $\lambda$-presentable and connected objects, $(T,\Theta_0)$ is a monad with arities in the sense of \cite{Weber-Fam2fun}. One has a canonical isomorphism
\[ \xymatrix @C=3.5em {{V^T} \ar[r]^-{V^T(i,1)} \ar[d]_{U^T} \save \POS?="d" \restore & {\PSh {\Theta}_T} \ar[d]^{\res_j} \save \POS?="c" \restore \\ V \ar[r]_-{V(i_0,1)} & {\PSh {\Theta}_0} \POS "d";"c" **@{}; ?*{\iso}} \]
in the notation of \cite{Weber-Fam2fun}. Since $T$ preserves $\lambda$-filtered colimits and coproducts, $U^T$ creates them. Since $j$ is bijective on objects $\res_j$ creates all colimits. Thus by the above isomorphism $V^T(i,1)$ preserves $\lambda$-filtered colimits and coproducts. By the nerve theorem of \cite{Weber-Fam2fun} $V^T(i,1)$ is also fully faithful, it has a left adjoint since $V^T$ is cocomplete given by left extending $i$ along the yoneda embedding, and so we have exhibited $V^T$ as conforming to theorem(\ref{thm:conn-GabUlm})(\ref{lc4}).
\end{proof}
\begin{examples}\label{ex:lc-noperad-algebras}
An $n$-operad for $0{\leq}n{\leq}\omega$ in the sense of \cite{Batanin-MonGlobCats}, gives a finitary coproduct preserving monad on the category $\PSh {\G}_{{\leq}n}$ of $n$-globular sets, and its algebras are just the algebras of the monad. Since $\PSh {\G}_{{\leq}n}$ as a presheaf topos is locally finitely c-presentable by example(\ref{ex:lc-groth-toposes}), the category of algebras of any $n$-operad is locally finitely c-presentable by theorem(\ref{thm:acc-monad}).
\end{examples}

\section*{Acknowledgements}
\label{sec:Acknowledgements}
Many thanks to Michael Batanin, Clemens Berger, Denis-Charles Cisinski and Paul-Andr\'{e} Melli\`{e}s for interesting discussions on the substance of this paper. Thanks are also due to the referee for various insightful remarks which helped a lot to improve the exposition. Finally, I would like to express my gratitude to the laboratory PPS (Preuves Programmes Syst\`{e}mes) in Paris, the Max Planck Institute in Bonn, the IHES and the Macquarie University Mathematics Department for the excellent working conditions I enjoyed during this project.

\bibliographystyle{plain}

\end{document}